\documentclass[11pt]{article}

\usepackage{microtype,bm}
\usepackage{graphicx}
\usepackage{subfigure}
\usepackage{booktabs} 
\usepackage{natbib}
\usepackage{epsfig,epsf,fancybox}
\usepackage{amsmath}
\usepackage{mathrsfs}
\usepackage{amssymb}
\usepackage{color}
\usepackage{multirow}
\usepackage{paralist}
\usepackage{verbatim}
\usepackage{algorithm}
\usepackage{algorithmic}
\usepackage{galois}
\usepackage{boxedminipage}
\usepackage{accents}
\usepackage{stmaryrd}
\usepackage[bottom]{footmisc}
\usepackage{hhline}

\usepackage{natbib}

\usepackage{url}
\usepackage[colorlinks,linkcolor=magenta,citecolor=blue, pagebackref=true,backref=true]{hyperref}
\renewcommand*{\backrefalt}[4]{%
    \ifcase #1 \footnotesize{(Not cited.)}%
    \or        \footnotesize{(Cited on page~#2.)}%
    \else      \footnotesize{(Cited on pages~#2.)}%
    \fi}

\newtheorem{theorem}{Theorem}[section]
\newtheorem{corollary}[theorem]{Corollary}
\newtheorem{lemma}[theorem]{Lemma}
\newtheorem{proposition}[theorem]{Proposition}

\newtheorem{definition}{Definition}[section]

\newtheorem{remark}[theorem]{Remark}
\newtheorem{assumption}[theorem]{Assumption}

\usepackage{hyperref}

\usepackage[framemethod=default]{mdframed} 

\def\beq{\begin{equation} }
\def\eeq{\end{equation} }

\def\thetaholder{\boldsymbol{\theta}}
\def\foneholder{\boldsymbol{\varphi}_1}
\def\ftwoholder{\boldsymbol{\varphi}_2}
\def\boneholder{\boldsymbol{\delta}_1}
\def\btwoholder{\boldsymbol{\delta}_2}
\def\arbholder{\mathbf{z}}

\newcommand{\zag}{\zholder^{\text{ag}}}
\newcommand{\zmd}{\zholder^{\text{md}}}
\newcommand{\xholder}{\boldsymbol{x}}
\newcommand{\yholder}{\boldsymbol{y}}
\newcommand{\zholder}{\boldsymbol{z}}
\newcommand{\cX}{\mathcal{X}}
\newcommand{\cY}{\mathcal{Y}}
\newcommand{\cL}{\mathcal{L}}
\newcommand{\RR}{\mathbb{R}}
\newcommand{\btheta}{\bm{\theta}}
\newcommand{\norm}[1]{||#1||}

\def\btheta{\xholder}
\def\bpsi{\yholder}
\def\pnlty{\rho}
\def\rds{\mathcal{R}}

\newcommand{\AlgoName}{AGOG-Avatar }

\newcommand{\Ab}{\mathbf{A}}

\newcommand{\bz}{\bm{z}}
\newcommand{\zstar}{\zholder^*}
\newcommand{\zhp}{\bz_{k+\frac12}}
\newcommand{\zhm}{\bz_{k-\frac12}}
\newcommand{\zp}{\bz_{k+1}}
\newcommand{\ph}{k+\frac12}
\newcommand{\mh}{k-\frac12}

\def\Matrix{H}
\newcommand{\EE}{\mathbb{E}}
\usepackage{cases}
\newcommand{\cO}{\mathcal{O}}
\newcommand{\bphi}{\bm{\phi}}
\newcommand{\cK}{\mathcal{K}}
\newcommand{\bu}{\bm{u}}
\newcommand{\cZ}{\mathcal{Z}}
\newcommand{\bx}{\bm{x}}
\newcommand{\by}{\bm{y}}
\newcommand{\Iholder}{I}
\newcommand{\op}{o_{\raisemath{-1.5pt}\PP}}
\renewcommand{\op}{\text{op}}

\usepackage{diagbox}
\usepackage[dvipsnames]{xcolor}
\definecolor{VioletRed}{rgb}{0.915686,0.025490,0.364706}
\usepackage{colortbl}
\definecolor{LightCyan}{rgb}{0.8, 0.9, 1}
\newcommand{\xmark}{\ding{56}}
\usepackage{pifont}
\def\normop#1{\left\Vert #1\right\Vert }%
\usepackage{mathtools}

\newtheorem{innercustom}{}
\newenvironment{custom}[1]
  {\renewcommand\theinnercustom{#1}\innercustom}
  {\endinnercustom}
\usepackage{enumitem}
\def\red#1{}\def\pb{}
\def\blue#1{\textcolor{black}{#1}}

\begin{document}


\begin{center}

{\bf{\LARGE{Nesterov Meets Optimism: Rate-Optimal Separable Minimax Optimization}}}

\vspace*{.2in}
{\large{
\begin{tabular}{ccccc}
Chris Junchi Li$^{\diamond*}$
&
Angela Yuan$^{\dagger*}$ 
&
Gauthier Gidel$^\ddagger$
&
Quanquan Gu$^\dagger$
&
Michael I.~Jordan$^{\diamond, \S}$
\\
\end{tabular}
}}

\vspace*{.2in}

\begin{tabular}{c}
Department of Electrical Engineering and Computer Sciences, University of California, Berkeley$^\diamond$
\\
Department of Computer Sciences, University of California, Los Angeles$^\dagger$
\\
DIRO, Université de Montréal and Mila$^\ddagger$
\\
Department of Statistics, University of California, Berkeley$^\S$
\end{tabular}

\vspace*{.2in}

\today

\vspace*{.2in}

\let\thefootnote\relax\footnotetext{$^*$Equal contribution.}

\begin{abstract} 
We propose a new first-order optimization algorithm --- AcceleratedGradient-OptimisticGradient (AG-OG) Descent Ascent---for separable convex-concave minimax optimization. The main idea of our algorithm is to carefully leverage the structure of the minimax problem, performing Nesterov acceleration on the individual component and optimistic gradient on the coupling component. Equipped with proper restarting, we show that AG-OG achieves the optimal convergence rate (up to a constant) for a variety of settings, including bilinearly coupled strongly convex-strongly concave minimax optimization (bi-SC-SC), bilinearly coupled convex-strongly concave minimax optimization (bi-C-SC), and bilinear games. We also extend our algorithm to the stochastic setting and achieve the optimal convergence rate in both bi-SC-SC and bi-C-SC settings. AG-OG is the first single-call algorithm with optimal convergence rates in both deterministic and stochastic settings for bilinearly coupled minimax optimization problems.
\end{abstract}
\end{center}

\pb\section{Introduction}\label{sec:intro}
Optimization is the workhorse for machine learning (ML) and artificial intelligence. While many ML learning tasks can be cast as a minimization problem, there is an increasing number of ML tasks, such as generative adversarial networks (GANs)~\citep{goodfellow2020generative}, robust/adversarial training~\citep{bai2020provable, madry2017towards}, Markov games (MGs)~\citep{shapley1953stochastic}, and reinforcement learning (RL)~\citep{sutton2018reinforcement, du2017stochastic, dai2018sbeed}, that are instead formulated as a minimax optimization problem in the following form:
\begin{align}
\label{eq:Minmax_opt}
\min_{\xholder\in \cX} \max_{\yholder\in \cY}~%
\cL(\xholder,\yholder)
.
\end{align}
When $\cL(\xholder,\yholder): \cX\times\cY \rightarrow \RR$ is a smooth function that is convex in $\bx$ and concave in $\by$, we refer to this problem as a \emph{convex-concave minimax problem} (a.k.a., convex-concave saddle point problem).
In this work, we focus on designing fast or even optimal deterministic and stochastic first-order algorithms for solving convex-concave minimax problems of the form~\eqref{eq:Minmax_opt}.

Unlike in the convex minimization setting, where gradient descent is the method of choice, the gradient descent-ascent method can exhibit divergence on convex-concave objectives.
Indeed, examples show the divergence of gradient descent ascent (GDA) on bilinear objectives~\citep{liang2019interaction,gidel2018variational}.
This has led to the development of extrapolation-based methods, including the extragradient (EG) method~\citep{korpelevich1976extragradient} and the optimistic gradient descent ascent (OGDA) method~\citep{popov1980modification}, both of which can be shown to converge in the convex-concave setting.
While the EG algorithm needs to call the gradient oracle twice at each iteration, the OGDA algorithm only needs a single call to the gradient oracle~\citep{gidel2018variational, hsieh2019convergence} and therefore has a practical advantage when the gradient evaluation is expensive. We build on this line of research, aiming to attain improved, and even optimal, convergence rates via algorithms that retain the spirit of simplicity of OGDA.

We focus on a specific instance of the general minimax optimization problem, namely the separable minimax optimization problem, which is formulated as follows
\begin{align}\label{eq:Minmax_separable}
\min_{\bx \in \cX} \max_{\by \in \cY}~
\cL(\xholder, \yholder)
=
f(\xholder)
+
\Iholder(\xholder, \yholder)
-
g(\yholder)
.
\end{align}
We refer to $f(\xholder) - g(\yholder)$ as the \emph{individual component}, and $\Iholder(\xholder, \yholder)$ as the \emph{coupling component} of Problem~\eqref{eq:Minmax_separable}.
Let $f$ be $\mu_f$-strongly convex and $L_f$-smooth and $g$ be $\mu_g$-strongly convex and $L_g$-smooth.
Let $\Iholder(\xholder,\yholder)$ be convex-concave with blockwise smoothness parameters $I_{xx}, I_{xy}, I_{yy}$ where $\norm{\nabla_{xx}^2 I}_{\op} \leq I_{xx}$, $\norm{\nabla_{xy}^2 I}_{\op} \leq I_{xy}$, and $\norm{\nabla_{yy}^2 I}_{\op} \leq I_{yy}$.
Let $\Iholder(\xholder, \yholder)$ be $L_H$-smooth, and it is straightforward to observe that $L_H$ can be picked as small as $I_{xx}\lor I_{yy} + I_{xy}$.
Throughout this paper, we focus on the unconstrained problem where $\cX = \RR^n$ and $\cY = \RR^m$ unless otherwise specified in certain applications.

A notable special case of the separable minimax Problem~\eqref{eq:Minmax_separable} is the so-called \emph{bilinearly coupled strongly convex-strongly concave minimax problem} (bi-SC-SC), which has the following form:
\begin{align}
\label{eq:Minmax_bilinear}
\min_{\xholder \in \cX}
\max_{\yholder \in \cY}~
\cL(\xholder, \yholder)
\equiv
f(\xholder)
+
\xholder^\top \mathbf{B} \yholder
-
g(\yholder)
.
\end{align}
Here we take $\Iholder(\xholder, \yholder)$ as the bilinear coupling function $\xholder^\top \mathbf{B} \yholder$ and is $L_H$-smooth where $L_H$ can be picked as small as the operator norm $\|\mathbf{B}\|_{\op}$ of matrix $\mathbf{B}$.

For the general minimax optimization problem~\eqref{eq:Minmax_opt}, standard algorithms such as mirror-prox \citep{nemirovski2004prox}, EG and OGDA---when operating on the entire objective---can be shown to exhibit a complexity upper bound of $\frac{\bar{L}}{\bar{\mu}}\log\left(\frac{1}{\epsilon}\right)$ for finding an $\epsilon$-accurate solution~\citep{gidel2018variational,mokhtari2020unified}, where $\bar{L}\equiv L_f\vee L_g\vee L_H$ and $\bar{\mu}\equiv \mu_f\wedge \mu_g$.
Such a complexity is \emph{optimal} when $L_f = L_g = L_H$ and $\mu_f = \mu_g$,  since the lower-bound complexity is $\Omega(\tfrac{\bar{L}}{\bar{\mu}} \log(\tfrac{1}{\epsilon}))$~\citet{nemirovskij1983problem,azizian2020accelerating}.
However, in the general case where the strong convexity and smoothness parameters are significantly different in $\bx$ and $\by$, fine-grained convergence rates that depend on the individual strong convexity $\mu_f,\mu_g$ and smoothness parameters $L_f, L_g$ and also $I_{xx}, I_{xy}, I_{yy}$ are more desirable.
In fact, \citet{zhang2021lower} have proved the following iteration complexity lower bound for solving \eqref{eq:Minmax_separable} via any first-order algorithms under the linear span assumption:
\begin{align}\label{eq:lower bound}
\tilde{\Omega}\bigg(
\sqrt{\frac{L_f+I_{xx}}{\mu_f} + \frac{L_g+I_{yy}}{\mu_g} + \frac{I_{xy}^2}{\mu_f\mu_g}}
\bigg)
.
\end{align}

With the goal of attaining this lower bound, several efforts have been made in the setting of bi-SC-SC \eqref{eq:Minmax_bilinear} or separable SC-SC \eqref{eq:Minmax_separable}. Two notable methods are LPD~\citep{thekumparampil2022lifted} and PD-EG~\citep{jin2022sharper}, which utilize techniques from primal-dual lifting and convex conjugate decomposition. Another approach is the APDG algorithm developed by~\citet{kovalev2021accelerated}, which is based on adding an extrapolation step to the forward-backward algorithm. The work of~\citet{du2022optimal} is also closely related to our work, in the sense that it uses iterate averaging and employs scaling reduction with scheduled restarting.
However, these algorithms are either limited to the bi-SC-SC setting~\citep{kovalev2021accelerated, thekumparampil2022lifted, du2022optimal}, or are not single-call algorithms~\citep{kovalev2021accelerated, jin2022sharper, du2022optimal}.\footnote{By single call, we mean the algorithm only needs to call the (stochastic) gradient oracle of the coupling component once in each iteration of the algorithm. This is in accordance with the concept of single-call variants of extragradient in~\citet{hsieh2019convergence}. Previous work  calls $\nabla I(\xholder, \yholder)$ at least twice per iteration.} In addition, only~\citet{kovalev2021accelerated} and~\citet{du2022optimal} can be extended to the stochastic setting \citep{metelev2022decentralized}, while the extension of~\citet{thekumparampil2022lifted,jin2022sharper} to the stochastic setting remains elusive.

In this paper, we design near-optimal single-call algorithms for both deterministic and stochastic separable minimax problems~\eqref{eq:Minmax_separable}.
We focus on accelerating OGDA because of its simplicity and because it enjoys the single-call property. We show that it achieves a \emph{fine-grained}, \emph{accelerated} convergence rate with a sharp dependency on the individual Lipschitz constants.
To the best of our knowledge, this is the first presentation of a single-call algorithm that matches the best-known result for the separable minimax problem~\eqref{eq:Minmax_separable} and the lower bounds under a bi-SC-SC setting~\eqref{eq:Minmax_bilinear}, bilinearly coupled convex-strongly concave (bi-C-SC) setting (i.e., $f$ is convex but not strongly convex in \eqref{eq:Minmax_bilinear}), and the bilinear game setting (i.e., setting $f=g = 0$ in \eqref{eq:Minmax_bilinear}).

\subsection{Contributions}
We highlight our contributions as follows.
\begin{enumerate}[label=(\roman*)]
\setlength{\itemsep}{0pt}%
\item
We present a novel algorithm that blends acceleration dynamics based on the single-call OGDA algorithm for the coupling component and Nesterov's acceleration for the individual component. We refer to this new algorithm as the \emph{AcceleratedGradient-OptimisticGradient (AG-OG)} Descent Ascent algorithm.
\blue{Using a scheduled restarting, we derive an \emph{AcceleratedGradient OptimisticGradient with restarting} (AG-OG with restarting) algorithm that achieves a sharp convergence rate in a variety of settings.}
\blue{We provide theoretical analysis of our algorithm for general separable SC-SC problem~\eqref{eq:Minmax_separable} and compare the results with existing literature under special cases in the form of~\eqref{eq:Minmax_bilinear} (bi-SC-SC, bi-C-SC and Bilinear).}

\item
Using a scheduled restarting, we derive an \emph{AcceleratedGradient-OptimisticGradient with restarting} (AG-OG with restarting) algorithm that achieves a sharp convergence rate in a variety of settings.
For general separable SC-SC setting in~\eqref{eq:Minmax_separable}, our algorithm achieves a complexity of $\left(\sqrt{\frac{L_f}{\mu_f}\vee \frac{L_g}{\mu_g}} + \frac{I_{xx}}{\mu_f}\vee \frac{I_{xy}}{\sqrt{\mu_f\mu_g}} \vee \frac{I_{yy}}{\mu_g}\right)
\log\left(\frac{1}{\epsilon}\right)$, matching the best known upper bound in~\citet{jin2022sharper}.
For the setting of bilinearly coupled SC-SC in \eqref{eq:Minmax_bilinear}, our algorithm achieves a complexity of $
\cO\Big(\sqrt{\frac{L_f}{\mu_f}\lor\frac{L_g}{\mu_g}} + \sqrt{\frac{\normop{\mathbf{B}}}{\mu_f\mu_g}}\Big)\log\left(\frac{1}{\epsilon}\right)
$ [Corollary~\ref{corr:determ_complexity}], which matches the lower bound established by~\citet{zhang2021lower}.
For bi-C-SC, we prove a
$
\cO\bigg(\sqrt{\frac{L_f}{\epsilon} \vee \frac{L_g}{\mu_g}} + \frac{\|\mathbf{B}\|_{\op}}{\sqrt{\epsilon\mu_g}}\bigg)\log\left(\frac{1}{\epsilon}\right)
$
complexity [Theorem~\ref{theo:C-SC}], 
which matches that of \citet{thekumparampil2022lifted} and is also optimal.

\item
In the stochastic setting where the algorithm can only query a stochastic gradient oracle with bounded noise, we propose a stochastic extension of AG-OG with restarting and establish a sharp convergence rate.
For both bi-SC-SC and bi-C-SC settings, the convergence rate of our algorithm is near-optimal in the sense that its bias error matches the respective deterministic lower bound and its variance error matches the statistical minimax rate, i.e., $\frac{\sigma^2}{\mu_f^2\epsilon^2}$ [Corollary~\ref{corr:stocha_complexity}].

\item
In the special case of the bilinear game (when $f = g = 0$ in \eqref{eq:Minmax_bilinear}), our algorithm has a complexity of $\Omega\left(
\frac{\|\mathbf{B}\|_{\op}}{\sqrt{\lambda_{\min}(\mathbf{B}^\top \mathbf{B})}}
\right)\log\left( \frac{1}{\epsilon}\right)$ [Theorem~\ref{theo:main_bilinear}], which matches the lower bound established by~\citet{ibrahim2020linear}. 
Note that prior work \citep{kovalev2021accelerated,thekumparampil2022lifted,jin2022sharper} cannot achieve the optimal rate when applied to bilinear games, which is an unique advantage of our algorithm. 
\end{enumerate}

A summary of the iteration complexity comparisons with the state-of-the-art methods can be found in Table~\ref{table:11}.

\begin{table*}[!tb]
\centering
\resizebox{\textwidth}{!}{%
\begin{tabular}{cccccccc}
\toprule
\diagbox{Method}{Setting}
& SC-SC  &  bi-SC-SC     & Bilinear   & bi-C-SC & \begin{tabular}{@{}c@{}}Stochastic\\ Rate\end{tabular} & \begin{tabular}{@{}c@{}}Single \\ Call\end{tabular}
\\
\midrule
OGDA & & & &
\\
\citep{mokhtari2020convergence}
& \multirow{-2}{*}{$\tilde{\cO}\bigg(\frac{L_f'\vee L_g'\vee I_{xy}}{\mu_f \wedge \mu_g}\bigg)$}
&\multirow{-2}{*}{$\tilde{\cO}\bigg(\frac{L_f\lor L_g\lor \norm{\mathbf{B}}}{\mu_f \wedge \mu_g}\bigg)$}
& \multirow{-2}{*}{$\tilde{\cO}\bigg(\frac{\norm{\mathbf{B}}^2}{\lambda_{\min}}\bigg)$}
& \multirow{-2}{*}{$\cO\bigg(\frac{L_f\lor L_g\lor \norm{\mathbf{B}}}{\epsilon}\bigg)$}
&\multirow{-2}{*}{$\pmb{\color{OliveGreen} \checkmark}$}
&\multirow{-2}{*}{$\pmb{\color{OliveGreen} \checkmark}$}
\\
\midrule
Proximal Best Response
\\
\citep{wang2020improved}
& \multirow{-2}{*}{$\tilde{\cO}\bigg(
\sqrt{\frac{L_f'}{\mu_f} \vee \frac{L_g'}{\mu_g}}
+
\sqrt{\frac{I_{xy} (L_f'\lor L_g'\lor I_{xy})}{\mu_f\mu_g}}
\bigg)$}
&\multirow{-2}{*}{$\tilde{\cO}\bigg(
\sqrt{\frac{L_f}{\mu_f} \vee \frac{L_g}{\mu_g}}
+
\sqrt{\frac{\norm{\mathbf{B}}(L_f\lor L_g\lor \norm{\mathbf{B}})}{\mu_f\mu_g}}
\bigg)$}
& \multirow{-2}{*}{---}
& \multirow{-2}{*}{---}
&\multirow{-2}{*}{{\color{red} {\color{red} \xmark}}}
& \multirow{-2}{*}{{\color{red} {\color{red} \xmark}}}\\
\midrule 
DIPPA & & & &\\
\citep{xie2021dippa}
& \multirow{-2}{*}{---}
&\multirow{-2}{*}{$\tilde{\cO}\bigg(\bigg(\frac{L_fL_g}{\mu_f\mu_g}\bigg( \frac{L_f}{\mu_f} \vee \frac{L_g}{\mu_g}\bigg)\bigg)^{\frac14} + \frac{\norm{\mathbf{B}}}{\sqrt{\mu_f\mu_g}}\bigg)$}
& \multirow{-2}{*}{---}
& \multirow{-2}{*}{---} &\multirow{-2}{*}{{\color{red} {\color{red} \xmark}}}
& \multirow{-2}{*}{{\color{red} {\color{red} \xmark}}}
\\
\midrule 
LPD & & & &\\
\citep{thekumparampil2022lifted}
& \multirow{-2}{*}{---} &\multirow{-2}{*}{$\tilde{\cO}\bigg(\sqrt{\frac{L_f}{\mu_f} \vee \frac{L_g}{\mu_g}} + \frac{\norm{\mathbf{B}}}{\sqrt{\mu_f\mu_g}}\bigg)$}
& \multirow{-2}{*}{$\tilde{\cO}\bigg(\frac{\norm{\mathbf{B}}^2}{\lambda_{\min}}\bigg)$}
& \multirow{-2}{*}{$\tilde{\cO}\bigg(\sqrt{\frac{L_f}{\epsilon} \vee \frac{L_g}{\mu_g}} + \frac{\|\mathbf{B}\|}{\sqrt{\epsilon\mu_g}}\bigg)$}
&\multirow{-2}{*}{{\color{red} \xmark}}
& \multirow{-2}{*}{$\pmb{\color{OliveGreen} \checkmark}$}\\
\midrule 
APDG & & & &\\
\begin{tabular}{@{}c@{}}\citep{kovalev2021accelerated} \\ \citep{metelev2022decentralized}\end{tabular}
& \multirow{-2}{*}{---} &\multirow{-2}{*}{$\tilde{\cO}\bigg(\sqrt{\frac{L_f}{\mu_f} \vee \frac{L_g}{\mu_g}} + \frac{\norm{\mathbf{B}}}{\sqrt{\mu_f\mu_g}}\bigg)$}
& \multirow{-2}{*}{$\tilde{\cO}\bigg(\frac{\norm{\mathbf{B}}^2}{\lambda_{\min}}\bigg)$}
& \multirow{-2}{*}{$\tilde{\cO}\bigg(
\sqrt{\frac{L_fL_g}{\lambda_{\min}}}
\vee
\frac{\norm{\mathbf{B}}}{\sqrt{\lambda_{\min}}}\sqrt{\frac{L_g}{\mu_g}}
\vee
\frac{\norm{\mathbf{B}}^2}{\lambda_{\min}}\bigg)$}
&\multirow{-2}{*}{$\pmb{\color{OliveGreen} \checkmark}$}
& \multirow{-2}{*}{{\color{red} \xmark}}\\
\midrule 
PD-EG & & & &\\
\citep{jin2022sharper}
& \multirow{-2}{*}{$\tilde{\cO}\bigg(\sqrt{\frac{L_f}{\mu_f} \vee \frac{L_g}{\mu_g}} + \frac{I_{xx}}{\mu_f} \vee \frac{I_{xy}}{\sqrt{\mu_f\mu_g}} \vee \frac{I_{yy}}{\mu_g}\bigg)$} 
&\multirow{-2}{*}{$\tilde{\cO}\bigg(\sqrt{\frac{L_f}{\mu_f} \vee \frac{L_g}{\mu_g}} + \frac{\norm{\mathbf{B}}}{\sqrt{\mu_f\mu_g}}\bigg)$}
& \multirow{-2}{*}{$\tilde{\cO}\bigg(\frac{\norm{\mathbf{B}}^2}{\lambda_{\min}}\bigg)$}
& \multirow{-2}{*}{---}&\multirow{-2}{*}{{\color{red} \xmark}}
& \multirow{-2}{*}{{\color{red} \xmark}}\\
\midrule 
EG+Momentum & & & &\\
\citep{azizian2020accelerating} 
& \multirow{-2}{*}{---} 
&\multirow{-2}{*}{---}
& \multirow{-2}{*}{$\tilde{\cO}\bigg(\frac{\norm{\mathbf{B}}}{\sqrt{\lambda_{\min}}}\bigg)$}
& \multirow{-2}{*}{---}&\multirow{-2}{*}{{\color{red} \xmark}}
& \multirow{-2}{*}{{\color{red} \xmark}}\\
\midrule 
SEG with Restarting& & & &\\
\citep{li2022convergence}
& \multirow{-2}{*}{---} 
&\multirow{-2}{*}{---}
& \multirow{-2}{*}{$\tilde{\cO}\bigg(\frac{\norm{\mathbf{B}}}{\sqrt{\lambda_{\min}}}\bigg)$}
& \multirow{-2}{*}{---}&\multirow{-2}{*}{$\pmb{\color{OliveGreen} \checkmark}$}
& \multirow{-2}{*}{{\color{red} \xmark}}\\
\midrule 
AG-EG with Restarting
\\\citep{du2022optimal}
& \multirow{-2}{*}{---} 
&\multirow{-2}{*}{$\tilde{\cO}\bigg(\sqrt{\frac{L_f}{\mu_f} \vee \frac{L_g}{\mu_g}} + \frac{\norm{\mathbf{B}}}{\sqrt{\mu_f\mu_g}}\bigg)$}
& \multirow{-2}{*}{$\tilde{\cO}\bigg(\frac{\norm{\mathbf{B}}}{\sqrt{\lambda_{\min}}}\bigg)$}
& \multirow{-2}{*}{---}
&  \multirow{-2}{*}{$\pmb{\color{OliveGreen} \checkmark}$}
& \multirow{-2}{*}{{\color{red} \xmark}}
\\
\midrule 
\rowcolor{LightCyan}
AG-OG with Restarting & & & & & &
\\
\rowcolor{LightCyan}
(this work)
& \multirow{-2}{*}{$\tilde{\cO}\bigg(\sqrt{\frac{L_f}{\mu_f} \vee \frac{L_g}{\mu_g}} + \frac{I_{xx}}{\mu_f} \vee \frac{I_{xy}}{\sqrt{\mu_f\mu_g}} \vee \frac{I_{yy}}{\mu_g}\bigg)$}
&\multirow{-2}{*}{$\tilde{\cO}\bigg(\sqrt{\frac{L_f}{\mu_f} \vee \frac{L_g}{\mu_g}} + \frac{\norm{\mathbf{B}}}{\sqrt{\mu_f\mu_g}}\bigg)$}
& \multirow{-2}{*}{$\tilde{\cO}\bigg(\frac{\norm{\mathbf{B}}}{\sqrt{\lambda_{\min}}}\bigg)$}
& \multirow{-2}{*}{$\tilde{\cO}\bigg(\sqrt{\frac{L_f}{\epsilon} \vee \frac{L_g}{\mu_g}} + \frac{\|\mathbf{B}\|}{\sqrt{\epsilon\mu_g}}\bigg)$}
& \multirow{-2}{*}{$\pmb{\color{OliveGreen} \checkmark}$}
& \multirow{-2}{*}{$\pmb{\color{OliveGreen} \checkmark}$}
\\
\midrule 
Lower Bound & & & &\\
\begin{tabular}{@{}c@{}}\citep{zhang2021lower} \\ \citep{ibrahim2020linear}\end{tabular}
& \multirow{-2}{*}{$\tilde{\Omega}\bigg(\sqrt{\frac{L_f'}{\mu_f} \vee \frac{L_g'}{\mu_g}} + \frac{I_{xy}}{\sqrt{\mu_f\mu_g}}\bigg)$}
&\multirow{-2}{*}{$\tilde{\Omega}\bigg(\sqrt{\frac{L_f}{\mu_f} \vee \frac{L_g}{\mu_g}} + \frac{\norm{\mathbf{B}}}{\sqrt{\mu_f\mu_g}}\bigg)$}
& \multirow{-2}{*}{$\tilde{\Omega}\bigg(\frac{\norm{\mathbf{B}}}{\sqrt{\lambda_{\min}}}\bigg)$}
& \multirow{-2}{*}{\blue{$\tilde{\cO}\bigg(\sqrt{\frac{L_f}{\epsilon} \vee \frac{L_g}{\mu_g}} + \frac{\|\mathbf{B}\|}{\sqrt{\epsilon\mu_g}}\bigg)$}}
& \multirow{-2}{*}{---}
& \multirow{-2}{*}{---}
\\
\bottomrule\end{tabular}
}
\caption{%
We present a comparison of the first-order gradient complexities of our proposed algorithm with selected prevailing algorithms for solving bilinearly-coupled minimax problems.
The comparison includes several cases such as general SC-SC, bilinear games, bi-SC-SC (bilinearly-coupled SC-SC), and the bi-C-SC cases.
We denote $\lambda_{\min}\equiv \lambda_{\min}(\mathbf{B}^\top \mathbf{B})$, $L_f'\equiv L_f+I_{xx}$ and $L_g'\equiv L_g+I_{yy}$.
We focus on comparing the gradient complexities of deterministic algorithms, and include a column to indicate whether the stochastic case has been discussed.
The row in blue background is the convergence result presented in this paper.
The "---" indicates that the complexity does not apply to the given case.
}\label{table:11}
\end{table*}

\pb\subsection{More Related Work}\label{sec:related_work}
\paragraph{Deterministic Case.}
Much attention has been paid to obtaining linear convergence rates for gradient-based methods applied to games in the context of strongly monotone operators (which is implied by strong convex-concavity)~\citep{mokhtari2020unified} and several recent works~\citep{yang2020catalyst,zhang2021complexity,cohen2020relative,wang2020improved,xie2021dippa} have bridged the gap with the lower bound provided for \emph{unbalanced} strongly-convex-strongly-concave objective. 
There has been a series of papers along this direction~\citep{mokhtari2020unified, cohen2020relative, lin2020near, wang2020improved, xie2021dippa}, and only very recently have optimal results that reach the lower bound been presented~\citep{kovalev2021accelerated,thekumparampil2022lifted,jin2022sharper}.
This work presented improved methods leveraging convex duality.
Among these works, only~\citet{jin2022sharper} considers non-bilinear coupling terms, and only~\citet{thekumparampil2022lifted} considers single gradient calls.
Note that~\citet{jin2022sharper} consider a finite-sum case, which differs from our setting of a general expectation.
\citet{kovalev2021accelerated, thekumparampil2022lifted} focus solely on the deterministic setting, and \citet{metelev2022decentralized} present a stochastic version of APDG algorithm~\citep{kovalev2021accelerated} and its extension to a decentralized setting, which is comparable and concurrent with the work of \citet{du2022optimal}.

\paragraph{Stochastic Case.}
There exists a rich literature on stochastic variational inequalities with application to solving stochastic minimax problems~\citep{juditsky2011solving,hsieh2019convergence,chavdarova2019reducing,alacaoglu2022stochastic,zhao2022accelerated,beznosikov2022smooth}. However, only a few works have proposed fine-grained bounds suited to the (bi-)SC-SC setting. To the best of our knowledge, most fine-grained bounds have been proposed in the finite-sum setting~\citep{palaniappan2016stochastic,jin2022sharper} or in the proximal-friendly case~\citep{zhang2021robust}. Two closely related works are \citet{li2022convergence}, who provide a convergence rate for stochastic extragradient method in the purely bilinear setting and~\citet{du2022optimal}, who study an accelerated version of extragradient, dubbed as AcceleratedGradient-ExtraGradient (AG-EG) in the bi-SC-SC setting. 
Our work is in the same vein as~\citet{du2022optimal} but instead employs the optimistic gradient instead of extragradient to handle the bilinear coupling component.
Optimistic-gradient-based methods have been considered extensively in the literature due to their need for fewer gradient oracle calls per iteration than standard extragradient and their applicability to the online learning setting~\citep{golowich2020tight}. 
Note that, in general, EG and OG methods  share some similarities in their analyses, but there are also significant differences~\citep[\S 3.1]{golowich2020tight},~\citep[\S2]{gorbunov2022last}.
Specifically in our case, using a \emph{single-call algorithm} that reuses previously calculated gradients alters a key recursion (Eq.~\eqref{eq:bound_diff}).
Although the main part of the proof follows the standard path of estimating Nesterov's acceleration terms first, an additional squared error norm involving the previous iterates is present, intrinsically implying an additional iterative rule (Eq.~\eqref{eq:OGDA_recurs}) in place of the original iterative rule that is essential for proving boundedness of the iterates.
In addition, due to the accumulated error across iterates, the maximum stepsize allowed in single-call algorithms is forced to be smaller. We believe that this is not an artifact of our analysis but is a general feature of OG methods.%
\footnote{Limited by space, we refer readers to~\S\ref{sec:proof_theo_convAG-OG} and~\S\ref{sec:proof_theo_stoc_conv_AG-OG} for technical details.}

\paragraph{Organization.}
The rest of this work is organized as follows.
\S\ref{sec:prelim} introduces the basic settings and assumptions necessary for our algorithm and theoretical analysis.
Our proposed AcceleratedGradient-OptimisticGradient (AG-OG) Descent Ascent algorithm is formally introduced in \S\ref{sec:AG-OG-Main} and further generalized to Stochastic AcceleratedGradient-OptimisticGradient (S-AG-OG) Descent Ascent in \S\ref{sec:S-AG-OG}.
We present our conclusions in \S\ref{sec:conclu}.
Due to space limitations, we defer all proof details along with results of numerical experiments to the supplementary materials.

\paragraph{Notation.}
For two sequences of positive scalars $\{a_n\}$ and $\{b_n\}$, we denote $a_n = \Omega(b_n)$ (resp.~$a_n = \cO(b_n)$) if $a_n \ge C b_n$ (resp.~$a_n \le C b_n$) for all $n$, and also $a_n = \Theta(b_n)$ if both $\Omega(b_n)$ and $a_n = \cO(b_n)$ hold, for some absolute constant $C > 0$, and $\tilde{\cO}$ or $\tilde{\Omega}$ is adopted in turn when $C$ contains a polylogarithmic factor in problem-dependent parameters.
Let $\lambda_{\max}(\Ab)$ and $\lambda_{\min}(\Ab)$ denote the maximal and minimal eigenvalues of a real symmetric matrix $\Ab$, and $\normop{\Ab}$ the operator norm $\sqrt{\lambda_{\max}(\Ab^\top\Ab)}$.
Let vector $\bz = [\xholder; \yholder]\in\RR^{n+m}$ denote the concatenation of $\xholder\in\RR^n$, $\yholder\in\RR^m$.
We use $\wedge$ (resp.~$\vee$) to denote the bivariate $\min$ (resp.~$\max$) throughout this paper.
For natural number $K$ let $[K]$ denote the set $\{1,\dots,K\}$.
Throughout the paper we also use the standard notation $\norm{\cdot}$ to denote the $\ell_2$-norm \blue{and $\normop{\cdot}$ to denote the operator norm of a matrix}.
We will explain other notations at their first appearances.

\section{Preliminaries}\label{sec:prelim}
In minimax optimization the goal is to find an (approximate) Nash equilibrium (or minimax point) of problem~\eqref{eq:Minmax_opt} (or~\eqref{eq:Minmax_separable}), defined as a pair $[\xholder^*; \yholder^*] \in \cX \times \cY$ satisfying:
\begin{align*}
\cL(\xholder^*, \yholder)
\leq 
\cL(\xholder^*, \yholder^*)
\leq
\cL(\xholder, \yholder^*)
.
\end{align*}
In order to analyze first-order gradient methods for this problem, we assume access to the gradients of the objective $\nabla_{\xholder}\cL(\xholder, \yholder)$ and $\nabla_{\yholder}\cL(\xholder, \yholder)$.
Finding the minimax point of the original convex-concave optimization problem~\eqref{eq:Minmax_opt} and~\eqref{eq:Minmax_separable} reduces to finding the point where the gradients vanish.
Accordingly, we use $W$ to denote the gradient vector field and $\bz = [\xholder; \yholder]\in\RR^{n+m}$:
\begin{align}\label{eq:grad_field}
W(\bz)
:=
\begin{pmatrix}
\nabla_{\xholder} \cL(\xholder, \yholder)
\\
-\nabla_{\yholder} \cL(\xholder, \yholder)
\end{pmatrix}
=
\begin{pmatrix}
\nabla f(\xholder) +  \nabla_{\bx}I(\bx, \by
)\\
-\nabla_{\by}I(\bx, \by
) + \nabla g(\yholder)
\end{pmatrix}
.
\end{align}
Based on this formulation, our goal is to find the stationary point of the vector field correponding to the monotone operator $W(\zholder)$, namely a point $\zholder^* = [\xholder^*; \yholder^*] \in \RR^{n + m}$ satisfying (in the unconstrained case) $W(\zholder^*) = 0$, which is referred to as the \emph{variational inequality (VI) formulation} of minimax optimization~\citep{gidel2018variational}.
The compact representation of the convex-concave minimax problem as a VI allows us to simplify the notation.

In the vector field~\eqref{eq:grad_field}, there are individual components that point along the direction optimizing $f, g$ individually, and a coupling component which corresponds to the gradient vector field of a separable minimax problem.
For the individual component, we let $
F(\zholder) := f(\xholder) + g(\yholder)
$ and correspondingly $
\nabla F(\zholder)
=
\left[
\nabla f(\xholder)
;
\nabla g(\yholder)
\right]
$.
For the coupling component, we define the operator $
\Matrix(\zholder) = 
\left[
\nabla_{\bx}I(\bx, \by)
;
-\nabla_{\by}I(\bx, \by)
\right]
$.
Note that the representation allows us to write $W(\zholder)$ as the summation of the two vector fields: $W(\zholder) = \nabla F(\bz) + H(\bz)$.

We introduce our main assumptions as follows:
\begin{assumption}[Convexity and Smoothness]\label{assu:convex_smooth}
We assume that $f(\cdot):\RR^n \rightarrow \RR$ is $\mu_f$-strongly convex and $L_f$-smooth, $g(\cdot): \RR^m \rightarrow \RR$ is $\mu_g$-strongly convex and $L_g$-smooth, and $I(\bx, \by)$ is convex-concave with blockwise smoothness parameters $L_H = I_{xx}\lor I_{yy} + I_{xy}$.
\end{assumption}
This implies that $F(\bz)$ is $(L_f\lor L_g)$-smooth and $(\mu_f\land\mu_g)$-strongly convex.
In addition $H(\cdot)$ is monotone, yielding the property that for all $\zholder, \zholder'\in \RR^{n + m}$:
\begin{align}\label{eq:H_prop}
\left\langle
H(\bz) - H(\bz')
,
\bz - \bz'
\right\rangle
\geq
0
.
\end{align}
The above assumption adds convexity and smoothness constraints to the individual components $f(\xholder)$ and $g(\yholder)$.
In addition, for the coupling component $\xholder^\top \mathbf{B} \yholder$ in the separable minimax problem~\eqref{eq:Minmax_separable}, without loss of generality, we assume that $\mathbf{B} \in \RR^{n\times m}, n \geq m >0$ is a tall matrix.
Note that as $\xholder$ and $\yholder$ are exchangeable, tall matrices cover all circumstances.

In the stochastic setting, we assume access to an unbiased stochastic oracle $\tilde{H}(\bz; \zeta)$ of $H(\bz)$ and an unbiased stochastic oracle $\nabla \tilde{F}(\bz; \xi)$ of $\nabla F(\bz)$. Furthermore, we consider the case where the variances of such stochastic oracles are bounded:

\begin{assumption}[Bounded Variance]\label{assu:bounded_variance}
We assume that the stochastic gradients admit bounded second moments $\sigma_H^2, \sigma_F^2 \ge 0$:
\begin{align*}
\EE_{\xi} \left[
\norm{ \tilde{H}(\bz; \zeta) - H(\bz)}^2
\right] \leq \sigma_H^2
,
\qquad
\EE_{\zeta} \left[
\norm{\nabla \tilde{F}(\bz; \xi) - \nabla F(\bz)}^2
\right] \leq \sigma_F^2
.
\end{align*}
\end{assumption}
For ease of exposition, we introduce the overall variance $\sigma^2 = 3\sqrt{2} \sigma_H^2 + 2\sigma_F^2$.
Note that the noise variance bound assumption is common in the stochastic optimization literature.%
\footnote{We leave the generalization to models of unbounded noise to future work.}
Under the above assumptions, our goal is to find an \emph{$\epsilon$-optimal minimax point}, a notion defined as follows.
\begin{definition}[$\epsilon$-Optimal Minimax Point]\label{defi_nearoptimal}
$[\xholder;\yholder]\in \cX\times\cY$ is called an \emph{$\epsilon$-optimal minimax point} of a convex-concave function $\cL(\xholder,\yholder)$ if $\|\xholder - \xholder^*\|^2 + \|\yholder - \yholder^*\|^2\le \epsilon^2$.
\end{definition}
It is obvious that when the accuracy $\epsilon=0$, $[\xholder;\yholder]$ is an (exact) optimal minimax point of $\cL(\xholder,\yholder)$.

\pb\section{AcceleratedGradient OptimisticGradient Descent Ascent}\label{sec:AG-OG-Main}
In this section, we discuss key elements of our algorithm design---consisting of \emph{OptimisticGradient Descent-Ascent} (OGDA) and \emph{Nesterov's acceleration method}---that together solve the separable minimax problem.
Such an approach allows us to demonstrate the main properties of our approach that will eventually guide our analysis in the discrete-time case.
In~\S\ref{sec:OGDA} and~\S\ref{sec:Nesterov} we review OGDA and Nesterov's acceleration.
In~\S\ref{sec:AG-OG} we present our approach to accelerating OGDA for bilinear minimax problems, yielding the Accelerated Gradient-Optimistic Gradient (AG-OG) algorithm, and we prove its convergence.
Finally in~\S\ref{sec:AG-OG_with_restarting} we show that proper restarting on top of the AG-OG algorithm achieves a sharp convergence rate that matches the lower bound of \citet{zhang2021lower}.

\pb\subsection{Optimistic Gradient Descent Ascent}\label{sec:OGDA}
The OptimisticGradient Descent Ascent (OGDA) algorithm has received considerable attention in the recent literature, especially for the problem of training Generative Adversarial Networks (GANs)~\citep{goodfellow2020generative}.
In the general variational inequality setting, the iteration of OGDA takes the following form~\citep{popov1980modification}:
\begin{align}
&
\zholder_{k+\frac12}=\zholder_k-\eta W(\zholder_{k-\frac12})
,\qquad 
\zholder_{k+1}=\zholder_k-\eta W(\zholder_{k+\frac12})
.\label{eq:OGDA_single}
\end{align}
Note that at step $k$, the scheme performs a gradient descent-ascent step at the \emph{extrapolated point} $\zhp$.
Equivalently, with simple algebraic modification~\eqref{eq:OGDA_single} can be written in a standard form~\citep{gidel2018variational}:
\begin{align}\label{eq:OGDA_oneline}
\zholder_{k+\frac12}
=
\zholder_{k-\frac12}
-
2\eta W(\zholder_{k-\frac12})
+
\eta W(\zholder_{k-\frac32})
.
\end{align}
Treating the difference $W(\zhm) - W(\bz_{t-\frac32})$ as a prediction of the future one $W(\zhp) - W(\zhm)$, this update rule can be viewed as an approximation of the implicit \emph{proximal point (PP) method}:
$$
\zholder_{k+\frac12}
=
\zholder_{k-\frac12}
-
\eta W(\zholder_{k+\frac12}).
$$
Another popular tractable approximation of the PP method is the extragradient (EG) method~\citep{korpelevich1976extragradient}:
Although conceptually similar to OGDA~\eqref{eq:OGDA_single}, EG requires two gradient queries per iteration and hence doubles the overall number of gradient computations.
Both OGDA and EG dynamics~\eqref{eq:OGDA_single} alleviate  cyclic behavior by extrapolation from the past and exhibit a complexity of $\cO(L/\mu \log(1/\epsilon))$~\citep{gidel2018variational, mokhtari2020unified} in general setting~\eqref{eq:Minmax_opt} with $L$-smooth, $\mu$-strongly-convex-$\mu$-strongly-concave objectives.\footnote{In fact an analogous result holds true for general smooth, strongly monotone variational inequalities~\citep{mokhtari2020unified}.}

\pb\subsection{Nesterov's Acceleration Scheme}\label{sec:Nesterov}
Turning to the minimization problem, while vanilla gradient descent enjoys an iteration complexity of $\cO(\kappa \log(1/\epsilon))$ on $L$-smooth, $\mu$-strongly convex problems, with $\kappa=L/\mu$ being the condition number,
Nesterov's method~\citep{nesterov1983method}, when equipped with proper restarting, achieves an improved iteration complexity of $\cO(\sqrt{\kappa}\log(1/\epsilon))$.
We adopt the following version of the Nesterov acceleration, known as the ``second scheme''~\citep{tseng2008accelerated, lin2020accelerated}:
\begin{subequations}\label{eq:nesterov_second}
\begin{numcases}{}
\zmd_k
&=
$
\frac{k}{k + 2} \zag_k + \frac{2}{k + 2} \zholder_k
$, \label{eq:extra_step} \\
\bz_{k+1}
&=
$
\bz_k - \eta_k \nabla F(\zmd_{k})
$, \label{eq:grad_step}            \\
\zag_{k+1}
&=
$
\frac{k}{k + 2} \zag_{k} + \frac{2}{k + 2} \bz_{k+1}
$.
\label{eq:avg_step}
\end{numcases}
\end{subequations}
Subtracting~\eqref{eq:extra_step} from~\eqref{eq:avg_step} and combining the resulting equation with~\eqref{eq:grad_step}, we conclude
\begin{align}
&\zag_{k+1} - \zmd_{k}
=
\frac{2}{k + 2} (\zholder_{k+1} - \zholder_{k})
=
-\frac{2\eta_k}{k+2} \nabla F(\zmd_{k})\notag
\\&\Rightarrow \qquad 
\zag_{k+1}
=
\zmd_{k} - \frac{2\eta_k}{k+2} \nabla F(\zmd_{k})
.
\label{eq:Nesterov_1}
\end{align}
Moreover, shifting the index forward by one in~\eqref{eq:extra_step} and combining it with~\eqref{eq:avg_step} to cancel the $\bz_{k+1}$ term, we obtain
\begin{align}
&\frac{k + 2}{k + 3} \zag_{k+1}
-
\zmd_{k+1}
= 
\frac{k}{k + 3} \zag_{k}
-
\frac{k+1}{k + 3} \zag_{k+1}
\\&\Rightarrow \qquad 
\zmd_{k+1}
=
\zag_{k+1} + \frac{k}{k + 3} \left( 
\zag_{k+1} - \zag_{k}
\right)
.\label{eq:Nesterov_2}
\end{align}
Thus, by a simple notational transformation,  \eqref{eq:Nesterov_1} plus~\eqref{eq:Nesterov_2} (and hence the original update rule~\eqref{eq:nesterov_second}) is exactly equivalent to the original updates of Nesterov's acceleration scheme~\citep{nesterov1983method}.
Here, $\zag_k$ denotes a \emph{$\frac{2}{k}$-weighted-averaged} iteration.
In other words, compared with vanilla gradient descent, $
\zholder_{k+1} = \zholder_{k} - \eta_k \nabla F(\zholder_{k})
$, Nesterov's acceleration conducts a step at the negated gradient direction evaluated at a \emph{predictive iterate} of the weighted-averaged iterate of the sequence.
This enables a larger choice of stepsize, reflecting the enhanced stability.
An analogous interpretation has been discussed in work on a heavy-ball-based acceleration method \citep[\S 1.3]{sebbouh2021almost}.

\pb\subsection{Accelerating OGDA on Separable Minimax Problems}\label{sec:AG-OG}
In this subsection and~\S\ref{sec:AG-OG_with_restarting}, we show that an organic combination of the two algorithms in~\S\ref{sec:OGDA} and~\S\ref{sec:Nesterov} achieves improved convergence rates and when equipped with scheduled restarting, obtains a sharp iteration complexity that matches~\citet{jin2022sharper} while only requiring a single gradient call per iterate.
Our algorithm is shown in Algorithm~\ref{alg:AG_OG}.
In Line~\ref{line:mdupd} and~\ref{line:agupd} the update rules of the evaluated point and the extrapolated point of $f$ follow that in~\eqref{eq:extra_step} and~\eqref{eq:avg_step}, while in Lines~\ref{line:halfupd} and~\ref{line:kupd} the updates follow the OGDA dynamics~\eqref{eq:OGDA_single} with each step modified by~\eqref{eq:grad_step}.
Algorithm~\ref{alg:AG_OG} can be seen as a synthesis of OGDA and Nesterov's acceleration, as it reduces to OGDA when $\nabla F = 0$ and to Nesterov's accelerated gradient when $H = 0$.

\begin{algorithm}[!t]
\caption{AcceleratedGradient-OptimisticGradient (AG-OG)$(\zag_0, \bz_0, \bz_{-1/2}, K)$}
\begin{algorithmic}[1]
\FOR{$k=0, 1,\dots,K-1$}
\STATE 
$
\zmd_{k}
=
(1 - \alpha_{k}) \zag_{k} + \alpha_{k} \bz_{k}
$\label{line:mdupd}
\STATE 
$\bz_{k + \frac12}
= 
\bz_{k} - \eta_k \left(H(\bz_{k-\frac12}) + \nabla F(\zmd_{k}) \right)
$\label{line:halfupd}
\STATE
$
\zag_{k+1}
=
(1 - \alpha_k) \zag_{k} + \alpha_k \bz_{k + \frac12}
$ \label{line:agupd}
\STATE 
$
\bz_{k+1} = 
\bz_{k} - \eta_k \left( H(\bz_{k + \frac12}) + \nabla F(\zmd_{k}) \right)
$\label{line:kupd}
\ENDFOR
\STATE {\bfseries Output:}
$\zag_K$
\end{algorithmic}
\label{alg:AG_OG}
\end{algorithm}

For theoretical analysis, we first state a nonexpansiveness lemma of $\bz_k$ with respect to $\zstar$, the unique solution to Problem~\eqref{eq:Minmax_separable}.
The proof is presented in~\S\ref{sec:proof_lem_deter_bound}.

\begin{lemma}[Nonexpansiveness]\label{lem:deter_bound}
Under Assumptions~\ref{assu:convex_smooth}, we set the parameters as $L = L_f \vee L_g$, $L_H = I_{xx}\lor I_{yy} + I_{xy}$, $\eta_k = \frac{k + 2}{2L + \sqrt{3+\sqrt{3}} L_H(k + 2)}$ and $\alpha_k = \frac{2}{k+2}$ in Algorithm~\ref{alg:AG_OG} \blue{and choose initialization $\bz_{-\frac12}=\zag_0=\bz_0$}, at any iterate $k < K$ we have
\begin{align*}
\norm{\bz_k - \zstar} 
\leq 
\norm{\bz_0 - \zstar}
.
\end{align*}
\end{lemma}

\begin{remark}
The result in Lemma~\ref{lem:deter_bound} is significant in that it establishes the \emph{last-iterate nonexpansiveness} ruled by the initialization $\bz_0$: the $\bz_k$ iteration stays in the ball centered at $\zstar$ with radius $\norm{\bz_0 - \zstar}$.
This is essential in proving convergence results of iteration $\zag_k$ where the main technical difficulty lies upon the additional recursive analysis due to gradient evaluation in a previous iterate.
From a past extragradient perspective, earlier analysis was focusing on the half iterates in extragradient step~\eqref{line:halfupd} ($\zholder_{k+\frac12}$ in our formulation).
In contrast, we perform a nonexpansiveness analysis on the integer steps ($\zholder_k$), serving as a critical improvement over the best previous result achieved by~\citet[Lemma 2(b)]{mokhtari2020convergence} (consider the bilinear coupling case where $f = 0$, $g = 0$), which merely admits a factor of $\sqrt{2}$ in terms of the Euclidean metric (i.e., $\norm{\zholder_k - \zstar} \leq \sqrt{2} \norm{\zholder_0 - \zstar}$).
\end{remark}

With the parameter choice in Lemma~\ref{lem:deter_bound}, Line~\ref{line:agupd} can also be seen as an average step that makes last iterates shrink toward the center of convergence.
Equipped with Lemma~\ref{lem:deter_bound}, we are ready to state the following convergence theorem for discrete-time AG-OG:
\begin{theorem}\label{theo_convAG-OG}
Under Assumption~\ref{assu:convex_smooth} and setting the parameters as in Lemma~\ref{lem:deter_bound}, the output of Algorithm~\ref{alg:AG_OG} on problem~\eqref{eq:Minmax_separable} satisfies:
\begin{equation}\label{convAG-OG}
\norm{\zag_K - \zstar}^2
\le
\left(
\frac{4L}{\mu(K + 1)^2}
+
\frac{2\sqrt{3+\sqrt{3}} L_H}{\mu(K + 1)}
\right)\|\bz_0 - \zstar\|^2
.
\end{equation}
\blue{Here we use $\mu$ to denote $\mu_f \wedge \mu_g$.}
\end{theorem}
The proof of Theorem~\ref{theo_convAG-OG} is provided in~\S\ref{sec:proof_theo_convAG-OG}.
The selection of $\alpha_k = \frac{2}{k+2}$ and $\eta_k = \frac{k+2}{2L + \sqrt{3+\sqrt{3}}L_H(k+2)}$ is vital for Nesterov's accelerated gradient descent to achieve desirable convergence behavior~\citep{nesterov1983method}.
This stepsize choice is larger
than the ones used in previous techniques~\citep{chen2017accelerated, du2022optimal}, which is brought by a fine-tuned analysis of~\eqref{eq:bound_diff} in the proof of Theorem~\ref{theo_convAG-OG}. 
The convergence rate in~\eqref{convAG-OG} for strongly convex problems is slow and not even linear.
However, in the next subsection we show how a simple restarting technique not only achieves the linear convergence rate, but also matches the lower bound in~\citet{zhang2021lower} in a broad regime of parameters.

\begin{algorithm}[!t]
\caption{AcceleratedGradient-OptimisticGradient with restarting (AG-OG with restarting)}
\begin{algorithmic}[1]
\REQUIRE Initialization $\bz_0^{0}$, total number of epochs $N\ge 1$, per-epoch iterates $(K_{n}: n=0,\dots,N-1)$
\FOR{$n=0, 1,\dots,N-1$}
\STATE 
$\bz_{\text{out}} = \text{AG-OG}(\bz_0^{n}, \bz_0^{n}, \bz_0^{n}, K_n)$
\STATE\label{lineoutput}
Set $
\bz_0^{n+1}\leftarrow \bz_{\text{out}}
$
\\
\text{//Warm-starting from the previous output}
\ENDFOR
\STATE {\bfseries Output:}
$\bz_0^{N}$
\end{algorithmic}
\label{alg:AG-OG_with_restarting}
\end{algorithm}

\subsection{Improving Convergence Rates via Restarting and Scaling Reduction}\label{sec:AG-OG_with_restarting}

Our algorithm design (Algorithm~\ref{alg:AG-OG_with_restarting}) utilizes the restarting technique, which is a well-established method to accelerate first-order methods in optimization literature~\citep{o2015adaptive, roulet2017sharpness, renegar2022simple}.
Our variant of restarting accelerates convergence through a novel approach inspired by contemporary variance-reduction strategies, similar to those presented in~\citet{li2022convergence, du2022optimal}. 
Our approach is distinct from previous ones~\citep{kovalev2021accelerated, thekumparampil2022lifted, jin2022sharper} that incorporate the last iterate EG/OGDA with Nesterov's acceleration.
By incorporating the extrapolated step of Nesterov's method as the average step of OGDA and utilizing restarting, we use a two-timescale analysis and scaling reduction technique to achieve optimal results under all regimes.
Although our algorithm is a multi-loop algorithm, the simplicity of restarting does not harm the practical aspect of our approach.

Normally, as $f$ and $g$ have different strong convexity parameters ($\mu_f$ and $\mu_g$), it is preferable in practice to have different stepsizes for the descent step on $f(x)$ and the ascent step on $g(y)$~\citep{du2017stochastic,lin2020near,du2022optimal}.
Accordingly, for our analysis we use a scaling reduction technique~\citep{du2022optimal} that allows us to consider applying a single scaling for all parameters without loss of generality.
Setting $\hat{\yholder} = \sqrt{\frac{\mu_g}{\mu_f}} \yholder$, we have $\nabla_{\hat{\yholder}} H(\zholder) = \sqrt{\frac{\mu_f}{\mu_g}} \nabla_{\yholder} H(\zholder)$ and $\nabla_{\hat{\yholder}} g(\yholder) = \sqrt{\frac{\mu_f}{\mu_g}} \nabla g(\yholder)$.
Other scaling changes are listed as follows:
\begin{align}\label{eq:scal_reduce}
&
L = L_f \vee \frac{\mu_f}{\mu_g} L_g
,\qquad
L_H = I_{xx} \vee I_{xy}\sqrt{\frac{\mu_f}{\mu_g}} \vee I_{yy}\frac{\mu_f}{\mu_g}
,\qquad
\eta_{k, \yholder} = \frac{\eta_k \mu_f}{\mu_g}
,\qquad
\mu = \mu_f
,
\end{align}
where by $\eta_{k, \yholder}$ we mean that when updating $\zholder = [\xholder; \yholder] \in \RR^{n + m}$, we adopt stepsize $\eta_k$ on the $\xholder$-part (first $n$ coordinates) and $\eta_{k, \yholder}$ on the $\yholder$-part (last $m$ coordinates).
Writing out in details, the update rules with adjusted stepsizes on $[\xholder;\yholder]$ are as follows:
$$
\left\{
\begin{aligned}
&
x_k^{{\rm md}}
=
(1-\alpha_k) x_k^{{\rm ag}}+\alpha_k x_k
\\&
y_k^{{\rm md}}
=
(1-\alpha_k) y_k^{{\rm ag}}+\alpha_k y_k
\\&
x_{k+\frac{1}{2}}
=
x_k - \eta_k\left(
I_x(x_{k-\frac{1}{2}}, y_{k-\frac{1}{2}}) + \nabla f(x_k^{{\rm md}})
\right)
\\&
y_{k+\frac{1}{2}}
=
y_k - \eta_{k, \yholder}\left(
-I_y(x_{k-\frac{1}{2}}, y_{k-\frac{1}{2}}) + \nabla g(y_k^{{\rm md}})\right)
\\&
x_{k+1}^{{\rm ag}}
=
(1-\alpha_k) x_k^{{\rm ag}}+\alpha_k x_{k+\frac{1}{2}}
\\&
y_{k+1}^{{\rm ag}}
=
(1-\alpha_k) y_k^{{\rm ag}}+\alpha_k y_{k+\frac{1}{2}}
\\&
x_{k+1}
=
x_k-\eta_k\left(
I_x(x_{k+\frac{1}{2}}, y_{k+\frac{1}{2}}) + \nabla f(x_k^{{\rm md}})
\right)
\\&
y_{k+1}
=
y_k - \eta_{k, \yholder}\left(
-I_y(x_{k+\frac{1}{2}}, y_{k+\frac{1}{2}}) + \nabla g(y_k^{{\rm md}})
\right)
\end{aligned}
\right.
$$
With the new scaling and restarting, we obtain Algorithm~\ref{alg:AG-OG_with_restarting}, which we refer to as ``AG-OG with restarting.''
The iteration complexity of AG-OG with restarting is stated in the following Corollary~\ref{corr:determ_complexity}.

\begin{corollary}\label{corr:determ_complexity}
Algorithm~\ref{alg:AG-OG_with_restarting}  on problem~\eqref{eq:Minmax_separable}  \blue{with $
K_n = \left\lceil
\sqrt{8e\frac{L}{\mu}}
\vee
4e\sqrt{3+\sqrt{3}} \frac{L_H}{\mu}
\right\rceil
$} outputs an $\epsilon$-optimal minimax point within a number $\cO(N)$ of iterates, for $N$ satisfying:
\begin{equation}\label{eq:determ_complexity}
N
=
\left( \sqrt{\frac{L}{\mu}} + \frac{L_H}{\mu}\right)
\log\left(\frac{1}{\epsilon}\right)
=
\left(\sqrt{\frac{L_f}{\mu_f}\vee \frac{L_g}{\mu_g}} + \frac{I_{xx}}{\mu_f}\vee \frac{I_{xy}}{\sqrt{\mu_f\mu_g}} \vee \frac{I_{yy}}{\mu_g}\right)
\log\left(\frac{1}{\epsilon}\right)
.
\end{equation}
\end{corollary}
We defer the proof of the corollary  to~\S\ref{sec:proof_corr_determ_complexity}.
\blue{
When restricted to the bilinear-coupled problem~\eqref{eq:Minmax_bilinear}, Eq.~\eqref{eq:determ_complexity} reduces to $\left(\sqrt{\frac{L_f}{\mu_f}\vee \frac{L_g}{\mu_g}} +  \frac{\normop{\mathbf{B}}}{\sqrt{\mu_f\mu_g}} \right)
\log\left(\frac{1}{\epsilon}\right)$, which exactly matches the lower bound result in~\citet{zhang2021lower} and therefore is optimal under this special instance.
}

The analysis in~\citet{du2022optimal}, which also adopts a restarted scheme is most similar with ours.
However, although OGDA can be written in past-EG form, the algorithm and theoretical analysis are fundamentally different~\citep{golowich2020tight}.
For example, in contrast to EG, the non-expansiveness argument for OGDA does not achieve a unity prefactor~\citep{mokhtari2020convergence}.
Our work proves a strict non-expansive property with prefactor $1$, and our technique is new compared with existing EG-based analysis and existing the OGDA-based analysis.

\subsection{Application to Separable Convex-Strongly-Concave (C-SC) Problem}
To extend our strongly-convex-strongly-concave (SC-SC) AG-OG algorithm complexity to the convex-strongly-concave (C-SC) setting, we define a \emph{regularization reduction} method that modifies the objective via the addition of a regularization term, which gives the objective function $
\cL_{\epsilon}(\bx, \by) = \cL(\bx, \by) + \epsilon \norm{\bx}^2
$, where $\epsilon$ is the desired accuracy of the solution.
The following Theorem~\ref{theo:C-SC} provides the complexity analysis; see~\S\ref{sec:proof_C-SC} for the proof.
\begin{theorem}\label{theo:C-SC}
The output of Algorithm~\ref{alg:AG_OG} under the same assumptions and stepsize choices of Theorem~\ref{theo_convAG-OG} on the objective function $\cL_{\epsilon}$ achieves the $\epsilon$-optimal minimax point of $L$ within the sample complexity of 
\begin{align*}
\cO\left(
\left(\sqrt{\frac{L_f}{\epsilon}\vee \frac{L_g}{\mu_g}} + \frac{I_{xx}}{\epsilon}\vee \frac{I_{xy}}{\sqrt{\epsilon\mu_g}} \vee \frac{I_{yy}}{\mu_g}\right)
\log\left(\frac{1}{\epsilon}\right)
\right)
\end{align*}
for the original C-SC problem.
\end{theorem}
The work of \citet{thekumparampil2022lifted} also provides a C-SC case result that is obtained by utilizing the smoothing technique~\citep{nesterov2005smooth}.
Additionally, they present a direct C-SC algorithm without smoothing.
On the other hand, \citet{kovalev2021accelerated} focuses on a different perspective on the C-SC problem where $\bx$ and $\by$ have strong interactions and obtains superlinear complexity of $\log(\frac{1}{\epsilon})$.
However, both of these papers are limited to bilinear coupling terms.
Our result, in contrast, targets a more general separable objective.
Our complexity in Theorem~\ref{theo:C-SC} matches the complexity for regularized reduction in \citet{thekumparampil2022lifted}.
Furthermore, Theorem~\ref{theo:C-SC} is optimal in the bilinear coupling case~\eqref{eq:Minmax_bilinear}.
The reason is that $\sqrt{\frac{L_f}{\epsilon}}$ is optimal for a pure minimization of convex $f$ \citep{nesterov2018lectures}, $\sqrt{\frac{L_g}{\mu_g}}$ is optimal for a pure maximization of strongly-concave $g$ \citep{nesterov2018lectures}, and $\frac{\|\mathbf{B}\|_{\op}}{\sqrt{\epsilon\mu_g}}$ matches the lower bound of bilinearly coupled concave-convex minimax optimization \citep{ouyang2021lower} when $f = 0$.

\subsection{Application to Bilinear Games}
While the complexity result for deterministic case in Corollary~\ref{corr:determ_complexity} has also been obtained in~\citet{thekumparampil2022lifted, kovalev2021accelerated} and \citet{jin2022sharper}, in addition to conceptual simplicity, our algorithm has the significant advantage that it yields a stochastic version and a convergence rate for the stochastic case.
By using proper averaging and scheduled restarting techniques, our algorithm is able to find near-optimal solutions and achieve an optimal sample complexity up to a constant prefactor. Additionally, we demonstrate that our algorithm can be reduced to a combination of the averaged iterates of the OGDA algorithm and a scheduled restarting procedure, which gives rise to a novel single-call algorithm that achieves an accelerated convergence rate on the bilinear minimax problem itself. Finally, we address the situation where there is stochasticity present in the problem.
Throughout this section, we consider Problem~\eqref{eq:Minmax_separable} with $\nabla f(\xholder), \nabla g(\yholder)$ being zero almost surely. Moreover, we assume the following bilinear form:
\begin{align}\label{eq:def_H_bilinear}
I(\xholder, \yholder)
=
\xholder^\top \mathbf{B} \yholder 
+
\xholder^\top \bu_{\bx}
+
\bu_{\by}^\top \yholder,
\end{align}
where $\xholder \in \RR^n$, $\yholder \in \RR^m$ with $n = m$, $\mathbf{B} \in \RR^{n \times n}$ is square and full-rank, and $\bu_{\bx}, \bu_{\by} \in \RR^n$ are two parameter vectors.
Algorithm~\ref{alg:AG_OG} reduces to an equivalent form of the OGDA algorithm (in the past extragradient form) with initial condition $
\zag_{0}
=
z_{-\frac12}
=
\zholder_0
$, which gives for all $k\ge 1$:
\begin{subequations}\label{eq:AG-OG-I}
\begin{numcases}{}
\bz_{k+\frac12}
&=
$
\bz_k - \eta_k H(\bz_{k-\frac12})
$,           \label{eq:agog_extra_step1}   
\\
\zag_{k+1}
&=
$
(1 - \alpha_k)\zag_{k} + \alpha_k \bz_{k+\frac12}
$     \label{eq:agog_ag_step}
\\
\bz_{k+1}
&=
$
\bz_k - \eta_kH(\bz_{k+\frac12})
$\label{eq:agog_extra_step2} 
.
\end{numcases}
\end{subequations}
By selecting the parameters $\alpha_k = \frac{2}{k+2}$ and $\eta_k = \frac{k+2}{2L + c_1 L_H(k+2)}$ with $L = 0$ and $c_1 = 2$ in~\eqref{eq:AG-OG-I}, we can prove a boundedness lemma (Lemma~\ref{lem:deter_bound_bilinear}, presented in~\ref{sec:proof_bilinear}), which is the bilinear game analogue of Lemma~\ref{lem:deter_bound} with an improved scheme of stepsize and demonstrates the non-expansiveness of the last iterate of the OGDA algorithm. The proof is deferred to~\S\ref{sec:proof_lem_deter_bound_bilinear}.

Non-expansiveness of the iterates further yields the following theorem whose proof is in~\S\ref{sec:theo_main_bilinear}.
\begin{theorem}\label{theo:main_bilinear}
When specified to the bilinear game case, setting the parameters as $\alpha_k = \frac{2}{k+2}$ and $\eta_k = \frac{1}{2L_H}$, the output of update rules~\eqref{eq:AG-OG-I} satisfies
\begin{equation}\tag{\ref{convAG-OG}}
\norm{\zag_K - \zstar}^2
\le
\frac{64\lambda_{\max}(\mathbf{B}^\top \mathbf{B})}{\lambda_{\min}(\mathbf{B}^\top \mathbf{B}) (K+1)^2} \norm{\zholder_0 - \zstar}^2
.
\end{equation}
Moreover, using the scheduled restarting technique, we obtain a complexity result that matches the lower bound of~\blue{\citet{ibrahim2020linear}}:
$$
O\left(
\sqrt{
\frac{
\lambda_{\max}\left(\mathbf{B}^{\top} \mathbf{B}\right)
}{
\lambda_{\min}\left(\mathbf{B}^{\top} \mathbf{B}\right)
}}
\log \left(\frac{1}{\varepsilon}\right)
\right)
.
$$
\end{theorem}
An extended result is also obtained in the stochastic setting; we refer interested readers to~\S\ref{sec:stoc_bilinear}.

\pb\section{Stochastic AcceleratedGradient OptimisticGradient Descent Ascent}\label{sec:S-AG-OG}
In this subsection, we generalize the theoretical performance of our AG-OG algorithm (Algorithm~\ref{alg:AG_OG} and~\ref{alg:AG-OG_with_restarting}) to the stochastic case where the rate-optimal convergence behavior is maintained.
The stochastic AG-OG algorithm replaces each batch gradient with its unbiased stochastic counterpart, with noise indices represented by $\zeta_t, \xi_t$.
The full stochastic AG-OG algorithm is shown in Algorithm~\ref{alg:AG_OG_stoc} in~\S\ref{sec_exp_sto}.

Based on a generalized nonexpasiveness lemma (Lemma~\ref{lem:stoc_boundedness}, presented in \S\ref{sec:proof_theo_stoc_conv_AG-OG}) which is the stochastic analogue of Lemma~\ref{lem:deter_bound}, we can proceed the analysis and arrive at our stochastic result.
See~\S\ref{sec:proof_theo_stoc_conv_AG-OG} for the proof.
\begin{theorem}\label{theo:stoc_conv_AG-OG}
Under Assumptions~\ref{assu:convex_smooth} and~\ref{assu:bounded_variance}, we take $\eta_k = \frac{k+2}{4L + D + 4\sqrt{2+\sqrt{2}} L_H(k+2)}$ where $D = \frac{\sigma}{C} \frac{A(K)}{\sqrt{\EE \norm{\bz_0 - \zstar}^2}}$ for $A(K) := \sqrt{(K+1)(K+2)(2K+3)/6}$ and some absolute constant $C > 0$.
Then the output of Algorithm~\ref{alg:AG_OG_stoc}  on problem~\eqref{eq:Minmax_separable} satisfies:
\begin{align*}
\EE \norm{\zag_K - \zstar}^2
\le
\left[
\frac{8L}{\mu(K+1)^2}
+
\frac{7.4(1+C^2)L_H}{\mu(K+1)}
\right]
\EE \norm{\bz_0 - \zstar}^2
+
\frac{2(C + \frac1C) \sigma }{\mu\sqrt{K+1}} \sqrt{\EE\norm{\bz_0 - \zstar}^2}
.
\end{align*}
\end{theorem}

\begin{remark}
Without knowledge of expected initial distance $\EE\norm{\bz_0 - \zstar}^2$ to the true minimax point, we need an alternative selection of stepsize $\eta_k$.
We assume an upper bound on $\norm{\bz_0 - \zstar}^2$ defined as $\Gamma_0$ and let $C=\frac{\Gamma_0}{\sqrt{\EE\norm{\bz_0 - \zstar}^2}}$.
The quantity $D = \frac{\sigma A(K)}{\Gamma_0}$ is hence known.
Thus
\begin{align*}
\EE \norm{\zag_K - \zstar}^2
\le
\left[
\frac{8L}{\mu(K+1)^2}
+
\frac{14.8L_H}{\mu(K+1)}
\right]
\Gamma_0^2
+
\frac{4\sigma}{\mu \sqrt{K+1}}\Gamma_0
.
\end{align*}
\end{remark}
Analogous to the method in~\S\ref{sec:AG-OG_with_restarting}, we restart the S-AG-OG algorithm properly and achieve the following complexity:
\begin{corollary}\label{corr:stocha_complexity}
With scheduled restarting imposed on top of Algorithm~\ref{alg:AG_OG_stoc}, Algorithm~\ref{alg:AG-OG_with_restarting}  on problem~\eqref{eq:Minmax_separable} outputs an $\epsilon$-optimal minimax point within $\cO(N)$ iterations, for $N$ satisfying:
\begin{align}
N
=
\left(
\sqrt{\frac{L}{\mu}} + \frac{L_H}{\mu}
\right)\log\left(\frac{1}{\epsilon}\right)
+ \frac{\sigma^2}{\mu_f^2 \epsilon^2}
\label{eq:determ_complexity}
=
\left(\sqrt{\frac{L_f}{\mu_f}\vee \frac{L_g}{\mu_g}} + \frac{I_{xx}}{\mu_f}\vee \frac{I_{xy}}{\sqrt{\mu_f\mu_g}} \vee \frac{I_{yy}}{\mu_g}\right)
\log\left(\frac{1}{\epsilon}\right)
+ \frac{\sigma^2}{\mu_f^2 \epsilon^2}
.
\end{align}
\end{corollary}
In the special case of bilinearly coupled SC-SC, the above result reduces to 
\begin{align*}
\left(\sqrt{\frac{L_f}{\mu_f}\vee \frac{L_g}{\mu_g}} +  \frac{I_{xy}}{\sqrt{\mu_f\mu_g}} \right)
\log\left(\frac{1}{\epsilon}\right)
+ \frac{\sigma^2}{\mu_f^2 \epsilon^2},
\end{align*}
which matches that of \citet{du2022optimal} and is rate-optimal.
The reason is that the first term (i.e., bias error) matches the lower bound of bilinearly coupled SC-SC in \citet{zhang2021lower}, and the second term (i.e., variance error) matches the worst-case statistical minimax rate.

\section{Discussion}\label{sec:conclu}
In this paper, we propose novel algorithms for both the deterministic setting (AG-OG) and a stochastic setting (S-AG-OG) which organically blends optimism with Nesterov's acceleration, featuring structural interpretability and simplicity.
Leveraging novel Lyapunov analysis, these algorithms achieve desirable polynomial convergence behavior.
Further by properly restarting the algorithms, AG-OG and its stochastic version theoretically enjoy rate-optimal sample complexity for finding an $\epsilon$-accurate solution.
Future directions include closing the gap between the upper and lower bounds for general separable minimax optimization, and generalizations to nonconvex settings.

\section*{Acknowledgements}
We sincerely thank the anonymous reviewers for their helpful comments, and thank Simon Shaolei Du for highly inspiring discussions.
This work is supported in part by Canada CIFAR AI Chair to GG, by the National Science Foundation CAREER Award 1906169 to QG, by the Mathematical Data Science program of the Office of Naval Research under grant number N00014-18-1-2764 and also the Vannevar Bush Faculty Fellowship program under grant number N00014-21-1-2941 and NSF grant IIS-1901252 to MIJ.

\pb
\bibliography{SAILreference}
\bibliographystyle{apalike}

\newpage
\appendix

\section{Examples of Separable Minimax Optimization}
In this section, we use two examples to showcase the applications of formulation~\eqref{eq:Minmax_separable}.
We refer the readers for other examples in prior works such as~\citet{thekumparampil2022lifted,kovalev2021accelerated,du2022optimal}.
In the first example, we demonstrate how the parameters of a linear state-value function can be estimated by solving~\eqref{eq:Minmax_separable}.
In the second example of robust learning problem, we illustrate how turning the disk constraint into a penalty term allows us to obtain an objective in the form of~\eqref{eq:Minmax_separable}.

\paragraph{Policy Evaluation in Reinforcement Learning.}
The policy evaluation problem in RL can be formulated as a convex-concave bilinearly coupled minimax problem.
We are provided a sequence of four-tuple $\{(s_t,a_t,r_t,s_{t+1})\}_{t=1}^n$, where

\begin{enumerate}[label=(\roman*)]
\setlength{\itemsep}{0pt}%
\item
$s_t$, $s_{t+1}$ are the current state (at time $t$) and future state (at time $t+1$), respectively;
\item
$a_t$ is the action at time $t$ generated by policy $\pi$, that is, $a_t = \pi(s_t)$;
\item
$r_t = r(s_t,a_t)$ is the reward obtained after taken action $a_t$ at state $s_t$.
\end{enumerate}

Our goal is to estimate the value function of a fixed policy $\pi$ in the discounted, infinite-horizon setting with discount factor $\gamma\in (0,1)$, where for each state $s$ the discounted reward
$$
V^\pi(s) \equiv \mathbb{E}\left[
\sum_{t=0}^\infty \gamma^t r_t
\,\bigg|\,
s_0 = s, a_t=\pi(s_t),\ \forall t\ge 0
\right]
.
$$
If a linear function approximation is adopted, i.e.~$V^\pi(s)=\phi(s)^{\top} \btheta$ where $\phi(\cdot)$ is a feature mapping from the state space to feature space, we estimate the model parameter $\btheta$ via minimizing the empirical \emph{mean-squared projected Bellman error (MSPBE)}:
\begin{equation}\label{MSPBE}
\min_{\btheta}~%
\frac12\|\mathbf{A}\btheta - \boldsymbol{b}\|_{\mathbf{C}^{-1}}^2
.
\end{equation}
where $\norm{\btheta}_{\mathbf{M}} \equiv \sqrt{\btheta^\top \mathbf{M} \btheta}$ denotes the $\mathbf{M}$-norm, for positive semi-definite matrix $\mathbf{M}$, of an arbitrary vector $\btheta$, and
\begin{align*}
\mathbf{A}
=
\frac{1}{n}\sum_{t=1}^n \phi(s_t) \left(\phi(s_t) - \gamma \phi(s_{t+1})\right)^\top
,
\quad
\boldsymbol{b}
=
\frac{1}{n}\sum_{t=1}^n r_t \phi(s_t)
,
\quad
\mathbf{C}
=
\frac{1}{n}\sum_{t=1}^n \phi(s_t) \phi(s_t)^\top
.
\end{align*}
Applying first-order optimization directly to \eqref{MSPBE} would necessitate computing (and storing) the inversion of matrix $\mathbf{C}$, or at least, the matrix-vector product $\mathbf{C}^{-1}\boldsymbol{v}$ for given vector $\boldsymbol{v}$ at each step, which would be computationally costly or even prohibited.
To circumvent inverting matrix $\mathbf{C}$ a reformulation via \emph{conjugate function} can be resorted to; that is, solving~\eqref{MSPBE} is equivalent to solving the following minimax problem~\citep{du2017stochastic,du2019linear}:
$$
\min_{\btheta} \max_{\bpsi}~%
- \bpsi^\top\mathbf{A}\btheta
- \frac12\|\bpsi\|_{\mathbf{C}}^2
+ \boldsymbol{b}^{\top} \bpsi
.
$$
Such an instance falls under the category of minimax problem~\eqref{eq:Minmax_separable} where the individual component is convex-concave, and is further enhanced to be strongly-convex-strongly-concave when a quadratic regularizer term is imposed and $\mathbf{C}$ is strictly positive definite.

\paragraph{Robust Learning.}
A robust learning or robust optimization problem targets to minimize an objective function (here the sum of squares) formulated as a minimax optimization problem~\citep{BEN[Robust],du2019linear,thekumparampil2022lifted}
\begin{equation}\label{ROBUST}
\min_{\xholder} \max_{\yholder: \|\yholder-\yholder_0\|\le\rds}~\frac12\|\mathbf{A}\xholder - \yholder\|^2
,
\end{equation}
where $\mathbf{A}$ is a coefficient matrix and $\yholder$ is a noisy observation vector, which is perturbed by a vector of $\rds$-bounded norm.
Transforming \eqref{ROBUST} to a penalized objective gives a formulation of
$
\min_{\xholder} \max_{\yholder}~%
\frac12\|\mathbf{A}\xholder - \yholder\|^2 - \pnlty\|\yholder - \yholder_0\|^2
.
$
When $\pnlty$ is selected to be strictly greater than $\frac12$, we get a strongly-convex-strongly-concave bilinearly coupled minimax optimization problem.

\section{Experiments}\label{sec_exp}
In this section, we empirically study the performance of our \AlgoName algorithm.
In these experimental results, we study both deterministic [\S\ref{sec_exp_det}] and stochastic settings [\S\ref{sec_exp_sto}], each of which we compare the state-of-the-art algorithms.
Throughout this section, the $x$-axis represents the number of gradient queries while the $y$-axis represents the squared distance to the minimax point.

\subsection{Deterministic Setting}\label{sec_exp_det}
We present results on synthetic quadratic game datasets:
\begin{align}\label{eq:quadratic_game}
\xholder^\top A_1 \xholder + \yholder^\top A_2 \xholder - \yholder^\top A_3 \yholder
,
\end{align}
with various selections of the eigenvalues of $A_1, A_2, A_3$.

\begin{figure*}[!tb]
\centering
\subfigure[$L_g = 64$, $\mu_g = 1$]{
\includegraphics[width=0.31\linewidth]{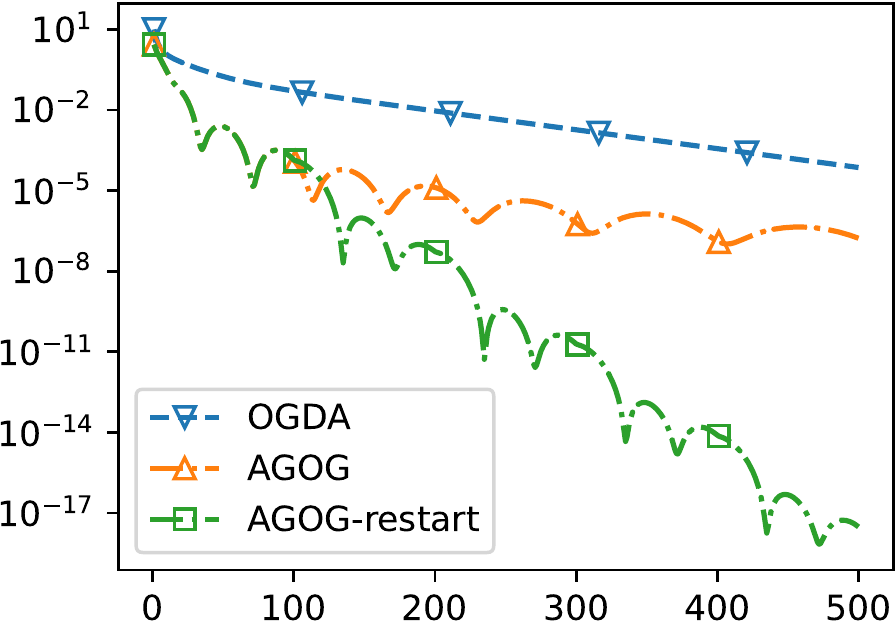}
\label{fig:large_bi_OGDA}
}
\subfigure[$L_g = 1$, $\mu_g = 1/64$]{
\includegraphics[width=0.31\linewidth]{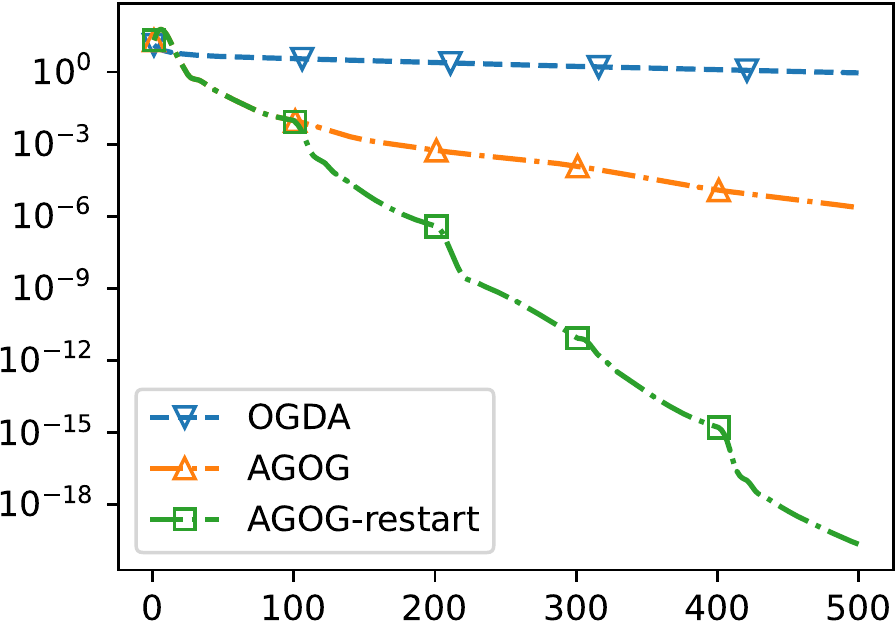}
\label{fig:larger_bi_OGDA}
}
\subfigure[$L_g = 4096$, $\mu_g = 64$]{
\includegraphics[width=0.3\linewidth]{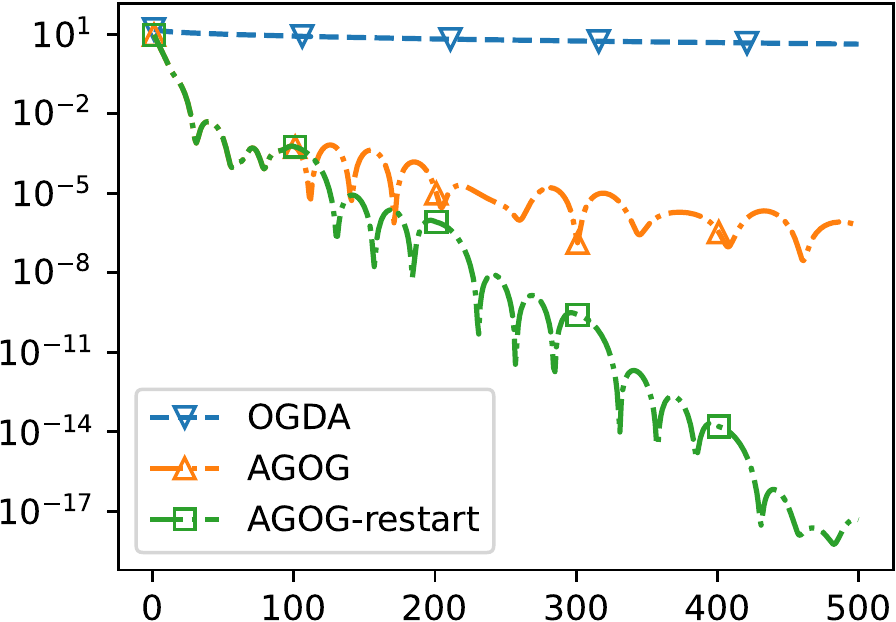}
\label{fig:large_L_OGDA}
}
\caption{Comparison with OGDA on different problem sets (Deterministic)}
\label{fig:figures_OGDA}
\end{figure*}

\paragraph{Comparison with OGDA.}

We use the single-call OGDA algorithm~\citep{gidel2018variational, hsieh2019convergence} as the baseline.
In Figure~\ref{fig:figures_OGDA} we plot the AG-OG algorithm and the AG-OG with restarting algorithm under three different instances.
We use stepsize $\eta_k = \frac{k+2}{2L + \sqrt{3+\sqrt{3}} L_H (k+2)}$ in both the AG-OG and the AG-OG with restarting algorithms and restart AG-OG with restarting once every $100$ iterates.
For the OGDA algorithm, we take stepsize $\eta = \frac{1}{2(L \vee L_H)}$ as is indicated by recent arts e.g.~\citep{mokhtari2020convergence}.
For the parameters of the problem~\eqref{eq:quadratic_game},
we fix $L_H = 1, L_f = 64, \mu_f = 1$ and scatter various values of $L_g, \mu_g$.
In Figure~\ref{fig:large_bi_OGDA} we take $L_g = 64$, $\mu_g = 1$.
In Figure~\ref{fig:larger_bi_OGDA} we take $L_g = 1$, $\mu_g = 1/64$ and in Figure~\ref{fig:large_L_OGDA} we take $L_g = 4096$, $\mu_g = 64$.
We see from Figures~\ref{fig:large_bi_OGDA},~\ref{fig:larger_bi_OGDA} and~\ref{fig:large_L_OGDA} when the problem has different $L_f, \mu_f$ and $L_g, \mu_g$, changing $L_g, \mu_g$ has larger impact on OGDA than on AG-OG, which matches our theoretical results.

\paragraph{Comparison with LPD.}

Next, we focus on comparison to the Lifted Primal-Dual (LPD) algorithm~\citep{thekumparampil2022lifted}.
We implement the AG-OG algorithm and its restarted version, the AG-OG with restarting.
Additionally, inspired by the technique of a single-loop direct-approach in~\citet{du2022optimal}, we consider a single-loop algorithm named \texttt{AG-OG-Direct} that takes advantage of the strongly-convex-strongly-concave nature of the problem.
We refer readers to~\citet{du2022optimal} for the ``direct" method.
The parameters of LPD are chosen as described in~\citet{thekumparampil2022lifted}.
For our AG-OG and AG-OG with restarting algorithms, we take $\eta_k = \frac{k+2}{2L + \sqrt{3+\sqrt{3}} L_H (k+2)}$ and the scaling parameters are taken as in Eq.~\eqref{eq:scal_reduce}.
For the AG-OG-direct algorithm, we take $\eta = \frac{1}{(1 + \sqrt{L/\mu_f + (\sqrt{3+\sqrt{3}}L_H)^2/\mu_f^2}) \mu_f}$ with the same set of scaling parameters.  We restart AG-OG with restarting once every 100 iterates.
\begin{figure*}[!tb]
\centering
\subfigure[$L_f = L_g = \mu_f = \mu_g = 1$, $L_H = 356, \mu_H = 101$]{
\includegraphics[width=0.31\linewidth]{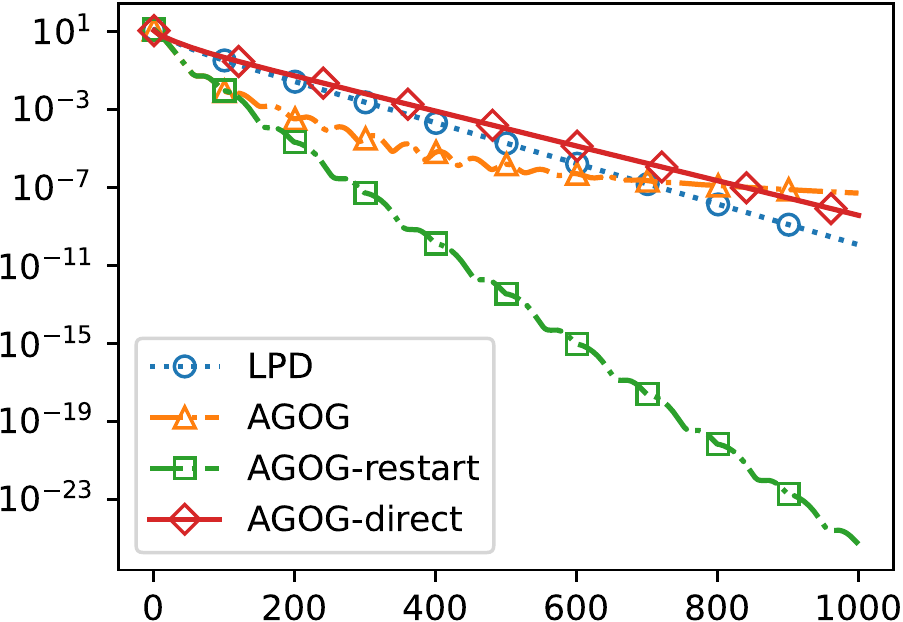}
\label{fig:large_bi_LPD}
}
\subfigure[$L_f = L_g = \mu_f = \mu_g = 1$, $L_H = 725, \mu_H = 101$]{
\includegraphics[width=0.31\linewidth]{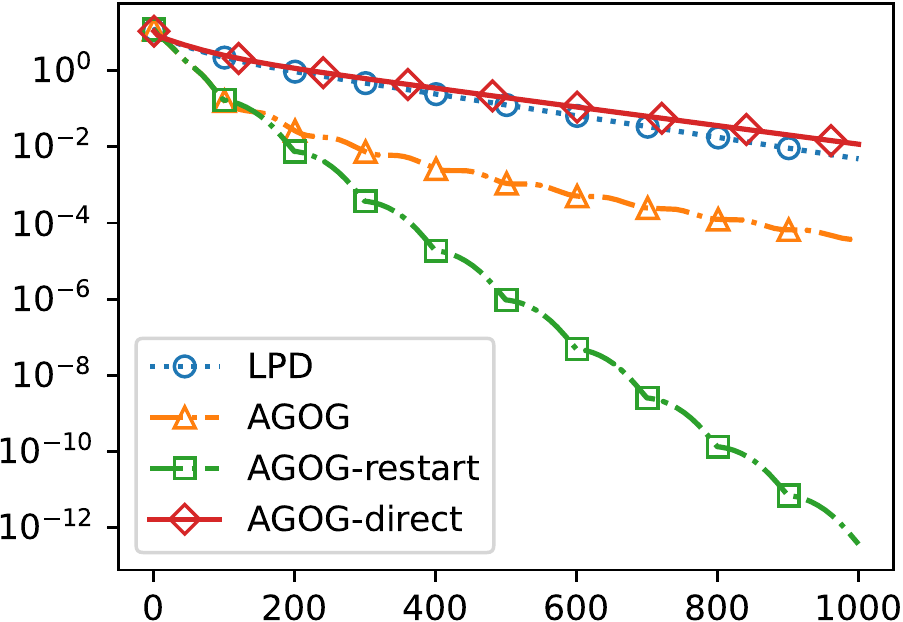}
\label{fig:larger_bi_LPD}
}
\subfigure[$L_f = L_g = 100$, $\mu_f = \mu_g = 1$, $L_H =\mu_H = 1$]{
\includegraphics[width=0.3\linewidth]{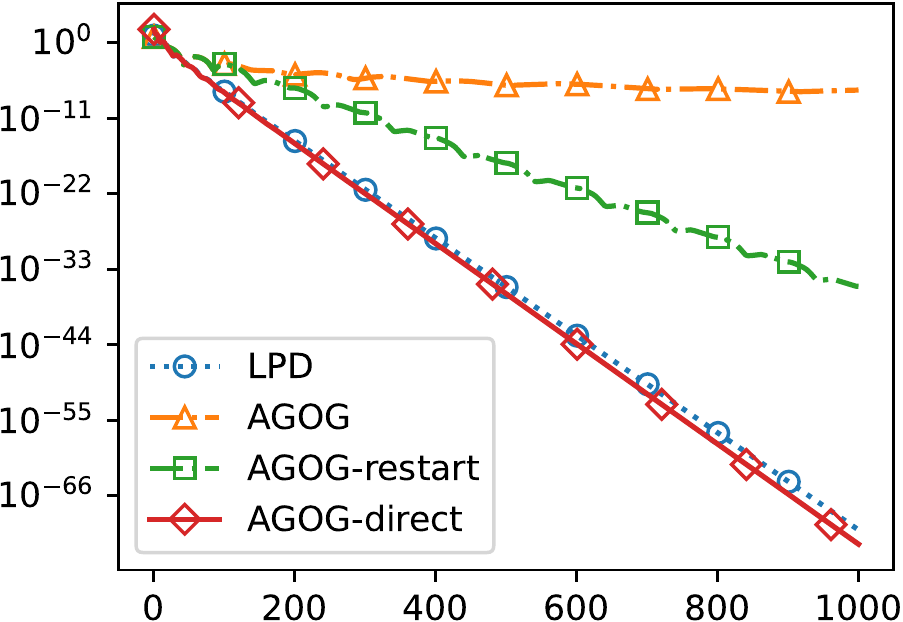}
\label{fig:large_L_LPD}
}
\caption{Comparison with LPD on different problem sets (Deterministic)}
\label{fig:figures_LPD}
\end{figure*}

In Figure~\ref{fig:large_bi_LPD}, the bilinear coupling component $\yholder^\top A_2 \xholder$ is the dominant part.
In Figure~\ref{fig:larger_bi_LPD}, we set the eigenvalues of $A_2$ even larger than in Figure~\ref{fig:large_bi_LPD}.
In Figure~\ref{fig:large_L_LPD}, $\xholder^\top A_1\xholder $ and $\yholder^\top A_3\yholder$ are the dominant terms.
More details on the specific designs of the matrices are shown in the caption of the corresponding figures.

We see from Figures~\ref{fig:large_bi_LPD} and~\ref{fig:larger_bi_LPD} that AG-OG with restarting (green line) outperforms LPD and MP in regimes where the bilinear term dominates, and when the eigenvalues of the coupling matrix increase, the performance of AG-OG with restarting relative to other algorithms is enhanced.
This is in accordance with our theoretical analysis.
In addition, AG-OG with restarting outperforms its non-restarted version (orange line) which has a gentle slope at the end.
On the other hand, when the individual component dominates, our AG-OG-direct (red line) slightly outperforms LPD.
Moreover, AG-OG-direct and LPD almost overlap in~\ref{fig:large_bi_LPD} and~\ref{fig:larger_bi_LPD}.

\begin{figure}[!tb]
\centering
\subfigure[$L_f = L_g = \mu_f = \mu_g = 1, L_H = 356, \mu_H = 101, \sigma=0.1$]{
\includegraphics[width=0.3\linewidth]{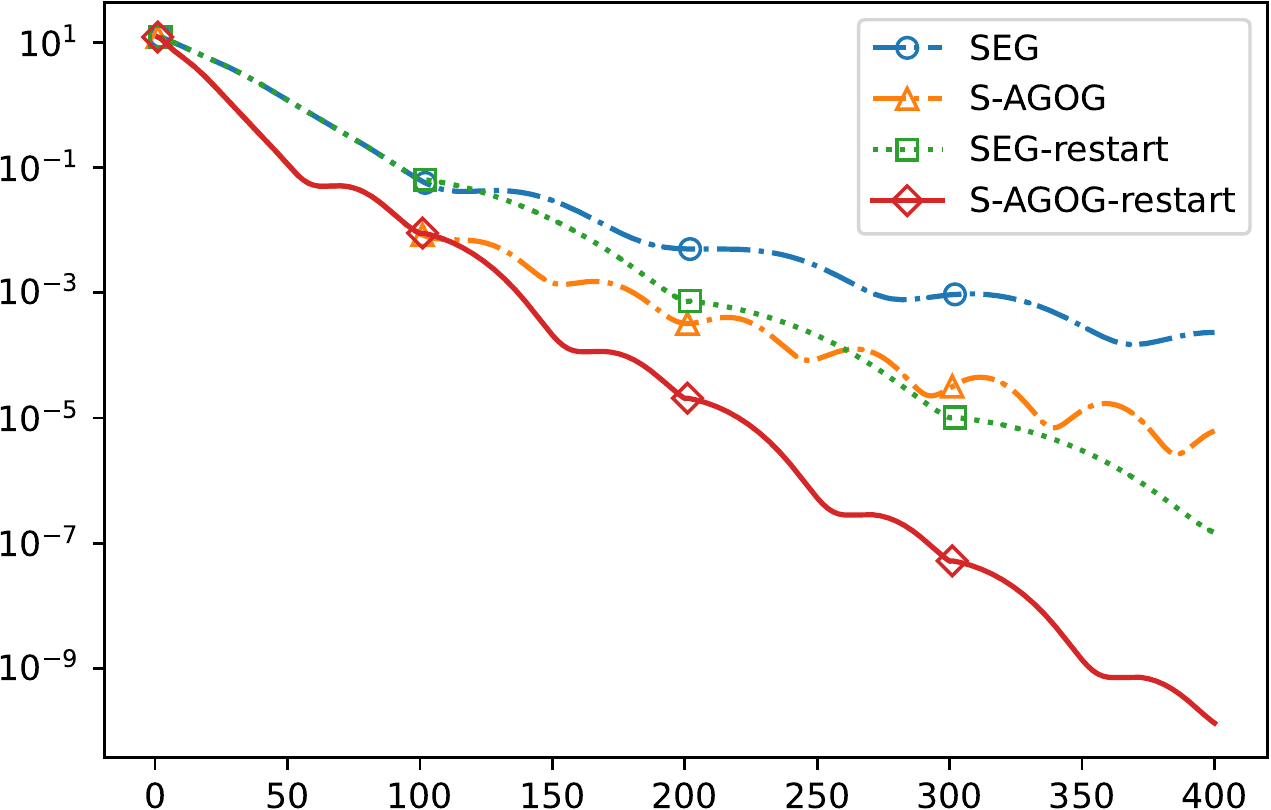}
\label{fig:dominate}
}
\subfigure[$L_f = L_g = 10, \mu_f = \mu_g = \mu_H = 1$, $L_H = 11, \sigma=0.1$]{
\includegraphics[width=0.3\linewidth]{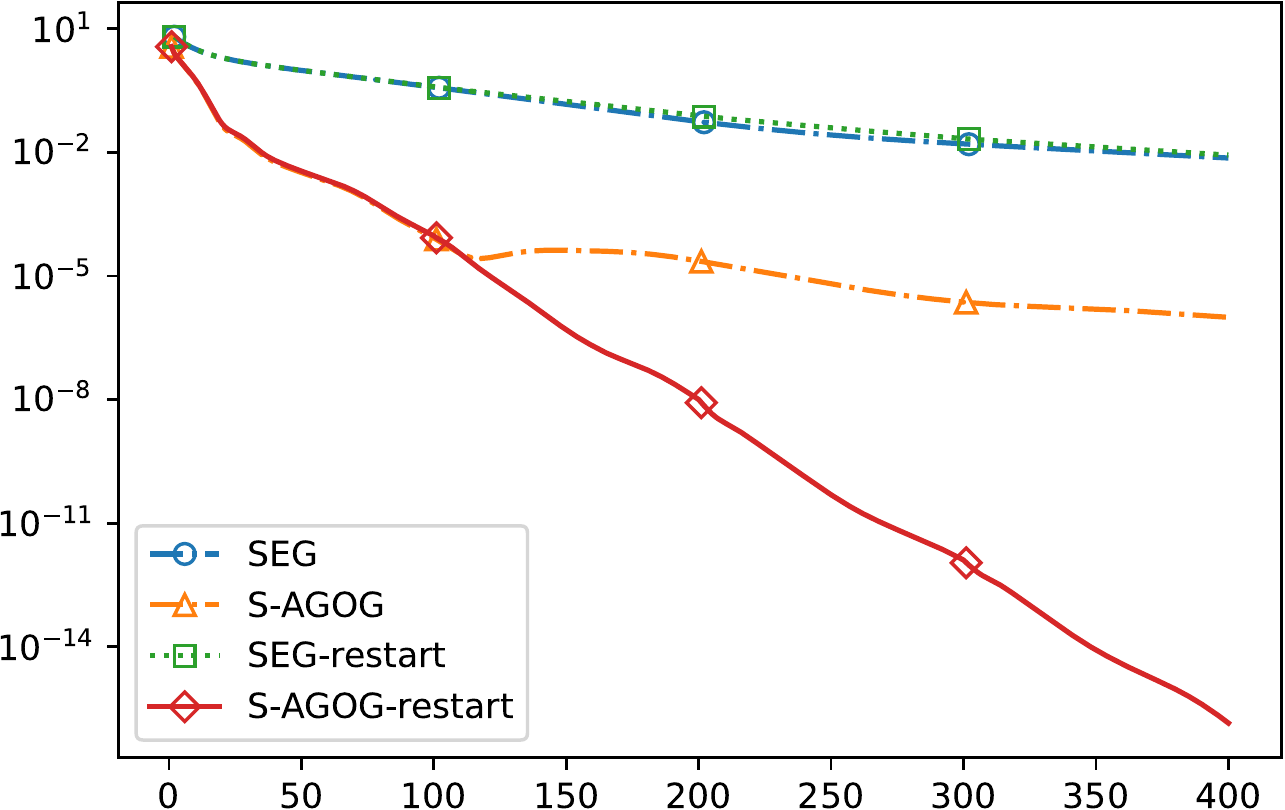}
\label{fig:balance_small_noise}
}
\subfigure[$L_f = L_g = 1, \mu_f = \mu_g = 1/8$, $L_H = \mu_H = 1, \sigma=0.1$]{
\includegraphics[width=0.3\linewidth]{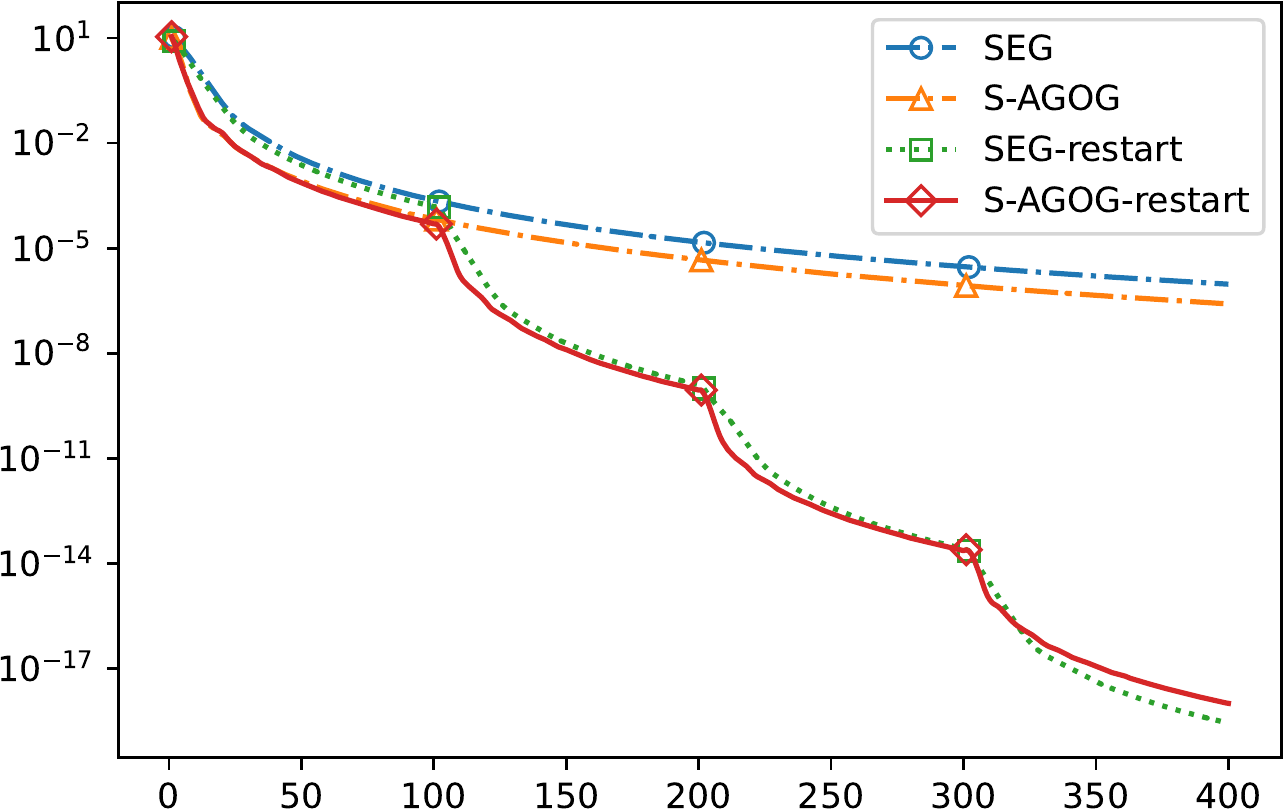}
\label{fig:balance_large_noise}
}
\caption{Comparison of algorithms on different problem sets (Stochastic)}
\label{fig:figures_stoc}
\end{figure}

\begin{algorithm}[!t]
\caption{Stochastic AcceleratedGradient-OptimisticGradient (S-AG-OG)$(\zag_0, \bz_0,\bz_{-1/2}, K)$}
\begin{algorithmic}[1]
\FOR{$k=0, 1,\dots,K-1$}
\STATE \label{line:mdupd_stoc}
$
\zmd_{k}
=
(1 - \alpha_{k}) \zag_{k} + \alpha_{k} \bz_{k}
$
\STATE \label{line:halfupd_stoc}
$\bz_{k + \frac12}
= 
\bz_{k} - \eta_k \left( \tilde{H}(\bz_{k-\frac12}; \zeta_{k-\frac12}) + \nabla \tilde{F}(\zmd_{k}; \xi_k) \right)
$
\STATE \label{line:agupd_stoc}
$
\zag_{k+1}
=
(1 - \alpha_k) \zag_{k} + \alpha_k \bz_{k + \frac12}
$
\STATE \label{line:kupd_stoc}
$
\bz_{k+1} = 
\bz_{k} - \eta_k \left( \tilde{H}(\bz_{k + \frac12}; \zeta_{k+\frac12}) + \nabla \tilde{F}(\zmd_{k}; \xi_k) \right)
$
\ENDFOR
\STATE {\bfseries Output:}
$\zag_K$
\end{algorithmic}
\label{alg:AG_OG_stoc}
\end{algorithm}

\subsection{Stochastic Setting}\label{sec_exp_sto}
We compared stochastic AG-OG and its restarted version (S-AG-OG) with Stochastic extragradient (SEG) SEG with restarting, respectively~\citep[cf.][]{li2022convergence}. The complete algorithm is shown in~\ref{alg:AG_OG_stoc}.
We note that we refer to the averaged iterates version of SEG everywhere when using SEG.
For SEG and SEG-restart, we use stepsize $\eta_k = \frac{1}{2(L\vee L_H)}$.
For AG-OG and AG-OG with restarting, we use stepsize $\eta_k = \frac{k+2}{2L + \sqrt{3 + \sqrt{3}}L_H(k+2)}$.
We restart every 100 gradient calculations for both SEG-restart and AG-OG-restart.

We use the same quadratic game setting as in~\eqref{eq:quadratic_game} except that we assume access only to noisy estimates of $A_1, A_2, A_3$. We add Gaussian noise to $A_1, A_2, A_3$ with $\sigma = 0.1$ throughout this experiment. We plot the squared norm error with respect to the number of gradient computations in Figure~\ref{fig:figures_stoc}.
In~\ref{fig:dominate} we consider larger eigenvalues for $A_2$ than $A_1$, $A_3$.
In~\ref{fig:balance_small_noise}, we let $A_1, A_2, A_3$ to be approximately of the same scale.
In~\ref{fig:balance_large_noise}, as the scale of the eigenvalues shrinks, the noise is relatively larger than in~\ref{fig:dominate} and~\ref{fig:balance_small_noise}.
The specific choice of parameters are shown in the caption of the corresponding figures.
We see from~\ref{fig:dominate},~\ref{fig:balance_large_noise} and~\ref{fig:balance_large_noise} that stochastic AG-OG with restarting achieves a more desirable convergence speed than SEG-restart. Also, the restarting technique significantly accelerates the convergence, validating our theory.

\section{Proof of Main Convergence Results}
This section collects the proofs of our main results,
Theorem~\ref{theo_convAG-OG} [\S\ref{sec:proof_theo_convAG-OG}],
Corollary~\ref{corr:determ_complexity} [\S\ref{sec:proof_corr_determ_complexity}],
and
Theorem~\ref{theo:stoc_conv_AG-OG} [\S\ref{sec:proof_theo_stoc_conv_AG-OG}].

\subsection{Proof of Theorem~\ref{theo_convAG-OG}}\label{sec:proof_theo_convAG-OG}
\begin{proof}[Proof of Theorem~\ref{theo_convAG-OG}]
We define the point-wise primal-dual gap function as:
\begin{align}\label{eq:def_gap}
V(\bz, \bz')
:=
F(\bz) - F(\bz') + \left\langle H(\bz'), \bz - \bz' \right\rangle
.
\end{align}
We first provide the following property for the primal-dual gap function:
\begin{lemma}\label{lemm_QuanBdd}
For $L$-smooth and $\mu$-strongly convex $F(\zholder)$, and for any $\zholder\in \RR^{n+m}$ we have
\beq\label{QuanBddx}
V(\zholder, \zstar)
=
F(\zholder) - F(\zstar)
+ 
\left\langle H(\zstar),	\zholder - \zstar \right\rangle
\ge
\frac{\mu}{2} \left\|\zholder - \zstar \right\|^2
.
\eeq
\end{lemma}

Proof of Lemma \ref{lemm_QuanBdd} is provided in \S\ref{sec_proof,lemm_QuanBdd}.

Our proof proceeds in the following steps:

\paragraph{Step 1: Estimating weighted temporal difference in squared norms.}
We first prove a result on bounding the temporal difference of the point-wise primal-dual gap between $\zag_k$ and $\zstar$, whose proof is delayed to~\S\ref{sec:proof_lem:basic}.

\begin{lemma}\label{lem:basic}
For arbitrary $\alpha_k \in (0, 1]$ and any $\omega_{\bz} \in \RR^{n+m}$ the iterates of Algorithm~\ref{alg:AG_OG} satisfy for $k = 1, \ldots, K$ almost surely
\begin{align}
&
V(\zag_{k+1}, \omega_{\bz})
-
(1 - \alpha_k) V(\zag_{k}, \omega_{\bz})
\leq 
\alpha_k\underbrace{
\left\langle 
\nabla F(\zmd_{k}) + \Matrix(\bz_{k + \frac12})
,
\bz_{k + \frac12} - \omega_{\bz}
\right\rangle 
}_{\mbox{I}}
+
\underbrace{
\frac{L \alpha_k^2}{2} \left\| \bz_{k + \frac12} - \bz_{k} \right\|^2
}_{\mbox{II}}
.
\label{eq:basic}
\end{align}
\end{lemma}
Note that in Lemma~\ref{lem:basic}, the term $\mbox{I}$ is an inner product that involves a gradient term.%
\footnote{In fact, this term reduces to $\left\langle \nabla f(\bz_k), \bz_k - \omega_{\bz} \right\rangle$ of the vanilla gradient algorithm if the features of accelerations and optimistic gradients are removed.}
The term $\mbox{II}$ is brought by gradient evaluated at $\zmd_k$.

Additionally, throughout the proof of Lemma~\ref{lem:basic}, we only leverage the convexity and $L$-smoothness of $f$ and the monotonicity of $H$ as in~\eqref{eq:H_prop}, as well as the update rules as in Line~\ref{line:mdupd} and Line~\ref{line:agupd}.
The proof involves no update rules regarding the gradient updates and hence Lemma~\ref{lem:basic} holds for the stochastic case as well.

Next, to further bound the inner product term $\mbox{I}$, we introduce a general proposition that holds for two updates starting from the same point.
Proposition~\ref{prop:PRecursion} is a slight modification from the proof of Proposition 4.2 in~\citet{chen2017accelerated} and analogous to Lemma 7.1 in~\citet{du2022optimal}.
We omit the proof here as the argument comes from simple algebraic tricks.
Readers can refer to~\citet{du2022optimal} for more details.

\begin{proposition}[Proposition 4.2 in~\citet{chen2017accelerated} and Lemma 7.1 in~\citet{du2022optimal}]\label{prop:PRecursion}
Given an initial point $\thetaholder \in \RR^d$, two update vectors $\boneholder, \btwoholder \in \RR^d$ and the corresponding outputs $\foneholder, \ftwoholder \in \RR^d$ satisfying:
\begin{align}\label{eq:Prox1}
\foneholder=\thetaholder - \boneholder
,\qquad
\ftwoholder=\thetaholder - \btwoholder
.
\end{align}
For any point $\arbholder\in \RR^d$ we have
\beq\label{eq:PRecursion}
\langle \btwoholder,\foneholder - \arbholder \rangle
\le
\frac{1}{2}\|\btwoholder - \boneholder\|^2 
+ 
\frac{1}{2} \left[
\|\thetaholder - \arbholder\|^2 - \|\ftwoholder - \arbholder\|^2 - \|\thetaholder - \foneholder\|^2
\right]
.
\eeq
\end{proposition}

Noting that the gradient term $\nabla F(\zmd_k) + H(\bz_{k+\frac12})$ in Term $\mbox{I}$ of inequality~\eqref{eq:basic} of Lemma~\ref{lem:basic} has been used in updating $\bz_k$ to $\bz_{k+1}$ in Line~\ref{line:kupd} in Algorithm~\ref{alg:AG_OG}.
Comparing Line~\ref{line:kupd} with Line~\ref{line:halfupd} and by letting $\thetaholder=\bz_{k}, \foneholder=\bz_{k+\frac12}, \ftwoholder=\bz_{k+1}$ in Proposition~\ref{prop:PRecursion}, we obtain an upper bound for the inner product term $\mbox{I}$:
\begin{align}
\notag
\eta_k \cdot \mbox{I} 
&\leq 
\frac{\eta_k^2}{2}\norm{H(\bz_{k+\frac12}) - H(\bz_{k-\frac12})}^2
+
\frac12 \left[ 
\norm{\bz_{k} - \omega_{\bz}}^2 - \norm{\bz_{k+1} - \omega_{\bz}}^2 - \norm{\bz_{k+\frac12} - \bz_{k}}^2    
\right]
\\&\leq 
\frac{L_H^2\eta_k^2}{2} \norm{\bz_{k+\frac12} - \bz_{k-\frac12}}^2
+
\frac12 \left[ 
\norm{\bz_{k} - \omega_{\bz}}^2 - \norm{\bz_{k+1} - \omega_{\bz}}^2 - \norm{\bz_{k+\frac12} - \bz_{k}}^2    
\right]
,
\label{eq:bound_I}
\end{align}
where the last inequality is due to properties of $H$ and the definition of $L_H$.
Combining Eqs.~\eqref{eq:basic} and~\eqref{eq:bound_I} we obtain
\begin{align}\label{eq:bound_diff}
&\notag
V(\zag_{k+1}, \omega_{\bz})
-
(1 - \alpha_k) V(\zag_{k}, \omega_{\bz})
\leq 
\frac{L_H^2\eta_k\alpha_k}{2} \norm{\bz_{k+\frac12} - \bz_{k-\frac12}}^2
\\&~\quad
+
\frac{\alpha_k}{2\eta_k} \left[ 
\norm{\bz_{k} - \omega_{\bz}}^2 - \norm{\bz_{k+1} - \omega_{\bz}}^2 - \norm{\bz_{k+\frac12} - \bz_{k}}^2    \right] 
+
\frac{L \alpha_k^2}{2}
\norm{\bz_{k+\frac12} - \bz_{k}}^2
.
\end{align}
This finishes Step 1.

\paragraph{Step 2: Building and solving the recursion.}
We first apply the following lemma to build connections between $\norm{\zhp - \zhm}^2$ and $\norm{\zhp - \bz_k}^2$, reducing Eq.~\eqref{eq:bound_diff} to the composition of sequences $\{\norm{\zholder_k - \omega_{\bz}}^2\}_{0\le k\le K-1}$ and $\{\norm{\bz_{k+\frac12} - \bz_k}^2\}_{0\le k\le K-1}$.
The proof of Lemma~\ref{lem:OGDA_recurs} is deferred to~\S\ref{sec:proof_lem:OGDA_recurs}.

\begin{lemma}\label{lem:OGDA_recurs}
For any stepsize sequence $\{\eta_k\}_{0\le k\le K-1}$ satisfying for some positive constant $c > 0$ and the Lipschitz parameter $L_H$ such that $
L_H\eta_k\leq \sqrt{\frac{c}{2}}
$ holds for all $k$.
Algorithm~\ref{alg:AG_OG} with initialization $\bz_{-\frac12}=\zag_0=\bz_0$ gives the following for any $k=0,\dots,K-1$:
\begin{equation}\label{eq:OGDA_recurs}
\left\|\bz_{\ph} - \bz_{\mh}\right\|^2
\le
2c^k\sum_{\ell=0}^k c^{-\ell} \left\|\bz_{\ell+\frac12} - \bz_{\ell}\right\|^2
.
\end{equation}
\end{lemma}
Combining Eqs.~\eqref{eq:bound_diff} and~\eqref{eq:OGDA_recurs}, bringing in the stepsize choice $\alpha_k = \frac{2}{k+2}$ and rearranging the terms, we obtain the following relation:
\begin{align*}
&
V(\zag_{k+1}, \omega_{\bz})
-
\frac{k}{k+2} V(\zag_{k}, \omega_{\bz})
\leq 
\frac{1}{\eta_k (k+2)} \left[ 
\norm{\bz_{k} - \omega_{\bz}}^2 - \norm{\bz_{k+1} - \omega_{\bz}}^2\right] 
\\&~\quad
-
\left( 
\frac{1}{\eta_k(k+2)}
-
\frac{2L}{(k+2)^2}
\right)\norm{\bz_{k+\frac12} - \bz_{k}}^2   
+
\frac{2L_H^2\eta_k}{k+2}
\sum_{\ell=0}^{k} c^{k-\ell}\norm{\bz_{\ell+\frac12} - \bz_{\ell}}^2
.
\end{align*}
Multiplying both sides by $(k+2)^2$, we obtain
\begin{align*}
&\notag
(k+2)^2V(\zag_{k+1}, \omega_{\bz})
-
[(k+1)^2 - 1] V(\zag_{k}, \omega_{\bz})
\leq 
\frac{k+2}{\eta_k} \left[ 
\norm{\bz_{k} - \omega_{\bz}}^2 - \norm{\bz_{k+1} - \omega_{\bz}}^2\right] 
\\&~\quad
-
\left(\frac{k+2}{\eta_k} - 2L\right)\norm{\bz_{k+\frac12} - \bz_{k}}^2   
+
2L_H^2 (k+2) \eta_k
\sum_{\ell=0}^{k} c^{k-\ell}\norm{\bz_{\ell+\frac12} - \bz_{\ell}}^2.
\end{align*}
Taking $
\eta_k = \frac{k+2}{2L + \sqrt{\frac{2}{c}}L_H(k+2)}
$, we have $
\frac{k+2}{\eta_k} - 2L
=
\sqrt{\frac{2}{c}}L_H (k+2)
$, and the previous inequality reduces to
\begin{align*}
&\notag
(k+2)^2V(\zag_{k+1}, \omega_{\bz})
-
[(k+1)^2 - 1] V(\zag_{k}, \omega_{\bz})
\\&\leq 
\left(2L + \sqrt{\frac{2}{c}}L_H(k+2)\right) \left[ 
\norm{\bz_{k} - \omega_{\bz}}^2 - \norm{\bz_{k+1} - \omega_{\bz}}^2\right]
\\&~\quad
- \sqrt{\frac{2}{c}}L_H(k+2)\norm{\bz_{k+\frac12} - \bz_{k}}^2   
+
\sqrt{2c}L_H(k+2)\sum_{\ell=0}^{k} c^{k-\ell}\norm{\bz_{\ell+\frac12} - \bz_{\ell}}^2
.
\end{align*}
Subtracting off term $V(\zag_{k+1}, \omega_{\bz})$ on both sides and summing over $k$ from $0$ to $K-1$, we have
\begin{align*}
&
\left[(K + 1)^2 - 1 \right] V(\zag_{K}, \omega_{\bz})
+
\left(2L + \sqrt{\frac{2}{c}}L_H(K+1)\right)
\|\bz_{K} - \omega_{\bz}\|^2
\\&\leq 
\left(2L + \sqrt{\frac{2}{c}} L_H \right) \norm{\bz_0 - \omega_{\bz}}^2 
+
\sqrt{\frac{2}{c}} L_H \sum_{k = 0}^{K-1} \left\|\bz_k - \omega_{\bz} \right\|^2
-
\sum_{k = 0}^{K-1} V(\zag_{k+1}, \omega_{\bz})
\\&~\quad
- \sqrt{\frac{2}{c}}  L_H \underbrace{\sum_{k = 0}^{K-1} (k + 2) \norm{\bz_{k+\frac12} - \bz_{k}}^2}_{\mbox{III}_1}
+
\sqrt{2c} L_H \underbrace{\sum_{k = 0}^{K-1} (k + 2) \sum_{\ell = 0}^{k} c^{k - \ell} \norm{\bz_{\ell+\frac12} - \bz_{\ell}}^2}_{\mbox{III}_2} 
.
\end{align*}
Simple algebra yields
\begin{align*}
\mbox{III}_2
&=
\sum_{\ell=0}^{K-1} \norm{\bz_{\ell+\frac12} - \bz_{\ell}}^2 \sum_{k = \ell}^{K-1} (k + 2) c^{k - \ell}
\leq 
\sum_{\ell=0}^{K-1} \left[\frac{\ell+2}{1 - c} + \frac{c}{(1 - c)^2}\right]\norm{\bz_{\ell + \frac12} - \bz_{\ell}}^2
.
\end{align*}
Straightforward derivations give that if we choose $c = \frac{2}{3 + \sqrt{3}}$, the inequality $
\sqrt{\frac{2}{c}} (k+2) 
\geq 
\sqrt{2c}
\left[ 
\frac{k+2}{1-c}+ \frac{c}{(1-c)^2}
\right]
$ holds for all $k \geq 0$.
Thus, summing $\mbox{III}_1$ and $\mbox{III}_2$ terms we have
\begin{align*}
&-
\sqrt{\frac{2}{c}} L_H \mbox{III}_1
+
\sqrt{2c} L_H \mbox{III}_2
\leq 0
.
\end{align*}
Finally, we solve the recursion and conclude
\begin{align}\label{eq:recurs_solve}
&\notag
\left[(K + 1)^2 - 1 \right] V(\zag_{K}, \omega_{\bz})
+
\left(2L + \sqrt{\frac{2}{c}}L_H(K+1)\right)
\|\bz_{K} - \omega_{\bz}\|^2
\\&\leq 
\left(2L + \sqrt{\frac{2}{c}} L_H \right) \norm{\bz_0 - \omega_{\bz}}^2 
+
\sqrt{\frac{2}{c}} L_H \sum_{k = 0}^{K-1} \left\|\bz_k - \omega_{\bz} \right\|^2
-
\sum_{k = 0}^{K-1} V(\zag_{k+1}, \omega_{\bz})
.
\end{align}
finishing Step 2.

\paragraph{Step 3: Proving $\bz_k$ stays nonexpansive with respect to $\zstar$.}
In Lemma~\ref{lem:deter_bound}, we show that $\bz_k$ always stays in the ball centered at $\zstar$ with radius $\norm{\bz_0 - \zstar}$.
The proof of this lemma is presented in~\S\ref{sec:proof_lem_deter_bound}.
\blue{
\begin{custom}{Lemma~\ref{lem:deter_bound}}[Nonexpansiveness, restated]
Under Assumptions~\ref{assu:convex_smooth}, we set the parameters as $L = L_f \vee L_g$, $L_H = I_{xx}\lor I_{yy} + I_{xy}$, $\eta_k = \frac{k + 2}{2L + \sqrt{3+\sqrt{3}} L_H(k + 2)}$ and $\alpha_k = \frac{2}{k+2}$ Algorithm~\ref{alg:AG_OG} with initialization $\bz_{-\frac12}=\zag_0=\bz_0$, at any iterate $k < K$ we have
\begin{align*}
\norm{\bz_k - \zstar} 
\leq 
\norm{\bz_0 - \zstar}
.
\end{align*}
\end{custom}
}

\paragraph{Step 4: Combining everything together.}
Bringing the nonexpansiveness result in Lemma~\ref{lem:deter_bound} into the solved recursion~\eqref{eq:recurs_solve}, setting $\omega_{\bz} = \zstar$ and rearranging, we obtain the following:
\begin{align*}
(K + 1)^2 V(\zag_{K}, \zstar)
&\leq (K + 1)^2 V(\zag_{K}, \zstar)
+
\left(2L + \sqrt{\frac{2}{c}}L_H(K+1)\right)
\|\bz_{K} - \zstar\|^2\notag
\\&\leq 
\left(2L + \sqrt{\frac{2}{c}}L_H(K+1)\right)
\|\bz_0 - \zstar\|^2
.
\end{align*}
Dividing both sides by $(K+1)^2$ and noting that Lemma~\ref{lemm_QuanBdd} implies $
V(\zag_K, \zstar) \geq \frac{\mu}{2}  \norm{\zag_K - \zstar}^2
$.
Hence, bringing in the choice of $c = \frac{2}{3+\sqrt{3}}$ concludes our proof of Theorem~\ref{theo_convAG-OG}.
\end{proof}
We finally remark that a limitation of this convergence rate bound is that the coefficient for $L_H$ in our stepsize choosing scheme is $\sqrt{3 + \sqrt{3}} \approx 2.175$ while an improved stepsize in this special case is $\frac{1}{2 L_H}$, yielding a sharper coefficient $2$.
Although the slight difference in constant factors does not harm the practical performance drastically, we anticipate that this constant might be further improved and leave it to future work.

\subsection{Proof of Corollary~\ref{corr:determ_complexity}}\label{sec:proof_corr_determ_complexity}
\begin{proof}[Proof of Corollary~\ref{corr:determ_complexity}]
The proof of restarting argument is direct.
By Eq.~\eqref{convAG-OG}, if we want $\norm{\zag_K - \zstar}^2 \leq \frac1e \norm{\bz_0 - \zstar}^2$ to hold, we can choose $K$ such that
\begin{align*}
\frac{4L}{\mu (K+1)^2} \leq \frac{1}{2e}
,\qquad 
\frac{2\sqrt{3+\sqrt{3}} L_H}{\mu (K+1)} \leq \frac{1}{2e}
.
\end{align*}
This is equivalent to 
\begin{align*}
K + 1 \geq \sqrt{\frac{8eL}{\mu}}, \qquad K+1 \geq \frac{4e\sqrt{3+\sqrt{3}} L_H}{\mu}
.
\end{align*}
For a given threshold $\epsilon > 0$, with the output of every epoch satisfying $\norm{\zag_K - \zstar}^2\leq \frac1e \norm{\bz_0 - \zstar}^2$, the total epochs required to obtain an $\epsilon$-optimal minimax point would be $\log\left(\frac{\norm{\bz_0 - \zstar}^2}{\epsilon}\right)$.
Thus, the total number of iterates required to get within the $\epsilon$ threshold would be:
\begin{align*}
\cO \left( \sqrt{\frac{L}{\mu}} + \frac{L_H}{\mu} \right)\cdot \log \left(\frac{1}{\epsilon}\right)
.
\end{align*}
Bringing the choice of scaling parameters in~\eqref{eq:scal_reduce} and we conclude our proof of Corollary~\ref{corr:determ_complexity}.
\end{proof}

\subsection{Proof of Theorem~\ref{theo:C-SC}}\label{sec:proof_C-SC}
\begin{proof}[Proof of Theorem~\ref{theo:C-SC}]  
For minimax problem, we recall that we define the primal-dual gap function as:
$$
V(\bz, \bz') = F(\bz) - F(\bz') + \left\langle 
H(\bz'), \bz - \bz'
\right\rangle
.
$$
We have for any pair of parameters $(\hat{\bx}, \hat{\by})$ and $(\bx, \by)$:
\begin{align*}
\cL(\hat{\bx}, \by) - \cL(\bx, \hat{\by})
&=
f(\hat{\bx}) - f(\bx)
+ 
g(\hat{\by})
-
g(\by) 
+
I(\hat{\bx}, \by) - I(\bx, \hat{\by})
\\&= 
f(\hat{\bx}) - f(\bx)
+ 
g(\hat{\by})
-
g(\by) 
+
I(\hat{\bx}, \by) 
- 
I(\bx, \by) 
+ 
I(\bx, \by)
-
I(\bx, \hat{\by})
\\&\leq 
f(\hat{\bx}) - f(\bx)
+ 
g(\hat{\by})
-
g(\by) 
+
\left\langle H(\bz),\hat{\bz} - \bz
\right\rangle 
\leq 
V(\hat{\bz}, \bz)
.
\end{align*}
Similarly for the regularized problem we define 
\begin{align*}
V_{\epsilon}(\bz, \bz') = V(\bz, \bz') + \frac{\epsilon}{2} \norm{\bx}^2 - \frac{\epsilon}{2} \norm{\bx'}^2
,
\end{align*}
and by the definition of $\cL_{\epsilon}$ we have
\begin{align*}
\cL_{\epsilon}(\hat{\bx}, \by) 
-
\cL_{\epsilon}(\bx, \hat{\by}) 
\leq 
V_{\epsilon}(\hat{\bz}, \bz)
,
\end{align*}
and moreover,
\begin{align*}
\max_{\by \in \cY} \cL_{\epsilon}(\hat{\bx}, \by) 
-
\min_{\bx \in \cX} \cL_{\epsilon}(\bx, \hat{\by})
\leq 
\max_{\bz \in \cX \times \cY} V_{\epsilon}(\hat{\bz}, \bz)
.
\end{align*}
By applying the AG-OG algorithm (Algorithm~\ref{alg:AG_OG}) onto the regularized objective $\cL_{\epsilon}$, if we can find a pair $\hat{\bz} = (\hat{\bx},\hat{\by})$ such that
\begin{align*}
\max_{\by \in \cY} \cL_{\epsilon}(\hat{\bx}, \by) 
-
\min_{\bx \in \cX} \cL_{\epsilon}(\bx, \hat{\by})
\leq 
\max_{\bz \in \cX \times \cY} V_{\epsilon}(\hat{\bz}, \bz) 
\leq \epsilon
.
\end{align*}
The result would imply for the original C-SC problem that
\begin{align*}
\max _{\by \in \mathcal{Y}} \cL\left(\hat{\bx}, \by\right)-\min _{\bx \in \mathcal{X}} \cL\left(\bx, \hat{\by}\right) 
&=
\max_{\by \in \mathcal{Y}} \cL\left(\hat{\bx}, \by\right) + \frac{\epsilon}{2} \norm{\hat{\bx}}^2 -\left(\min _{\bx \in \mathcal{X}} \cL\left(\bx, \hat{\by}\right)  + \frac{\epsilon}{2} \norm{\hat{\bx}}^2\right)
\\&\leq 
\max _{\by \in \mathcal{Y}} \left(\cL\left(\hat{x}, \by\right) + \frac{\epsilon}{2} \norm{\hat{\bx}}^2\right) -\min _{\bx \in \mathcal{X}}\left( \cL\left(\bx, \hat{\by}\right)  + \frac{\epsilon}{2} \norm{\bx}^2\right)
\\&\leq 
\max_{\by \in \mathcal{Y}} \cL_{\epsilon}\left(\hat{\bx}, \by\right)-\min _{\bx \in \mathcal{X}} \cL_{\epsilon}\left(\bx, \hat{\by}\right) 
\leq
\epsilon
.
\end{align*}
The left of this subsection is devoted to finding the gradient complexity of finding a $\hat{z} \in \cZ$ such that $\max_{\bz \in \cZ} V_{\epsilon}(\hat{\bz}, \bz)$.
\newcommand{\oz}{\boldsymbol{\omega}_{\bz}}
\\
By utilizing the results in \textbf{Step 2} in the proof of Theorem~\ref{theo_convAG-OG} to the objective $V_{\epsilon}$, we have
\begin{align*}
&V_{\epsilon}(\zag_{k+1}, \omega_{\bz})
-
\frac{k}{k+2} V_{\epsilon}(\zag_{k}, \omega_{\bz})\notag
\leq 
\frac{1}{\eta_k (k+2)} \left[ 
\norm{\bz_{k} - \omega_{\bz}}^2 - \norm{\bz_{k+1} - \omega_{\bz}}^2\right] 
\\&~\quad - 
\left( 
\frac{1}{\eta_k(k+2)}
-
\frac{2L}{(k+2)^2}
\right)\norm{\bz_{k+\frac12} - \bz_{k}}^2   
+
\frac{2\eta_k L_H^2}{k+2}
\sum_{\ell=0}^{k} c^{k-\ell}\norm{\bz_{\ell+\frac12} - \bz_{\ell}}^2.
\end{align*}
Multiplying both sides by $(k+2)(k+1)$, we obtain
\begin{align*}
&(k+2)(k+1)V_{\epsilon}(\zag_{k+1}, \omega_{\bz})
-
(k+1)k V_{\epsilon}(\zag_{k}, \omega_{\bz})\notag
\leq 
\frac{k+1}{\eta_k} \left[ 
\norm{\bz_{k} - \omega_{\bz}}^2 - \norm{\bz_{k+1} - \omega_{\bz}}^2\right] 
\\&~\quad - 
\left( 
\frac{k+1}{\eta_k}
-
2L
\right)\norm{\bz_{k+\frac12} - \bz_{k}}^2   
+
2(k+1)\eta_k L_H^2
\sum_{\ell=0}^{k} c^{k-\ell}\norm{\bz_{\ell+\frac12} - \bz_{\ell}}^2.
\end{align*}
Taking 
$
\eta_k = \frac{k+1}{2L + \sqrt{\frac{2}{c}}L_H(k+1)}
$, $c = \frac{2}{3 + \sqrt{5}}$ and adopting similar techniques as in the proof of Theorem~\ref{theo_convAG-OG}, we have
\begin{align}
&(K+2)(K+1) V_{\epsilon}(\zag_{K}, \oz)
+
\left(2L + \sqrt{\frac{2}{c}}L_HK\right)
\|\bz_{K} - \oz\|^2
\leq 
2L \norm{\bz_0 - \oz}^2 
+
\sqrt{\frac{2}{c}} L_H \sum_{k = 0}^{K-1} \left\|\bz_k - \oz \right\|^2
.
\label{eq:C-SC}
\end{align}
Taking $\oz = \zstar_{\epsilon}$ where $\zstar_{\epsilon}$ is the solution of the objective. Similarly as in Lemma~\ref{lem:deter_bound} in the proof of Theorem~\ref{theo_convAG-OG}, we can apply the same bootstrapping argument and derive for $\mu_{\epsilon}$ being the strongly convexity parameter of $V_{\epsilon}$, $L_{\epsilon}$ being the smoothness parameter of the regularized $F$, the following inequality
\begin{align*}
\norm{\zag_K - \zstar_{\epsilon}}^2
\le
\left(
\frac{4L}{\mu_{\epsilon}(K + 1)K}
+
\frac{2\sqrt{3+\sqrt{5}} L_H}{\mu_{\epsilon}(K + 1)}
\right)\|\bz_0 - \zstar_{\epsilon}\|^2
.
\end{align*}
Applying the same restarting as in Corollary~\ref{corr:determ_complexity}, the total number of iterates required to get within the $\epsilon$ threshold (in terms of $\norm{\zag_K - \zstar_{\epsilon}}^2$) should be 
\begin{align*}
\cO \left( 
\sqrt{\frac{L_{\epsilon}}{\mu_{\epsilon}}} 
+
\frac{L_H}{\mu_{\epsilon}}
\right)
\cdot 
\log \left(\frac{1}{\epsilon}\right)
.
\end{align*}
We note that in previous iterates $n = 0, \ldots, N-2$ in Algorithm~\ref{alg:AG-OG_with_restarting}, we have obtained a $\bz_0$ such that $\norm{\bz_0 - \zstar_{\epsilon}}^2 \leq \epsilon$. We then analyze at iteration $n = N - 1$.
Again from Equation~\eqref{eq:C-SC}, letting 
$\oz := \zstar_K = \arg\max_{\bz \in \cZ}V_{\epsilon}(\zag_K, \bz)$, we have
\begin{align}
&(K+2)(K+1) V_{\epsilon}(\zag_K, \zstar_K)
+
\left(2L + \sqrt{\frac{2}{c}}L_HK\right)
\|\bz_{K} - \zstar_K\|^2\notag
\leq 
2L \norm{\bz_0 - \zstar_K}^2 
+
\sqrt{\frac{2}{c}} L_H \sum_{k = 0}^{K-1} \left\|\bz_k - \zstar_K\right\|^2
.
\end{align}
As $V_{\epsilon}(\zag_K, \zstar_K) \geq 0$, we can apply the same boostrapping argument and derive 
\begin{align}
\frac{\mu_{\epsilon}}{2} \norm{\zag_K - \zstar_K}^2
\leq V_{\epsilon}(\zag_{K}, \zstar_K)
&\leq 
\frac{2L_{\epsilon} + \sqrt{3 + \sqrt{5}}L_HK}{K(K+1)}\norm{\bz_0 - \zstar_K}^2
.
\label{eq:needed_1}
\end{align}
On the other hand, we also have that 
\begin{align}
\frac{\mu_{\epsilon}}{2} \norm{\zag_K - \zstar_{\epsilon}}^2
\leq V_{\epsilon}(\zag_{K}, \zstar_{\epsilon})
&\leq 
\frac{2L + \sqrt{3 + \sqrt{5}}L_HK}{K(K+1)}\norm{\bz_0 - \zstar_{\epsilon}}^2
.
\label{eq:needed_2}
\end{align}
By analyzing the two inequalities~\eqref{eq:needed_1} and~\eqref{eq:needed_2}, we obtain that for sufficiently large $K$ (in an order of $\cO \left( 
\sqrt{\frac{L_{\epsilon}}{\mu_{\epsilon}}} 
+
\frac{L_H}{\mu_{\epsilon}}
\right)$) such that $\frac{4L + 2\sqrt{3 + \sqrt{5}}L_HK}{\mu_{\epsilon}K(K+1)} \leq \frac{1}{c_2}$,
\begin{align*}
\norm{\zstar_K - \zstar_{\epsilon}} 
&\leq 
\norm{\zag_K - \zstar_K}  +
\norm{\zag_K - \zstar_{\epsilon}
}
\leq 
\frac{1}{c_2} \left[ 
\norm{\bz_0 - \zstar_K}
+
\norm{\bz_0 - \zstar_{\epsilon}}
\right]
\\&\leq 
\frac{1}{c_2} \left[ 
\norm{\bz_0 - \zstar_{\epsilon}}
+
\norm{\zstar_{\epsilon} - \zstar_K}
+
\norm{\bz_0 - \zstar_{\epsilon}}
\right]
\leq 
\frac{2}{c_2 - 1}
\norm{\bz_0 - \zstar_{\epsilon}}
.
\end{align*}
Furthermore, we apply~\eqref{eq:needed_1} again and derive
\begin{align*}
\max_{\bz \in \cZ} V_{\epsilon}(\zag_K, \bz) = 
V_{\epsilon}(\zag_K, \zstar_K)
\leq 
\frac{1}{c_2}\norm{\bz_0 - \zstar_L}^2
\leq 
\frac{
\norm{\bz_0 - \zstar_K}}{c_2}
\leq 
\frac{\norm{\bz_0 - \zstar_{\epsilon}}
+
\norm{\zstar_{\epsilon} - \zstar_K}}{c_2}
\leq 
\frac{c_2 + 1}{c_2(c_2 - 1)}\norm{\bz_0 - \zstar_{\epsilon}}
.
\end{align*}
Taking $c_2 = 3$, and noting that we have obtained $\norm{\bz_0 - \zstar_{\epsilon}}^2\leq \epsilon$ in previous restarted iterates and combining the technique at the beginning of the proof of Theorem~\ref{theo:C-SC}, the total number of iterates in order to get 
$
\max _{y \in \mathcal{Y}} \cL_{\epsilon}\left(\hat{\bx}, \by\right)-\min _{\bx \in \mathcal{X}} \cL_{\epsilon}\left(\bx, \hat{\by}\right)
\leq 
\max_{z \in \cZ} V_{\epsilon}(\hat{\bz}, \bz)
\leq \epsilon 
$
is 
$
\cO \left( \sqrt{\frac{L}{\epsilon \wedge \mu_g}} + \frac{L_H}{\epsilon \wedge \mu_g} \right)\cdot \log \left(\frac{1}{\epsilon}\right)
$.
Applying a scaling reduction argument as in~\eqref{eq:scal_reduce} (the stepsize is dependent on $\epsilon$ after scaling reduction) gives a final complexity of
\begin{align*}
\cO\left(
\left(\sqrt{\frac{L_f}{\epsilon}\vee \frac{L_g}{\mu_g}} + \frac{I_{xx}}{\epsilon}\vee \frac{I_{xy}}{\sqrt{\epsilon\mu_g}} \vee \frac{I_{yy}}{\mu_g}\right)
\log\left(\frac{1}{\epsilon}\right)
\right).
\end{align*}
\end{proof}

\subsection{Proof of Theorem~\ref{theo:stoc_conv_AG-OG}}\label{sec:proof_theo_stoc_conv_AG-OG}
\begin{proof}[Proof of Theorem~\ref{theo:stoc_conv_AG-OG}]
For the stochastic case, we use the primal-dual gap function~\eqref{eq:def_gap} and proceeds in the following

\paragraph{Step 1: Estimating weighted temporal difference in squared norms.}
We mentioned in the proof of Theorem~\ref{theo_convAG-OG} that Lemma~\ref{lem:basic} holds for the stochastic case as well. Thus, we have 
\begin{align}
&
V(\zag_{k+1}, \omega_{\bz})
-
(1 - \alpha_k) V(\zag_{k}, \omega_{\bz})
\notag
\leq 
\alpha_k
\underbrace{\left\langle 
\nabla F(\zmd_{k}) + \Matrix(\bz_{k + \frac12}),
\bz_{k + \frac12} - \omega_{\bz}
\right\rangle 
}_{\mbox{I}}
+
\underbrace{\frac{L \alpha_k^2}{2} 
\left\| \bz_{k + \frac12} - \bz_{k} \right\|^2}_{\mbox{II}}
.
\tag{\ref{eq:basic}}
\end{align}
By applying Proposition~\ref{prop:PRecursion} to the iterates of Algorithm~\ref{alg:AG_OG_stoc}. Taking $\btheta = \bz_k, \bphi_1 = \bz_{k+\frac12}, \bphi_2 = \bz_{k+1}$ in Proposition~\ref{prop:PRecursion} and recalling the update rules in Algorithm~\ref{alg:AG_OG_stoc}, we obtain the following stochastic version of inequality~\eqref{eq:bound_I}:
\begin{align*}
&\eta_k \cdot \left\langle 
\nabla \tilde{F}(\zmd_{k}; \xi_k) + \nabla \tilde{\Matrix}(\bz_{k + \frac12};\zeta_{k+\frac12}),
\bz_{k + \frac12} - \omega_{\bz}
\right\rangle 
\\&\leq 
\frac12 \eta_k^2 \underbrace{\norm{\tilde{H}(\bz_{k+\frac12};\zeta_{k+\frac12}) - \tilde{H}(\bz_{k-\frac12};\zeta_{k-\frac12})}^2}_{(a)}
+
\frac12 \left[ 
\norm{\bz_{k} - \omega_{\bz}}^2 - \norm{\bz_{k+1} - \omega_{\bz}}^2 - \norm{\bz_{k+\frac12} - \bz_{k}}^2
\right]
.
\end{align*}

\paragraph{Step 2: Building and solving the recursion.}
Note that in the stochastic case, unlike Step 2 in the proof of Theorem~\ref{theo_convAG-OG}, before connecting $\norm{\bz_{k+\frac12} - \bz_{k-\frac12}}^2$ with $\norm{\bz_{k+\frac12} - \bz_k}^2$ to get an iterative rule, we need to bound the expectation of $(a)$ with additional noise first.

Throughout the rest of the proof of Theorem~\ref{theo:stoc_conv_AG-OG}, we denote
\begin{align*}
\Delta_h^{\ph}
=
\tilde{H}(\zhp; \zeta_{\ph})
-
H(\zhp)
,\qquad 
\Delta_f^{k}
=
\nabla \tilde{F}(\zmd_k; \xi_{k})
-
\nabla F(\zmd_k)
.
\end{align*}
Taking expectation over term $(a)$ in above, we use the following lemma to depict the upper bound of the quantity. The proof is delayed to~\S\ref{sec:proof_lem:stoc_nablaH}.
\begin{lemma}\label{lem:stoc_nablaH}
For any $\beta > 0$, under Assumption~\ref{assu:bounded_variance}, we have
\begin{align}\label{eq:stoc_nablaH}
\EE \norm{\tilde{H}(\zhp; \zeta_{\ph}) - \tilde{H}(\zhm; \zeta_{\mh})}^2
\leq 
(1 + \beta) L_H^2 \EE\norm{\zhp - \zhm}^2
+
\left(2 + \frac{1}{\beta}\right)\sigma_H^2
.
\end{align}
\end{lemma}
Taking $\beta = 1$ in Lemma~\ref{lem:stoc_nablaH} and bringing the result into the expectation of~\eqref{eq:basic}, we obtain that
\begin{align}\label{eq:stoc_starting_from}
&\EE V(\zag_{k+1}, \omega_{\bz})
-
(1 - \alpha_k) \EE V(\zag_{k}, \omega_{\bz})\notag
\\&\leq 
\frac{\alpha_k\eta_k}{2} \left[2L_H^2 \EE\norm{\bz_{k+\frac12}-\bz_{k-\frac12}}^2 
+ 
3\sigma_H^2\right]
+
\alpha_k\EE \left\langle 
\Delta_f^k
+
\Delta_h^{k+\frac12},
\bz_{k+\frac12} - \omega_{\bz}
\right\rangle \notag
\\&~\quad
+
\frac{L \alpha_k^2}{2} 
\EE \norm{\bz_{k + \frac12} - \bz_{k}}^2
+
\frac{\alpha_k}{2\eta_k} \EE \left[ 
\norm{\bz_{k} - \omega_{\bz}}^2 - \norm{\bz_{k+1} - \omega_{\bz}}^2 - \norm{\bz_{k+\frac12} - \bz_{k}}^2
\right]
.
\end{align}
Following the above inequality and following similar techniques as in Step 2 of the proof of Theorem~\ref{theo_convAG-OG}, we can derive the following Lemma~\ref{lem:ineq_stoc}, whose proof is delayed to~\S\ref{sec:proof_lem:ineq_stoc}.

\begin{lemma}\label{lem:ineq_stoc}
For the choice of stepsize such that $\eta_k L_H \leq \frac{\sqrt{c}}{2}$ holds for all $k$ and any constant $r > 0$, we have
\begin{align*}
&\EE V(\zag_{k+1}, \omega_{\bz})
-
(1 - \alpha_k) \EE V(\zag_{k}, \omega_{\bz})
\leq 
\frac{\alpha_k}{2\eta_k} \EE \left[ 
\norm{\bz_{k} - \omega_{\bz}}^2 - \norm{\bz_{k+1} - \omega_{\bz}}^2 
\right]
+
\frac{3\alpha_k\eta_k}{2(1-c)} \sigma_H^2 
\\&~\quad
+
2\alpha_k\eta_k L_H^2 \sum_{\ell=0}^{k} 
c^{k-\ell} 
\EE \norm{\bz_{\ell+\frac12} - \bz_{\ell}}^2
- 
\left(\frac{r\alpha_k}{2\eta_k}
-
\frac{L \alpha_k^2}{2} 
\right)
\EE \norm{\bz_{k+\frac12} - \bz_{k}}^2
+
\frac{\alpha_k\eta_k}{2(1-r)}\sigma_F^2
.
\end{align*}
\end{lemma} 

Recalling that $\alpha_k = \frac{2}{k+2}$, we have
\begin{align*}
&\EE V(\zag_{k+1}, \omega_{\bz})
-
\frac{k}{k+2} \EE V(\zag_{k}, \omega_{\bz})
\leq 
\frac{1}{\eta_k(k+2)} \EE \left[ 
\norm{\bz_{k} - \omega_{\bz}}^2 - \norm{\bz_{k+1} - \omega_{\bz}}^2 
\right]
\\&~\quad
+
\frac{4\eta_k L_H^2}{k+2} \sum_{\ell=0}^{k} 
c^{k-\ell} 
\EE \norm{\bz_{\ell+\frac12} - \bz_{\ell}}^2
- 
\left(\frac{r}{\eta_k(k+2)}
-
\frac{2L}{(k+2)^2} 
\right)
\EE \norm{\bz_{k+\frac12} - \bz_{k}}^2
\\&~\quad +
\frac{3\eta_k}{(1-c)(k+2)} \sigma_H^2 
+
\frac{\eta_k}{(1-r)(k+2)}\sigma_F^2
.
\end{align*}
Multiplying both sides by $(k+2)^2$ and taking $r = \frac12$, we obtain
\begin{align*}
&(k+2)^2 \EE V(\zag_{k+1}, \omega_{\bz})
-
[(k+1)^2 - 1] \EE V(\zag_{k}, \omega_{\bz})
\\&\leq 
\frac{k+2}{\eta_k} \EE \left[ 
\norm{\bz_k - \omega_{\bz}}^2 
-
\norm{\zp - \omega_{\bz}}^2
\right]
+
4\eta_k L_H^2(k+2)
\sum_{\ell=0}^{k} 
c^{k-\ell} 
\EE \norm{\bz_{\ell+\frac12} - \bz_{\ell}}^2
\\&~\quad - 
\left(\frac{r(k+2)}{\eta_k} - 2L\right) \EE \norm{\bz_{\ph} - \bz_k}^2 
+
\frac{3\eta_k(k+2)}{1-c}  \sigma_H^2
+
\frac{\eta_k(k+2)}{1-r} \sigma_F^2
\\&\leq 
\frac{k+2}{\eta_k} \EE \left[ 
\norm{\bz_k - \omega_{\bz}}^2 
-
\norm{\zp - \omega_{\bz}}^2
\right]
+
4\eta_k L_H^2(k+2)
\sum_{\ell=0}^{k} 
c^{k-\ell} 
\EE \norm{\bz_{\ell+\frac12} - \bz_{\ell}}^2
\\&~\quad - 
\left(\frac{k+2}{2\eta_k} - 2L\right) \EE \norm{\bz_{\ph} - \bz_k}^2 
+
\frac{3\eta_k(k+2)}{1-c}  \sigma_H^2
+
2\eta_k(k+2) \sigma_F^2
.
\end{align*}
Telescoping over $k=0, 1, \ldots K-1$ and using the same techniques as in the proof of Theorem~\ref{theo_convAG-OG}, we have for $\frac{k+2}{2\eta_k} \geq  2L + \frac{1}{\sqrt{c}} L_H(k+2)$ and $c = \frac{1}{2 + \sqrt{2}}$ ($c/(1-c) = \sqrt{2} - 1$, and recall $\sigma^2 = 3\sqrt{2} \sigma_H^2 + 2\sigma_F^2$ so that
\begin{align}\label{eq:stoc_control_V}
&\left[(K + 1)^2 - 1 \right] \EE V(\zag_{K}, \zstar)
+
\frac{K+1}{\eta_{K-1}}
\EE \norm{\bz_{K} - \zstar}^2\notag
\\&\leq 
\frac{2}{\eta_0} \EE \norm{\bz_0 - \zstar}^2 
+
\frac{2}{\sqrt{c}} L_H \sum_{k = 1}^{K-1} \EE \left\|\bz_k - \zstar \right\|^2
+
\sum_{k=0}^{K-1} (k+2) \eta_k \sigma^2
-
\sum_{k = 0}^{K-1} \EE V(\zag_{k+1}, \zstar)
.
\end{align}

\paragraph{Step 3: Proving $\bz_k$ stays within a neighbourhood of $\zstar$.}
We introduce the following Lemma~\ref{lem:stoc_boundedness}, whose proof is in~\S\ref{sec:proof_lem_stoc_boundedness}
\begin{lemma}\label{lem:stoc_boundedness}
Given the maximum epoch number $K>0$ and stepsize sequence $\{\eta_k\}_{k\in[K]}$ satisfying

\begin{enumerate}[label=(\alph*)]
\item\label{item:a}
$\frac{k+2}{\eta_k} - \frac{k+1}{\eta_{k-1}} = \frac{2}{\sqrt{c}}L_H$ for any $k < K$, we have for $\forall k \in [K-1]$:
\begin{align*}
\norm{\bz_{k} - \zstar}^2
\leq 
\norm{\bz_0 - \zstar}^2 
+
\frac{\eta_0}{2} \sum_{k=0}^{K-1} (k+2) \eta_k \sigma^2
.
\end{align*}

\item\label{item:b}
In addition if $\eta_k \leq \frac{k+2}{D}$ for $\forall k \in [K-1]$ where $D$ will be specified in~\ref{item:c} and taking $A(K) := \sqrt{(K+1)(K+2)(2K+3)/6}$, we have
\begin{align}\label{eq:stoc_bounded}
\norm{\bz_k - \zstar}^2 
&\leq 
\norm{\bz_0 - \zstar}^2
+
\frac{A(K)^2\sigma^2}{D^2}
.
\end{align}

\item\label{item:c}
Taking $D = \frac{\sigma}{C} \frac{A(K)}{\sqrt{\EE \norm{\bz_0 - \zstar}^2}}$ for some absolute constant $C > 0$ , bound~\eqref{eq:stoc_bounded} reduces to
\begin{align}\label{eq:stoc_bounded_final}
\norm{\bz_k - \zstar}^2
\leq 
\left(1 + C^2 \right)\norm{\bz_0 - \zstar}^2
.
\end{align}
\end{enumerate}
\end{lemma}

\paragraph{Step 4: Combining everything together.}
Combining the choice of stepsize $\eta_k$ in~\ref{item:a},~\ref{item:b} in Lemma~\ref{lem:stoc_boundedness} and $
\frac{k+2}{2\eta_k} \geq 2L + \frac{1}{\sqrt{c}} L_H(k+2)$, and bound~\eqref{eq:stoc_control_V} with Eq.~\eqref{eq:stoc_bounded_final}, by rearranging the terms again, we conclude that for $
\eta_k = \frac{k+2}{4L + D + 4\sqrt{2+\sqrt{2}} L_H(k+2)}
$,
\begin{align*}
(K + 1)^2 \EE V(\zag_{K}, \zstar)
&\leq 
\left(4L + 2\sqrt{2+\sqrt{2}}(K+1)\left(1+C^2\right) L_H\right)
\EE \norm{\bz_0 - \zstar}^2 
\\&~\quad
+
\left(C + \frac1C\right) \sigma A(K) \sqrt{\EE\norm{\bz_0  - \zstar}^2}
.
\end{align*}
Dividing both sides by $(K+1)^2$ and noting that $V(\zag_K, \zstar) \geq \frac{\mu}{2} \EE \norm{\zag_K - \zstar}^2$, we have
\begin{align*}
\EE \norm{\zag_K - \zstar}^2
\leq 
\left[
\frac{8L}{\mu(K+1)^2}
+
\frac{7.4(1+C^2)L_H}{\mu(K+1)}
\right]
\EE \norm{\bz_0 - \zstar}^2
+
\frac{2(C + \frac1C) \sigma }{\mu\sqrt{K+1}} \sqrt{\EE\norm{\bz_0 - \zstar}^2}
,
\end{align*}
hence concluding the entire proof of Theorem~\ref{theo:stoc_conv_AG-OG}.

\end{proof}

\section{Proof of Bilinear Game Cases}\label{sec:proof_bilinear}

\subsection{Proof of Theorem~\ref{theo:main_bilinear}}\label{sec:theo_main_bilinear}
\begin{proof}[Proof of Theorem~\ref{theo:main_bilinear}]\label{sec:proof_theo_main_bilinear}

\paragraph{Step 1: Non-expansiveness of OGDA last-iterate.}
We start by the non-expansiveness Lemma, whose proof is in~\S\ref{sec:proof_lem_deter_bound_bilinear}.
\begin{lemma}[Bounded Iterates]\label{lem:deter_bound_bilinear}
Following~\eqref{eq:AG-OG-I}, at any iterate $k < K$, $\bz_k$ stays within the region defined by the initialization $\bz_0$:
\begin{align*}
\norm{\bz_k - \zstar} 
\leq 
\norm{\bz_0 - \zstar}
,
\end{align*}
where we recall that $\zstar = [\xholder^*; \yholder^*]$ denotes the unique solution of Problem~\eqref{eq:Minmax_bilinear} with $\nabla f, \nabla g = 0$ and $I$ defined in~\eqref{eq:def_H_bilinear}.
\end{lemma}

By Lemma~\ref{lem:deter_bound_bilinear}, for any $0 \leq k < K$, we have
\begin{align*}
\norm{\bz_k - \zstar} \leq \norm{\bz_0 - \zstar}
.
\end{align*}

\paragraph{Step 2:}
Recalling that we take $\alpha_k = \frac{2}{k+2}$in~\eqref{eq:agog_ag_step} of~\eqref{eq:AG-OG-I}. Thus, we obtain the following
\begin{align*}
\zag_{k+1}
&=
\frac{k}{k+2} \zag_k + \frac{2}{k+2} \zholder_{k+\frac12}
.
\end{align*}
Subtracting both sides by $\zstar$ and multiplying both sides by $(k+1)(k+2)$, we have
\begin{align*}
(k+1)(k+2) \left(\zag_{k+1} - \zstar\right) 
&=
k(k+1) \left(\zag_k - \zstar\right) + 2(k+1) \left(\zholder_{k+\frac12} - \zstar\right)
.
\end{align*}
Telescoping over $k = 0, \ldots, K-1$ and we conclude that
\begin{align}\label{eq:recurs}
K(K+1) \left(\zag_{K} - \zstar \right)
&= 
2 \sum_{k=0}^{K-1} (k+1) \left(\zholder_{k+\frac12} - \zstar \right)
.
\end{align}
Moreover, according to the update rule~\eqref{eq:agog_extra_step2}, we have that 
\begin{align}\label{eq:recurs_2}
K z_K - \sum_{k=0}^{K-1} z_k 
=
\sum_{k=0}^{K-1}(k+1)\zholder_{k+1} - (k+1) \zholder_k 
=
\sum_{k=0}^{K-1}\eta_k (k+1) H(\zholder_{k+\frac12})
.
\end{align}
Recalling that $\eta_k = \frac{1}{2L_H}$, combining~\eqref{eq:recurs_2} with~\eqref{eq:recurs} and taking the squared norm on both sides, we conclude that
\begin{align}
\norm{K(K+1) \left(\zag_{K} - \zstar\right)}^2
&= 
\norm{2 \sum_{k=0}^{K-1} (k+1) \left(\zholder_{k+\frac12} - \zstar\right)}^2
\leq 
\frac{1}{\lambda_{\min}(\mathbf{B}^\top \mathbf{B})}\norm{2 \sum_{k=0}^{K-1} (k+1) H(\zholder_{k+\frac12})}^2\notag
\\&=
\frac{16L_H^2}{\lambda_{\min}(\mathbf{B}^\top \mathbf{B})} 
\norm{K\bz_K - \sum_{k=0}^{K-1}\bz_k}^2
=
\frac{16L_H^2}{\lambda_{\min}(\mathbf{B}^\top \mathbf{B})} 
\norm{\sum_{k=0}^{K-1} \left[\left(z_K - \zstar\right) - \left(\bz_k - \zstar\right)\right]}^2\notag
\\&\leq 
\frac{16L_H^2}{\lambda_{\min}(\mathbf{B}^\top \mathbf{B})} 
K\cdot 
\sum_{k=0}^{K-1} 
\left[
2\norm{\bz_K - \zstar}^2
+
2\norm{\bz_k - \zstar}^2
\right]
.
\label{eq:OGDA_proof_final}
\end{align}
Applying non-expansiveness in Lemma~\ref{lem:deter_bound_bilinear} in~\eqref{eq:OGDA_proof_final}, bringing $L_H = \sqrt{\lambda_{\max}(\mathbf{B}^\top \mathbf{B})} $and rearranging, we conclude that 
\begin{align*}
\norm{\zag_K - \zstar}^2
\leq 
\frac{64 \lambda_{\max}(\mathbf{B}^\top \mathbf{B})}{\lambda_{\min}(\mathbf{B}^\top \mathbf{B}) (K+1)^2} \norm{\zholder_0 - \zstar}^2
.
\end{align*}
Restarting every $\left\lceil 8\sqrt{\frac{e\lambda_{\max}(\mathbf{B}^\top \mathbf{B})}{\lambda_{\min}(\mathbf{B}^\top \mathbf{B})}}\right\rceil$ iterates for a total of $\log\left( \frac{\norm{\zholder_0 - \zstar}}{\epsilon}\right)$ times, we obtain the final sample complexity of 
\begin{align*}
\cO \left(\sqrt{\frac{\lambda_{\max}(B^\top B)}{\lambda_{\min}(B^\top B)}}\log\left(\frac{1}{\epsilon} \right) \right)
\end{align*}
for the nonstochastic setting.
\end{proof}

\subsection{Stochastic Bilinear Game Case}\label{sec:stoc_bilinear}
For the stochastic AG-OG with restarting for bilinear case, our iteration spells
\begin{subequations}\label{eq:AG-OG-I_stoc}
\begin{numcases}{}
\bz_{k+\frac12}
&=
$
\bz_k - \eta_k \tilde{H}(\bz_{k-\frac12}; \zeta_{k-\frac12})
$,           \label{eq:agog_extra_step1_stoc}   
\\
\zag_{k+1}
&=
$
(1 - \alpha_k)\zag_{k} + \alpha_k \bz_{k+\frac12}
$,
\label{eq:agog_ag_step_stoc}
\\
\bz_{k+1}
&=
$
\bz_k - \eta_k\tilde{H}(\bz_{k+\frac12}; \zeta_{k+\frac12})
$.
\label{eq:agog_extra_step2_stoc} 
\end{numcases}
\end{subequations}
We are able to derive the following theorem.
\begin{theorem}[Convergence of stochastic AG-OG, bilinear case)]\label{theo:main_bilinear_stoc}
When specified to the stochastic bilinear game case, setting the parameters as $\alpha_k = \frac{2}{k+2}$ and $\eta_k = \frac{1}{2L_H}$, the output of update rules~\eqref{eq:AG-OG-I_stoc} satisfies for any $\gamma > 0$,
\begin{align*}
\EE \norm{\zag_K - \zstar}^2
&\leq 
(1 + \gamma) \left(\frac{128L_H^2}{\lambda_{\min}(\mathbf{B}^\top \mathbf{B}) (K+1)^2} \norm{\bz_0 - \zstar}^2 
+ \frac{48 }{\lambda_{\min}(\mathbf{B}^\top \mathbf{B}) (K+1)}
\sigma_H^2\right)
+
\frac{4(1+\frac{1}{\gamma})\sigma_H^2}{3\lambda_{\min}(\mathbf{B}^\top \mathbf{B})K}
.
\end{align*}
Moreover, by operating scheduled restarting technique, it yields a complexity result of
$$
O\left(\sqrt{\frac{\lambda_{\max }\left(\mathbf{B}^{\top} \mathbf{B}\right)}{\lambda_{\min }\left(\mathbf{B B}^{\top}\right)}} \log \left(\frac{\sqrt[4]{\lambda_{\min }\left(\mathbf{B B}^{\top}\right) \lambda_{\max }\left(\mathbf{B}^{\top} \mathbf{B}\right)}}{\sigma_{H}}\right)+\frac{\sigma_{H}^2}{\lambda_{\min }\left(\mathbf{B B}^{\top}\right) \varepsilon^2}\right)
$$
\end{theorem}
We begin by proving the following lemma, which is the stochastic version of Lemma~\ref{lem:deter_bound_bilinear}. It is worth noting that when setting $\sigma_H = 0$ and $\beta = 0$ in Lemma~\ref{lem:stoc_bound_bilinear} below, it reduces to the non-expansiveness of deterministic iterates~\eqref{eq:AG-OG-I}.
Moreover, Lemma~\ref{lem:stoc_bound_bilinear} holds for any monotonic $H(\cdot)$.

\begin{lemma}[Bounded Iterates (Stochastic)]\label{lem:stoc_bound_bilinear}
Following~\eqref{eq:AG-OG-I_stoc}, for any $\beta > 0$, taking $\eta_k = \frac{1}{2L_H\sqrt{(1 + \beta)}}$ at any iterate $k < K$, $\bz_k$ stays within the region defined by the initialization $\bz_0$:
\begin{align*}
\norm{\bz_k - \zstar}^2 
\leq 
\norm{\bz_0 - \zstar}^2
+
\eta_k^2\left(2 + \frac{1}{\beta}\right)K\sigma_H^2,
\end{align*}
where we recall that $\zstar$ denotes the unique solution of Problem~\eqref{eq:Minmax_separable} with $\nabla f, \nabla g = 0$ and $I$ defined in~\eqref{eq:def_H_bilinear}.
\end{lemma}

\begin{proof}[Proof of Lemma~\ref{lem:stoc_bound_bilinear}]
The optimal condition of the problem yields 
$H(\zstar) = 0$ for $\zstar$ being the solution of the VI.
By the monotonicity of $H(\cdot)$, let $\zholder = \bz_{k+\frac12}$ and $\zholder' =  \omega_{\bz} $ in~\eqref{eq:H_prop}, we have that 
\begin{align}\label{eq:H_prop_yields_1_stoc}
\left\langle 
H(\bz_{k+\frac12}) - H(\omega_{\bz}), \bz_{k+\frac12}- \omega_{\bz}
\right\rangle
\geq 0
,\qquad
\forall \omega_{\bz} \in \cZ
.
\end{align}
Let $\foneholder = \zholder_{k+\frac12}$, $\ftwoholder = \zholder_{k+1}$, $\thetaholder = \zholder_k$, $\boneholder = \eta_k \tilde{H} (\zholder_{k-\frac12}; \zeta_{k-\frac12})$, $\btwoholder = \eta_k \tilde{H} (\zholder_{k+\frac12}; \zeta_{k+\frac12})$ and $\zholder = \omega_z$ in Proposition~\ref{prop:PRecursion}, we have
\begin{align}
&
\eta_k \left\langle 
\tilde{H}(\zholder_{k+\frac12}; \zeta_{k+\frac12}),
\zholder_{k+\frac12} - \omega_z
\right\rangle
\notag
\\&\leq 
\frac{\eta_k^2}{2} \norm{\tilde{H}(\bz_{k+\frac12}; \zeta_{k+\frac12}) - \tilde{H}(\bz_{k-\frac12}; \zeta_{k-\frac12})}^2
+
\frac12 \left[ 
\norm{\zholder_k - \omega_z}^2 - \norm{\zholder_{k+1} - \omega_z}^2 - \norm{\zholder_k - \zholder_{k+\frac12}}^2 
\right]
.
\label{eq:ttt_stoc}
\end{align}
Recalling the results in Lemma~\ref{lem:stoc_nablaH} that 
\begin{align*}
\EE \norm{\tilde{H}(\zhp; \zeta_{\ph}) - \tilde{H}(\zhm; \zeta_{\mh})}^2
\leq 
(1 + \beta) L_H^2 \EE\norm{\zhp - \zhm}^2
+
\left(2 + \frac{1}{\beta}\right)\sigma_H^2
.
\end{align*}
Taking expectation over~\eqref{eq:ttt_stoc}, combining it with~\eqref{eq:H_prop_yields_1_stoc} and letting $\omega_{\bz} = \zstar$, we obtain
\begin{align}\label{eq:before}
&0 = \eta_k 
\EE \left\langle 
H (\zstar), \bz_{k+\frac12} - \zstar
\right\rangle \notag
\\&\leq 
\frac{\eta_k^2 (1 + \beta) L_H^2}{2} \EE \norm{\bz_{k+\frac12} - \bz_{k-\frac12}}^2
+
\frac12 \EE \left[ 
\norm{\zholder_k - \zstar}^2 - \norm{\zholder_{k+1} - \zstar}^2 - \norm{\zholder_k - \zholder_{k+\frac12}}^2 
\right]
+
\frac{\eta_k^2\left(2 + \frac{1}{\beta}\right)}{2} \sigma_H^2
.
\end{align}
Next, we move on to estimate $\norm{\zholder_{k+\frac12} - \zholder_{k-\frac12}}^2$. As we know that via Young’s and Cauchy-Schwarz’s inequalities and the update rules~\eqref{eq:agog_extra_step1} and~\eqref{eq:agog_extra_step2}, for all $k \geq 1$
\begin{align*}
\norm{\bz_{k + \frac12} - \bz_{k - \frac12}}^2
&\leq
2\norm{\bz_{k + \frac12} - \bz_k}^2
+
2\norm{\bz_{k} - \bz_{k - \frac12}}^2
\\&\leq
2 \norm{\bz_{k + \frac12} - \bz_{k}}^2
+
2 \eta_{k-1}^2 L_H^2\norm{\bz_{k - \frac12} - \bz_{k - \frac32}}^2
.
\end{align*}
Multiplying both sides by 2 and moving one term to the right hand gives for all $k\ge 1$
\begin{align*}
\norm{\bz_{k + \frac12} - \bz_{k - \frac12}}^2
&\leq 
4 \norm{\bz_{k + \frac12} - \bz_{k}}^2
+
4 \eta_{k-1}^2 L_H^2\norm{\bz_{k - \frac12} - \bz_{k - \frac32}}^2
-
\norm{\bz_{k + \frac12} - \bz_{k - \frac12}}^2
.
\end{align*}
Bringing this into~\eqref{eq:before} and noting that $\eta_{k-1} \leq \frac{1}{2L_H}$ as well as $\eta_k \leq \frac{1}{2L_H}$, we have
\begin{align*}
0
&\leq 
\frac{\eta_k^2 (1 + \beta) L_H^2}{2} \EE \norm{\bz_{k+\frac12} - \bz_{k-\frac12}}^2
+
\frac12 \EE \left[ 
\norm{\zholder_k - \zstar}^2 - \norm{\zholder_{k+1} - \zstar}^2 - \norm{\zholder_k - \zholder_{k+\frac12}}^2    
\right]
+
\frac{\eta_k^2 \left(2 + \frac{1}{\beta}\right)}{2}\sigma_H^2
\\&\leq 
\frac12 \EE \left[ 
\norm{\bz_k - \zstar}^2 - \norm{\bz_{k+1} -\zstar}^2    
\right] 
+
\frac{\eta_k^2(1 + \beta) L_H^2}{2}
\EE \left[ 
\norm{\bz_{k-\frac12} - \bz_{k - \frac32}}^2 - \norm{\bz_{k+\frac12} - \bz_{k - \frac12}}^2    
\right]
\\&~\quad -
\left(
\frac12 - 
2 \eta_k^2 (1 + \beta) L_H^2
\right) \EE \norm{\zholder_k - \zholder_{k+\frac12}}^2 
+
\frac{\eta_k^2\left(2 + \frac{1}{\beta}\right)}{2}\sigma_H^2
.
\end{align*}
Taking $\eta_k$ satisfying $\eta_k \leq \frac{1}{L_H\sqrt{4(1 + \beta)}}$ and by rearraing the above inequality, we have
\begin{align*}
0   &\leq
\frac12 \EE \left[ 
\norm{\bz_k - \zstar}^2 - \norm{\bz_{k+1} - \zstar}^2    
\right] 
+
\frac18
\EE \left[ 
\norm{\bz_{k-\frac12} - \bz_{k - \frac32}}^2 - \norm{\bz_{k+\frac12} - \bz_{k - \frac12}}^2    
\right]
\\&~\quad -
\left(
\frac12 - 
\frac12 
\right) \EE \norm{\zholder_k - \zholder_{k+\frac12}}^2 
+
\frac{\eta_k^2\left(2 + \frac{1}{\beta}\right)}{2}\sigma_H^2
\\&\leq 
\frac12 \EE \left[ 
\norm{\bz_k - \zstar}^2 + \frac14 \norm{\bz_{k-\frac12} - \bz_{k - \frac32}}^2 
-
\norm{\bz_{k+1} - \zstar}^2 - \frac14 \norm{\bz_{k+\frac12} - \bz_{k - \frac12}}^2
\right]
+
\frac{\eta_k^2\left(2 + \frac{1}{\beta}\right)}{2}\sigma_H^2
.
\end{align*}
Rearranging the above inequality and we conclude that 
\begin{align*}
\EE \left[\norm{\zholder_{k+1} - \zstar}^2
+ 
\frac14 \norm{\zholder_{k+\frac12} - \zholder_{k-\frac12}}^2 
\right]
\leq 
\EE \left[
\norm{\zholder_k - \zstar}^2 + \frac14\norm{\zholder_{k-\frac12} - \zholder_{k-\frac32}}^2
\right]
+
\eta_k^2\left(2 + \frac{1}{\beta}\right)\sigma_H^2
.
\end{align*}
Telescoping over $k = 0, 1, \ldots, K-1$ and noting that $\zholder_{-\frac12} = \zholder_{-\frac32} = \zholder_0$ and $\eta_k$ is taken as a constant for the bilinear case, we have 
\begin{align*}
\norm{\zholder_K - \zstar}^2 \leq \norm{\zholder_K - \zstar}^2
+
\frac14 \norm{\zholder_{K-\frac12} - \zholder_{K-\frac32}}^2 
\leq 
\norm{\zholder_0 - \zstar}^2
+
\eta_k^2\left(2 + \frac{1}{\beta}\right)K\sigma_H^2
,
\end{align*}
which concludes our proof of Lemma~\ref{lem:stoc_bound_bilinear}.
\end{proof}
Recalling that we take $\alpha_k = \frac{2}{k+2}$ in~\eqref{eq:agog_ag_step_stoc} of~\eqref{eq:AG-OG-I_stoc}. Thus, we obtain the following
\begin{align*}
\zag_{k+1} &= \frac{k}{k+2} \zag_k + \frac{2}{k+2} \zholder_{k+\frac12}
.
\end{align*}
Subtracting both sides by $\zstar$ and multiplying both sides by $(k+1)(k+2)$, we have
\begin{align*}
(k+1)(k+2) \left(\zag_{k+1} - \zstar\right) 
&=
k(k+1) \left(\zag_k - \zstar\right)
+
2(k+1) \left(\zholder_{k+\frac12} - \zstar\right)
.
\end{align*}
Telescoping over $k = 0, \ldots, K-1$ and we conclude that
\begin{align}\label{eq:recurs_stoc}
K(K+1) \left(\zag_{K} - \zstar \right)
&= 
2 \sum_{k=0}^{K-1} (k+1) \left(\zholder_{k+\frac12} - \zstar \right)
.
\end{align}
Moreover, according to the update rule~\eqref{eq:agog_extra_step2_stoc}, we have that
\begin{align}\label{eq:recurs_2_stoc}
K z_K - \sum_{k=0}^{K-1} z_k 
&=
\sum_{k=0}^{K-1}(k+1)\zholder_{k+1} - (k+1) \zholder_k 
=
\sum_{k=0}^{K-1}\eta_k (k+1) \tilde{H}(\zholder_{k+\frac12}, \zeta_{k+\frac12})\notag
\\&=
\sum_{k=0}^{K-1}\eta_k (k+1) 
\left[H(\bz_{k+\frac12}) + \Delta_h^{k+\frac12}\right]
.
\end{align}
Recalling that $\eta_k = \frac{1}{L_H\sqrt{4(1+\beta)}}$, combining~\eqref{eq:recurs_2_stoc} with~\eqref{eq:recurs_stoc} and taking the squared norm on both sides, we conclude that
\begin{align}
&
\norm{K(K+1) \left(\zag_{K} - \zstar\right)}^2
\notag
\\&= 
\norm{2 \sum_{k=0}^{K-1} (k+1) \left(\zholder_{k+\frac12} - \zstar\right)}^2
\leq 
\frac{1}{\lambda_{\min}(\mathbf{B}^\top \mathbf{B})}\norm{2 \sum_{k=0}^{K-1} (k+1) H(\zholder_{k+\frac12})}^2\notag
\\&=
\frac{16(1 + \beta)(1 + \gamma)L_H^2}{\lambda_{\min}(\mathbf{B}^\top \mathbf{B})} 
\norm{K\bz_K - \sum_{k=0}^{K-1}\bz_k}^2
+
\frac{4(1 + \frac{1}{\gamma})}{\lambda_{\min}(\mathbf{B}^\top \mathbf{B})}
\norm{\sum_{k=0}^{K-1}(k+1) \Delta_h^{k+\frac12}}^2
\notag
\\&\leq 
\frac{16(1+\beta)(1+\gamma)L_H^2}{\lambda_{\min}(\mathbf{B}^\top \mathbf{B})} 
K\cdot 
\sum_{k=0}^{K-1} 
\left[
2\norm{\bz_K - \zstar}^2
+
2\norm{\bz_k - \zstar}^2
\right]
+
\frac{4(1 + \frac{1}{\gamma})}{\lambda_{\min}(\mathbf{B}^\top \mathbf{B})}
\norm{\sum_{k=0}^{K-1}(k+1) \Delta_h^{k+\frac12}}^2
.
\label{eq:OGDA_proof_final}
\end{align}
Dividing both sides of~\eqref{eq:OGDA_proof_final} by $K^2(K+1)^2$ and taking expectation, we have
\begin{align*}
\EE \norm{\zag_K - \zstar}^2 
&\leq 
\frac{16(1+\beta)(1+\gamma)L_H^2}{\lambda_{\min}(\mathbf{B}^\top \mathbf{B}) K(K+1)^2} 
\sum_{k=0}^{K-1} 
\EE \left[
2\norm{\bz_K - \zstar}^2
+
2\norm{\bz_k - \zstar}^2
\right]
\\&~\quad
+
\frac{4(1+\frac{1}{\gamma})}{\lambda_{\min}(\mathbf{B}^\top \mathbf{B})K^2 (K+1)^2}\sum_{k=0}^{K-1} (k+1)^2\EE \norm{\Delta_h^{k+\frac12}}^2
\\&\leq 
\frac{64(1+\beta)(1+\gamma)L_H^2}{\lambda_{\min}(\mathbf{B}^\top \mathbf{B}) (K+1)^2} \left[ \norm{\bz_0 - \zstar}^2 
+ \eta_k^2 (2+\frac{1}{\beta}) K \sigma_H^2
\right]
+
\frac{4(1+\frac{1}{\gamma})\sigma_H^2}{3\lambda_{\min}(\mathbf{B}^\top \mathbf{B})K}
\\&=
\frac{64(1+\beta)(1+\gamma)L_H^2}{\lambda_{\min}(\mathbf{B}^\top \mathbf{B}) (K+1)^2} \norm{\bz_0 - \zstar}^2 
+ \frac{16(1+\gamma) (2+\frac{1}{\beta})}{\lambda_{\min}(\mathbf{B}^\top \mathbf{B}) (K+1)}
\sigma_H^2
+
\frac{4(1+\frac{1}{\gamma})\sigma_H^2}{3\lambda_{\min}(\mathbf{B}^\top \mathbf{B})K}
,
\end{align*}
Minimize over $\beta$, we have
$\beta = \frac{\sigma_H\sqrt{K+1}}{2L_H\norm{\bz_0 - \zstar}}$ and the first two terms become
\begin{align*}
\frac{64(1+\gamma)L_H^2}{\lambda_{\min}(\mathbf{B}^\top \mathbf{B}) (K+1)^2} \norm{\bz_0 - \zstar}^2 
+ \frac{32(1+\gamma)} {\lambda_{\min}(\mathbf{B}^\top \mathbf{B}) (K+1)}
\sigma_H^2
+
\frac{32 (1+\gamma)L_H \sigma_H}{\lambda_{\min}(\mathbf{B}^\top \mathbf{B}) (K+1)^{3/2}} \norm{\bz_0 - \zstar} 
\end{align*}

If we take $\beta = 1$, the above reduces to
\begin{align*}
\EE \norm{\zag_K - \zstar}^2
&\leq 
\frac{128(1+\gamma)L_H^2}{\lambda_{\min}(\mathbf{B}^\top \mathbf{B}) (K+1)^2} \norm{\bz_0 - \zstar}^2 
+ \frac{48(1+\gamma) }{\lambda_{\min}(\mathbf{B}^\top \mathbf{B}) (K+1)}
\sigma_H^2
+
\frac{4(1+\frac{1}{\gamma})\sigma_H^2}{3\lambda_{\min}(\mathbf{B}^\top \mathbf{B})K}
\\&=
(1 + \gamma) \left(\frac{128L_H^2}{\lambda_{\min}(\mathbf{B}^\top \mathbf{B}) (K+1)^2} \norm{\bz_0 - \zstar}^2 
+ \frac{48 }{\lambda_{\min}(\mathbf{B}^\top \mathbf{B}) (K+1)}
\sigma_H^2\right)
+
\frac{4(1+\frac{1}{\gamma})\sigma_H^2}{3\lambda_{\min}(\mathbf{B}^\top \mathbf{B})K}
\end{align*}
Further minimizing over $\gamma$, we conclude that 
so
\begin{align*}
\sqrt{\EE\norm{\zag_K - \zstar}^2}
&\leq 
\sqrt{
\frac{128L_H^2}{\lambda_{\min}(\mathbf{B}^\top \mathbf{B}) (K+1)^2} \norm{\bz_0 - \zstar}^2 
+ \frac{48 }{\lambda_{\min}(\mathbf{B}^\top \mathbf{B}) (K+1)}
\sigma_H^2
}
+
\sqrt{
\frac{4\sigma_H^2}{3\lambda_{\min}(\mathbf{B}^\top \mathbf{B})K}
}
\\&\leq 
\sqrt{
\frac{128L_H^2}{\lambda_{\min}(\mathbf{B}^\top \mathbf{B}) (K+1)^2} \norm{\bz_0 - \zstar}^2
}
+
\sqrt{
\frac{48\sigma_H^2}{\lambda_{\min}(\mathbf{B}^\top \mathbf{B}) (K+1)}
}
+
\sqrt{
\frac{4\sigma_H^2}{3\lambda_{\min}(\mathbf{B}^\top \mathbf{B})K}
}
\\&\leq 
\frac{1}{\sqrt{\lambda_{\min}(\mathbf{B}^\top \mathbf{B})}}\left(
\frac{8\sqrt{2}L_H\norm{\bz_0 - \zstar}}{K+1}
+
\frac{8.083\sigma_H}{\sqrt{K}}
\right)
\end{align*}

By operating restarting techniques the same way as in the explanation of Corollary 3.2 in~\citet{du2022optimal}, we conclude Theorem~\ref{theo:main_bilinear_stoc}.

\section{Proof of Auxiliary Lemmas}

\subsection{Proof of Lemma~\ref{lemm_QuanBdd}}\label{sec_proof,lemm_QuanBdd}
\begin{proof}[Proof of Lemma~\ref{lemm_QuanBdd}]

Since $F(\zholder)$ is $L$-smooth and $\mu$-strongly convex.
For the rest of this proof, we observe that the saddle definition of $\zstar$ satisfies the first-order stationary condition:
\begin{align}\label{FOSC}
\begin{split}
\nabla F(\zstar) 
+  H(\zstar)
=
0
.
\end{split}
\end{align}
Furthermore,  we have
$$\begin{aligned}
&
F(\zholder) - F(\zstar) 
+ \left\langle H(\zstar), \zholder - \zstar \right\rangle 
\\&\ge
\left\langle
\nabla F(\zstar), \zholder - \zstar
\right\rangle
+
\frac{\mu}{2}\left\| \zholder - \zstar\right\|^2
+
\left\langle
 H(\zstar), \zholder - \zstar
\right\rangle
\\&=
\left\langle
\nabla F(\zstar) + H(\zstar), \zholder - \zstar
\right\rangle
+
\frac{\mu}{2}\left\| \zholder - \zstar\right\|^2
=
\frac{\mu}{2}\left\| \zholder - \zstar\right\|^2
,
\end{aligned}$$
where in both of the two displays, the inequality holds due to the $\mu$-strong convexity of $F$, and the equality holds due to the first-order stationary condition \eqref{FOSC}.
This completes the proof.
\end{proof}

\subsection{Proof of Lemma~\ref{lem:deter_bound}}\label{sec:proof_lem_deter_bound}
\begin{proof}[Proof of Lemma~\ref{lem:deter_bound}]
Following~\eqref{eq:recurs_solve}, we let $\omega_{\bz} = \zstar$. Due to the non-negativity of $V(\cdot, \zstar)$, we can eliminate the $V$ terms and have:
\begin{align*}
&
\left(2L + \sqrt{\frac{2}{c}}L_H(K+1)\right)
\|\bz_{K} - \zstar\|^2
\leq 
\left(2L + \sqrt{\frac{2}{c}} L_H \right) \norm{\bz_0 - \zstar}^2 
+
\sqrt{\frac{2}{c}} L_H \sum_{k = 0}^{K-1} \left\|\bz_k - \zstar \right\|^2
.
\end{align*}
We adopt a "bootstrapping" argument.
We define $M_K = \max_{0 \leq k \leq K-1} \norm{\bz_k - \zstar}^2$ and taking a maximum on each term on the right hand side of the above inequality, we conclude that
\begin{align*}
\left(2L + \sqrt{\frac{2}{c}}L_H(K+1)\right)
\|\bz_{K} - \zstar\|^2
&\leq 
\left(2L + \sqrt{\frac{2}{c}} L_H \right) M_{K-1}
+
\sqrt{\frac{2}{c}} L_H \sum_{k = 0}^{K-1} M_{K-1}
\\&=
\left(2L + \sqrt{\frac{2}{c}}L_H(K+1)\right)
M_{K-1}
.
\end{align*}
Thus, we know that $\norm{\bz_K - \zstar}^2 \leq M_{K-1}$ and hence $M_{K} = M_{K-1}$ always holds.
That yields $M_K = M_0$, and we conclude the proof of Lemma~\ref{lem:deter_bound}.
\end{proof}

\subsection{Proof of Lemma~\ref{lem:stoc_boundedness}}\label{sec:proof_lem_stoc_boundedness}
\begin{proof}[Proof of Lemma~\ref{lem:stoc_boundedness}]
Starting from~\eqref{eq:stoc_control_V} that
\begin{align*}
&\left[(K + 1)^2 - 1 \right] \EE V(\zag_{K}, \zstar)
+
\frac{K+1}{\eta_{K-1}}
\EE \norm{\bz_{K} - \zstar}^2\notag
\\&\leq 
\frac{2}{\eta_0} \EE \norm{\bz_0 - \zstar}^2 
+
\frac{2}{\sqrt{c}} L_H \sum_{k = 1}^{K-1} \EE \left\|\bz_k - \zstar \right\|^2
+
\sum_{k=0}^{K-1} (k+2) \eta_k \sigma^2
-
\sum_{k = 0}^{K-1} \EE V(\zag_{k+1}, \zstar)
.
\end{align*}
We first omit the $V(\cdot, \cdot)$ terms and have
\begin{align}
&
\frac{K+1}{\eta_{K-1}}
\EE \norm{\bz_{K} - \zstar}^2
\leq 
\frac{2}{\eta_0} \EE \norm{\bz_0 - \zstar}^2 
+
\frac{2}{\sqrt{c}} L_H \sum_{k = 1}^{K-1} \norm{\bz_k - \zstar}^2
+
\sum_{k=0}^{K-1} (k+2) \eta_k \sigma^2
.\label{eq:stoc_bound_midstep}
\end{align}
Rewrite $\norm{\bz_K - \zstar}^2$ as the difference between two summations, we obtain:
\begin{align*}
&
\frac{K+1}{\eta_{K-1}}
\left(\sum_{k=1}^K - \sum_{k=1}^{K-1}\right)
\EE \norm{\bz_k - \omega_{\bz}}^2
\leq 
\frac{2}{\eta_0} \EE \norm{\bz_0 - \zstar}^2 
+
\frac{2}{\sqrt{c}} L_H \sum_{k = 1}^{K-1} \EE \norm{\bz_k - \zstar}^2
+
\sum_{k=0}^{K-1} (k+2) \eta_k \sigma^2
.
\end{align*}
Rearranging the terms and by the first condition~\ref{item:a} that $\frac{k+2}{\eta_k} - \frac{k+1}{\eta_{k-1}} = \frac{2}{\sqrt{c}}L_H$, we have:
\begin{align*}
&
\frac{K+1}{\eta_{K-1}}
\sum_{k=1}^K \EE \norm{\bz_k - \zstar}^2
\leq 
\frac{2}{\eta_0} \EE \norm{\bz_0 - \zstar}^2 
+ 
\frac{K+2}{\eta_K}\sum_{k=1}^{K-1}
\EE \norm{\bz_k - \omega_{\bz}}^2
+
\sum_{k=0}^{K-1} (k+2) \eta_k \sigma^2
.
\end{align*}
To construct a valid iterative rule, we divide both sides of the above inequality with $\frac{(K+1)(K+2)}{\eta_{K-1} \eta_K}$ and obtain the following:
\begin{align*}
&
\frac{\eta_K}{K+2}
\sum_{k=1}^K \EE \norm{\bz_k - \zstar}^2
\leq 
\frac{\eta_{K-1}}{K+1}\sum_{k=1}^{K-1}
\EE \norm{\bz_k - \omega_{\bz}}^2
+
\frac{\eta_{K-1}\eta_{K}}{(K+1)(K+2)}
\left[
\frac{2}{\eta_0} \EE \norm{\bz_0 - \zstar}^2 
+
\sum_{k=0}^{K-1} (k+2) \eta_k \sigma^2
\right]
.
\end{align*}
Here we slightly abuse the notations and use $K$ to denote an arbitrary iteration during the process of the algorithm and use $\cK$ to denote the fixed total number of iterates.
Thus, $
\sum_{k=0}^{K-1}(k+2)\eta_k\sigma^2 \leq \sum_{k=0}^{\cK -1}(k+2)\eta_k\sigma^2
$ is an upper bound that does not change with the choice of $K$. It follows that:
\begin{align*}
\frac{\eta_K}{K+2}
\sum_{k=1}^K \EE \norm{\bz_k - \zstar}^2
&\leq 
\frac{\eta_{K-1}}{K+1}\sum_{k=1}^{K-1}
\EE \norm{\bz_k - \omega_{\bz}}^2
+
\frac{\sqrt{c}}{2L_H}
\left[
\frac{\eta_{K-1}}{K+1} - \frac{\eta_K}{K+2}
\right]\left[
\frac{2}{\eta_0} \EE \norm{\bz_0 - \zstar}^2 
+
\sum_{k=0}^{\cK -1} (k+2) \eta_k \sigma^2
\right]
\\&\leq 
\frac{\sqrt{c}}{2L_H}
\left[
\frac{\eta_0}{2} - \frac{\eta_K}{K+2}
\right]\left[
\frac{2}{\eta_0} \EE \norm{\bz_0 - \zstar}^2 
+
\sum_{k=0}^{\cK-1} (k+2) \eta_k \sigma^2
\right]
.
\end{align*}
Dividing both sides by $\frac{\eta_K}{K+2}$, the result follows:
\begin{align*}
\sum_{k=1}^K \EE \norm{\bz_k - \zstar}^2
&\leq 
\frac{\sqrt{c}}{2L_H}
\left[
\frac{\eta_0(K+2)}{2\eta_K} - 1
\right]\left[
\frac{2}{\eta_0} \EE \norm{\bz_0 - \zstar}^2 
+
\sum_{k=0}^{\cK-1} (k+2) \eta_k \sigma^2
\right]
.
\end{align*}
Bringing this into Eq.~\eqref{eq:stoc_bound_midstep}, we conclude that
\begin{align*}
&
\frac{K+1}{\eta_{K-1}}
\EE\norm{\bz_{K} - \zstar}^2
\leq 
\frac{\eta_0(K+1)}{2\eta_{K-1}}\left[
\frac{2}{\eta_0} \EE\norm{\bz_0 - \zstar}^2 
+
\sum_{k=0}^{\cK -1} (k+2) \eta_k \sigma^2\right]
.
\end{align*}
Dividing both sides by $\frac{K+1}{\eta_{K-1}}$ and we have:
\begin{align*}
&
\EE\norm{\bz_{K} - \zstar}^2
\leq 
\frac{\eta_0}{2}\left[
\frac{2}{\eta_0} \EE \norm{\bz_0 - \zstar}^2 
+
\sum_{k=0}^{\cK-1} (k+2) \eta_k \sigma^2\right]
.
\end{align*}
Now we change back using the notation $K$ to denote the total iterates and $k$ is the iterates indexes, we have
\begin{align*}
&
\EE \norm{\bz_{k} - \zstar}^2
\leq 
\EE \norm{\bz_0 - \zstar}^2 
+
\frac{\eta_0}{2} \sum_{k=0}^{K-1} (k+2) \eta_k \sigma^2
,
\end{align*}
which concludes the proof of~\ref{item:a} of Lemma~\ref{lem:stoc_boundedness}.
Additionally, if $\eta_k \leq \frac{k+2}{D}$ for some quantity $D$, we have
\begin{align*}
\sum_{k=0}^{K-1} (k+2) \eta_k
\leq 
\sum_{k=0}^{K-1} \frac{(k+2)^2}{D}
\leq 
\frac{(K + 1)(K + 2)(2K + 3)}{6D}
.
\end{align*}
We use $A(K) = \sqrt{(K+1)(K+2)(2K+3)/6}$ and noting that $\eta_0 \leq \frac{2}{D}$, we have
\begin{align*}
&
\EE \norm{\bz_{k} - \zstar}^2
\leq 
\EE \norm{\bz_0 - \zstar}^2 
+
\frac{A(K)^2\sigma^2}{D^2}
,
\end{align*}
which concludes our proof of~\ref{item:b}.
And~\ref{item:c} follows by straightforward calculations.
\end{proof}

\subsection{Proof of Lemma~\ref{lem:basic}}\label{sec:proof_lem:basic}
\begin{proof}[Proof of Lemma~\ref{lem:basic}]
Recalling that $F$ is $L$-smooth.
To upper-bound the difference in pointwise primal-dual gap between iterates, we first estimate the difference in function values of $f$ via gradients at the extrapolation point $\zmd_{k}$.
For any given $\bu \in \cZ$, the convexity and $L$-smoothness of $F(\cdot)$ implies that
\begin{align*}
F(\zag_{k+1}) - F(\bu) 
&= 
F(\zag_{k+1}) - F(\zmd_{k}) 
-
\left(F(\bu) - F(\zmd_{k})\right)
\\&\leq 
\left\langle \nabla F(\zmd_{k}), \zag_{k+1} - \zmd_{k}\right\rangle 
+
\frac{L}{2} \left\|\zag_{k+1} - \zmd_{k} \right\|^2
-
\left\langle 
\nabla F(\zmd_{k}), \bu - \zmd_{k}
\right\rangle 
.
\end{align*}
Taking $\bu = \omega_{\bz}$ and $\bu = \zag_k$ respectively, we conclude that
\begin{align}
F(\zag_{k+1}) - F(\omega_{\bz}) 
&\leq 
\left\langle \nabla F(\zmd_{k}), \zag_{k+1} - \zmd_{k}\right\rangle 
+
\frac{L}{2} \left\|\zag_{k+1} - \zmd_{k} \right\|^2
-
\left\langle 
\nabla F(\zmd_{k}), \omega_{\bz} - \zmd_{k}
\right\rangle
\label{eq:first}
,\\
F(\zag_{k+1}) - F(\zag_{k}) 
&\leq 
\left\langle \nabla F(\zmd_{k}), \zag_{k+1} - \zmd_{k}\right\rangle 
+
\frac{L}{2} \left\|\zag_{k+1} - \zmd_{k} \right\|^2
-
\left\langle \nabla F(\zmd_{k}), \zag_{k} - \zmd_{k}\right\rangle 
.
\label{eq:second}
\end{align}
Multiplying~\eqref{eq:first} by $\alpha_k$ and~\eqref{eq:second} by $(1 - \alpha_k)$ and adding them up, we have
\begin{align}
F(\zag_{k+1}) - \alpha_k F(\omega_{\bz}) - (1 - \alpha_k) F(\zag_{k}) 
&\leq 
\left\langle 
\nabla F(\zmd_k), 
\zag_{k+1} - (1 - \alpha_k) \zag_k 
-
\alpha_k\omega_{\bz}
\right\rangle 
+
\frac{L}{2} \norm{\zag_{k+1} - \zmd_k}^2 \nonumber
\\&=
\underbrace{\alpha_k
\left\langle 
\nabla F(\zmd_{k}), \bz_{k + \frac12} - \omega_{\bz}
\right\rangle 
}_{\mbox{I(a)}}
+
\underbrace{
\frac{L \alpha_k^2}{2} \left\|\bz_{k + \frac12} - \bz_{k} \right\|^2
}_{\mbox{II}}
,
\label{eq:diff}
\end{align}
where by substracting Line~\ref{line:mdupd} from Line~\ref{line:agupd} of Algorithm~\ref{alg:AG_OG} and by Line~\ref{line:agupd} itself, the last equality of~\eqref{eq:diff} follows.

Recalling that $\zag_k$ corresponds to regular iterates and $\zmd_k$ corresponds to the extrapolated iterates of Nesterov's acceleration scheme.
The squared error term $\mbox{II}$ in~\eqref{eq:diff} is brought by gradient evaluated at the extrapolated point instead of the regular point.
Note that if we do an implicit version of Nesterov such that $\zmd_{k - 1} = \zag_k$, this squared term goes to zero, and the convergence analysis would be the same as in OGDA. This could potentially result in a new implicit algorithm with better convergence guarantee.

On the other hand, for the coupling term of the updates, we have
\begin{align}
&
\left\langle 
\Matrix (\omega_{\bz}), \zag_{k+1} - \omega_{\bz}
\right\rangle 
-
(1 - \alpha_k)
\left\langle 
\Matrix (\omega_{\bz}), \zag_{k} - \omega_{\bz}
\right\rangle
=
\alpha_k
\left\langle 
H(\omega_{\bz}), \bz_{k+\frac12} - \omega_{\bz}
\right\rangle 
\leq 
\underbrace{\alpha_k
\left\langle 
\Matrix (\bz_{k + \frac12}), \bz_{k + \frac12} - \omega_{\bz}
\right\rangle
}_{\mbox{I(b)}}
,
\label{eq:diff_H}
\end{align}
where the last equality comes from the monotonicity property of $\Matrix(\cdot)$ that
\begin{align*}
\left\langle 
H(\bz_{k+\frac12}) - H(\omega_{\bz}), 
\bz_{k+\frac12} - \omega_{\bz}
\right\rangle \geq 0
.
\end{align*}
Summing both sides of Eq.~\eqref{eq:diff} and Eq.~\eqref{eq:diff_H} we obtain the following:
\begin{align*}
&
F(\zag_{k+1}) - \alpha_k F(\omega_{\bz}) - (1 - \alpha_k) F(\zag_{k}) 
+
\left\langle 
\Matrix(\omega_{\bz}), \zag_{k+1} - \omega_{\bz}
\right\rangle 
-
(1 - \alpha_k)\left\langle 
\Matrix(\omega_{\bz}), \zag_{k} - w_z
\right\rangle 
\\&\leq 
\underbrace{\alpha_k
\left\langle 
\nabla F(\zmd_{k}) + \Matrix(\bz_{k + \frac12}),
\bz_{k + \frac12} - \omega_{\bz}
\right\rangle 
}_{\mbox{I}}
+
\underbrace{\frac{L \alpha_k^2}{2} 
\left\| \bz_{k + \frac12} - \bz_{k} \right\|^2
}_{\mbox{II}}
,
\end{align*}
where $\mbox{I}$ is the summation of $\mbox{I(a)}$ and $\mbox{I(b)}$.
This concludes our proof of Lemma~\ref{lem:basic} by bringing in the definitions of $V(\zag_{k+1}, \zstar)$ and $V(\zag_k, \zstar)$.
\end{proof}

\subsection{Proof of Lemma~\ref{lem:OGDA_recurs}}\label{sec:proof_lem:OGDA_recurs}
\begin{proof}[Proof of Lemma~\ref{lem:OGDA_recurs}]
We focus on $k=1,\dots,K-1$ since the $k=0$ case holds automatically.
By Young's inequality and Cauchy-Schwarz inequality, we have that
\begin{equation}\label{recursion_init}
\begin{aligned}
&
\left\|\bz_{k+\frac12} - \bz_{k-\frac12}\right\|^2
\leq 
2 \left\|\bz_{k+\frac12} - \bz_k \right\|^2
+
2 \left\|\bz_k - \bz_{k-\frac12} \right\|^2
\\&\overset{(a)}{\leq}
2 \left\|\bz_{k+\frac12} - \bz_k \right\|^2
+
2 \eta_{k - 1}^2 L_H^2\left\|\bz_{k-\frac12} - \bz_{k - \frac32} \right\|^2
\overset{(b)}{\leq} 
2 \left\|\bz_{k+\frac12} - \bz_k \right\|^2
+
c\left\|\bz_{k-\frac12} - \bz_{k - \frac32} \right\|^2
,
\end{aligned}\end{equation}
where $(a)$ is due to Lines~\ref{line:halfupd} and~\ref{line:kupd} of Algorithm~\ref{alg:AG_OG} and the definition of $L_H$, and $(b)$ is due to the condition in Lemma~\ref{lem:OGDA_recurs} that $\eta_k L_H \leq \sqrt{\frac{c}{2}}$.
Recursively applying the above gives~\eqref{eq:OGDA_recurs} which is repeated as:
\begin{equation}\tag{\ref{eq:OGDA_recurs}}
\left\|\bz_{k+\frac12} - \bz_{k-\frac12}\right\|^2
\le
2c^k\sum_{\ell=0}^k c^{-\ell}\left\|\bz_{\ell+\frac12} - \bz_{\ell}\right\|^2
.
\end{equation}
Indeed, from \eqref{recursion_init}
$$\begin{aligned}
c^{-k}\left\|\bz_{k+\frac12} - \bz_{k-\frac12}\right\|^2
-
c^{-(k-1)}\left\|\bz_{k-\frac12} - \bz_{k - \frac32} \right\|^2
&\le
2c^{-k} \left\|\bz_{k+\frac12} - \bz_k \right\|^2
,
\end{aligned}$$
so telescoping over $k = 1, \ldots, K$ gives
$$\begin{aligned}
c^{-K}\left\|\bz_{K+\frac12} - \bz_{K-\frac12}\right\|^2
-
\left\|\bz_{\frac12} - \bz_{-\frac12} \right\|^2
&\le
2\sum_{k=1}^K c^{-k} \left\|\bz_{k+\frac12} - \bz_k \right\|^2
,
\end{aligned}$$
which simply reduces to (due to $\bz_{0} = \bz_{-\frac12}$)
$$\begin{aligned}
c^{-K}\left\|\bz_{K+\frac12} - \bz_{K-\frac12}\right\|^2
&\le
2\sum_{k=1}^K c^{-k} \left\|\bz_{k+\frac12} - \bz_k \right\|^2
+
\left\|\bz_{\frac12} - \bz_0 \right\|^2
\le
2\sum_{k=0}^K c^{-k} \left\|\bz_{k+\frac12} - \bz_k \right\|^2
.
\end{aligned}$$
This gives~\eqref{eq:OGDA_recurs} and the entire Lemma~\ref{lem:OGDA_recurs}.
\end{proof}

\subsection{Proof of Lemma~\ref{lem:stoc_nablaH}}\label{sec:proof_lem:stoc_nablaH}
\begin{proof}[Proof of Lemma~\ref{lem:stoc_nablaH}]
We recall that we denote
\begin{align*}
\Delta_h^{\ph}
=
\tilde{H}(\zhp; \zeta_{\ph})
-
H(\zhp)
,\qquad 
\Delta_f^{k}
=
\nabla \tilde{F}(\zmd_k; \xi_{k})
-
\nabla F(\zmd_k)
.
\end{align*}
Then, we have
\begin{align*}
&
\EE\norm{\tilde{H}(\zhp; \zeta_{\ph}) - \tilde{H}(\zhm; \zeta_{\mh})}^2
=
\EE \norm{H(\bz_{k+\frac12}) - H(\bz_{k-\frac12}) + \Delta_h^{k+\frac12} - \Delta_h^{k - \frac12}}^2
.
\end{align*}
By first taking expectation over $\zeta_{k+\frac12}$ condition on $\bz_{k+\frac12}$ given, we have
\begin{align*}
\text{LHS}&\leq 
\EE \norm{H(\bz_{k+\frac12}) - H(\bz_{k-\frac12}) - \Delta_h^{k-\frac12}}^2
+
\EE \norm{\Delta_h^{k+\frac12}}^2 \notag
\\&\leq
(1+\beta)\EE \norm{H(\bz_{k+\frac12}) - H(\bz_{k-\frac12})}^2 + (1+\frac{1}{\beta})\EE\norm{\Delta_h^{k-\frac12}}^2
+
\EE \norm{\Delta_h^{k+\frac12}}^2  \notag
\\&\leq 
(1+\beta)L_H^2\EE \norm{\bz_{k+\frac12} - \bz_{k-\frac12}}^2 + (1+\frac{1}{\beta})\EE\norm{\Delta_h^{k-\frac12}}^2
+
\EE \norm{\Delta_h^{k+\frac12}}^2
.
\end{align*}
Recalling that by Assumption~\ref{assu:bounded_variance}, $\EE \norm{\Delta_h^{k+\frac12}}^2 \leq \sigma_H^2$ and $\EE \norm{\Delta_h^{k-\frac12}}^2 \leq \sigma_H^2$, we conclude our proof of Lemma~\ref{lem:stoc_nablaH}.
\end{proof}

\subsection{Proof of Lemma~\ref{lem:ineq_stoc}}\label{sec:proof_lem:ineq_stoc}
\begin{proof}[Proof of Lemma~\ref{lem:ineq_stoc}]
By inequality~\eqref{eq:stoc_starting_from}, we have
\begin{align*}
&\EE V(\zag_{k+1}, \omega_{\bz})
-
(1 - \alpha_k) \EE V(\zag_{k}, \omega_{\bz})\notag
\\&\leq 
\frac{\alpha_k\eta_k}{2} \left[2L_H^2 \EE\norm{\bz_{k+\frac12}-\bz_{k-\frac12}}^2 
+ 
3\sigma_H^2\right]
+
\alpha_k\EE \left\langle 
\Delta_f^k
+
\Delta_h^{k+\frac12},
\bz_{k+\frac12} - \omega_{\bz}
\right\rangle \notag
\\&~\quad
+
\frac{L \alpha_k^2}{2} 
\EE \norm{\bz_{k + \frac12} - \bz_{k}}^2
+
\frac{\alpha_k}{2\eta_k} \EE \left[ 
\norm{\bz_{k} - \omega_{\bz}}^2 - \norm{\bz_{k+1} - \omega_{\bz}}^2 - \norm{\bz_{k+\frac12} - \bz_{k}}^2
\right]
\end{align*}
The inner product term can be decomposed into
\begin{align*}
&\EE \left\langle 
\Delta_f^k
+
\Delta_h^{k+\frac12},
\bz_{k+\frac12} - \omega_{\bz}
\right\rangle
\\&=
\EE \left\langle 
\Delta_h^{k+\frac12}, \bz_{k+\frac12} - \omega_{\bz}
\right\rangle 
+
\EE \left\langle 
\Delta_f^k, \bz_k - \omega_{\bz}
\right\rangle 
+
\EE \left\langle 
\Delta_f^k, \bz_{k+\frac12} - \bz_k
\right\rangle 
=
\EE \left\langle 
\Delta_f^k, \bz_{k+\frac12} - \bz_k
\right\rangle 
,
\end{align*}
Where the expectation of the first two terms all equals $0$. Thus, we obtain
\begin{align*}
&\EE V(\zag_{k+1}, \omega_{\bz})
-
(1 - \alpha_k) \EE V(\zag_{k}, \omega_{\bz})\notag
\\&\leq 
\frac{\alpha_k\eta_k}{2} \left[2L_H^2 \EE\norm{\bz_{k+\frac12}-\bz_{k-\frac12}}^2 
+ 
3\sigma_H^2\right]
+
\alpha_k\EE \left\langle 
\Delta_f^k
,
\bz_{k+\frac12} - \bz_k
\right\rangle \notag
\\&~\quad
+
\frac{\alpha_k}{2\eta_k} \EE \left[ 
\norm{\bz_{k} - \omega_{\bz}}^2 - \norm{\bz_{k+1} - \omega_{\bz}}^2 
\right]
- 
\left(\frac{\alpha_k}{2\eta_k}
-
\frac{L \alpha_k^2}{2} 
\right)
\EE \norm{\bz_{k+\frac12} - \bz_{k}}^2
.
\end{align*}
For any $r > 0$, we pair up
\begin{align*}
-\frac{(1 - r) \alpha_k}{2\eta_k} 
\EE \norm{\bz_{k+\frac12} - \bz_k}^2 
+
\alpha_k 
\EE 
\left\langle 
\Delta_f^k, \bz_{k+\frac12} - \bz_k
\right\rangle
\leq 
\frac{\alpha_k \eta_k}{2(1-r)}\EE \norm{\Delta_f^k}^2 
,
\end{align*}
and thus
\begin{align}
&\EE V(\zag_{k+1}, \omega_{\bz})
-
(1 - \alpha_k) \EE V(\zag_{k}, \omega_{\bz})\notag
\\&\leq 
\frac{\alpha_k\eta_k}{2} \left[2L_H^2 \EE\norm{\bz_{k+\frac12}-\bz_{k-\frac12}}^2 
+ 
3\sigma_H^2\right]
+
\frac{\alpha_k\eta_k}{2(1-r)} \EE \norm{\Delta_f^k}^2 \notag
\\&~\quad
+
\frac{\alpha_k}{2\eta_k} \EE \left[ 
\norm{\bz_{k} - \omega_{\bz}}^2 - \norm{\bz_{k+1} - \omega_{\bz}}^2 
\right]
- 
\left(\frac{r\alpha_k}{2\eta_k}
-
\frac{L \alpha_k^2}{2} 
\right)
\EE \norm{\bz_{k+\frac12} - \bz_{k}}^2
.\label{eq:last_18}
\end{align}

Next, we connect $\norm{\bz_{k+\frac12} - \bz_{k-\frac12}}^2$ with the squared norms $\norm{\bz_{\ell+\frac12} - \bz_{\ell}}^2$.
For $\eta_k$ satisfying $\eta_k L_H \leq \frac{\sqrt{c}}{2}$, we have
\begin{equation}\label{eq:recursion_init_stoc}
\begin{aligned}
&\EE \left\|\zhp - \zhm\right\|^2
\leq 
2 \EE \norm{\zhp - \bz_k}^2
+
2 \EE \norm{\bz_k - \zhm}^2
\\&=
2 \EE \norm{\zhp - \bz_k}^2
+
2 \eta_{k-1}^2 \EE \norm{\tilde{H}(\zhm) - \tilde{H}(\bz_{k-\frac32})}^2
\\&=
2 \EE \norm{\zhp - \bz_k}^2
+
2 \eta_{k-1}^2 \EE \norm{H(\zhm) - H(\bz_{k-\frac32}) + \Delta_h^{k-\frac32}}^2     +
2\eta_{k-1}^2\EE \norm{\Delta_h^{k-\frac12}}^2
\\&\leq 
2 \EE \norm{\zhp - \bz_k}^2
+
4 \eta_{k-1}^2 L_H^2 \EE \norm{\zhm - \bz_{k-\frac32}}^2 +
6 \eta_{k-1}^2 \sigma_H^2
\\&=
2 \sum_{\ell=0}^{k} 
c^{k-\ell} 
\EE \norm{\bz_{\ell+\frac12} - \bz_{\ell}}^2
+
6 \sum_{\ell=0}^k c^{k-\ell}\eta_{\ell-1}^2 \sigma_H^2
.
\end{aligned}\end{equation}
Bringing Eq.~\eqref{eq:recursion_init_stoc} into~\eqref{eq:last_18}, we have
\begin{align*}
&\EE V(\zag_{k+1}, \omega_{\bz})
-
(1 - \alpha_k) \EE V(\zag_{k}, \omega_{\bz})
\\&\leq 
\frac{\alpha_k\eta_k}{2} \left[4L_H^2 
\sum_{\ell=0}^{k} 
c^{k-\ell} 
\EE \norm{\bz_{\ell+\frac12} - \bz_{\ell}}^2
+
12L_H^2 \sum_{\ell=0}^k c^{k-\ell}\eta_{\ell-1}^2 \sigma_H^2
+ 
3\sigma_H^2\right]
+
\frac{\alpha_k\eta_k}{2(1-r)} \sigma_F^2
\notag
\\&~\quad
+
\frac{\alpha_k}{2\eta_k} \EE \left[ 
\norm{\bz_{k} - \omega_{\bz}}^2 - \norm{\bz_{k+1} - \omega_{\bz}}^2 
\right]
- 
\left(\frac{r\alpha_k}{2\eta_k}
-
\frac{L \alpha_k^2}{2} 
\right)
\EE \norm{\bz_{k+\frac12} - \bz_{k}}^2
\\&\leq 
\frac{\alpha_k\eta_k}{2} \left[4L_H^2 
\sum_{\ell=0}^{k} 
c^{k-\ell} 
\EE \norm{\bz_{\ell+\frac12} - \bz_{\ell}}^2
+
3 \frac{c}{1-c} \sigma_H^2  
+ 
3\sigma_H^2\right]
+
\frac{\alpha_k\eta_k}{2(1-r)} \sigma_F^2 \notag
\\&~\quad
+
\frac{\alpha_k}{2\eta_k} \EE \left[ 
\norm{\bz_{k} - \omega_{\bz}}^2 - \norm{\bz_{k+1} - \omega_{\bz}}^2 
\right]
- 
\left(\frac{r\alpha_k}{2\eta_k}
-
\frac{L \alpha_k^2}{2} 
\right)
\EE \norm{\bz_{k+\frac12} - \bz_{k}}^2
\\&\leq 
2\alpha_k\eta_k L_H^2 \sum_{\ell=0}^{k} 
c^{k-\ell} 
\EE \norm{\bz_{\ell+\frac12} - \bz_{\ell}}^2
+
\frac{3\alpha_k\eta_k}{2(1-c)} \sigma_H^2 
+
\frac{\alpha_k\eta_k}{2(1-r)}\sigma_F^2
\\&~\quad
+
\frac{\alpha_k}{2\eta_k} \EE \left[ 
\norm{\bz_{k} - \omega_{\bz}}^2 - \norm{\bz_{k+1} - \omega_{\bz}}^2 
\right]
- 
\left(\frac{r\alpha_k}{2\eta_k}
-
\frac{L \alpha_k^2}{2} 
\right)
\EE \norm{\bz_{k+\frac12} - \bz_{k}}^2
,
\end{align*}
which concludes our proof of Lemma~\ref{lem:ineq_stoc}.
\end{proof}

\subsection{Proof of Lemma~\ref{lem:deter_bound_bilinear}}\label{sec:proof_lem_deter_bound_bilinear}
\begin{proof}[Proof of Lemma~\ref{lem:deter_bound_bilinear}]
The optimal condition of the problem yields 
$H(\zstar) = 0$ for $\zstar$ being the solution of the VI.
By the monotonicity of $H(\cdot)$, let $\zholder = \bz_{k+\frac12}$ and $\zholder' =  \omega_{\bz} $ in~\eqref{eq:H_prop}, we have that 
\begin{align}\label{eq:H_prop_yields_1}
\left\langle         H(\bz_{k+\frac12}) - H(\omega_{\bz}), \bz_{k+\frac12}- \omega_{\bz}
\right\rangle \geq 0
,\quad
\forall \omega_{\bz} \in \cZ
.
\end{align}
Let $\foneholder = \zholder_{k+\frac12}$, $\ftwoholder = \zholder_{k+1}$, $\thetaholder = \zholder_k$, $\boneholder = \eta_k H (\zholder_{k-\frac12})$, $\btwoholder = \eta_k H (\zholder_{k+\frac12})$ and $\zholder = \omega_z$ in Proposition~\ref{prop:PRecursion}, we have
\begin{align}
\eta_k \left\langle    H(\zholder_{k+\frac12}),
\zholder_{k+\frac12} - \omega_z
\right\rangle
&\leq 
\frac{\eta_k^2}{2} \norm{H(\bz_{k+\frac12}) - H(\bz_{k-\frac12})}^2
+
\frac12 \left[ 
\norm{\zholder_k - \omega_z}^2 - \norm{\zholder_{k+1} - \omega_z}^2 - \norm{\zholder_k - \zholder_{k+\frac12}}^2 
\right]\notag
\\&\leq 
\frac{\eta_k^2 L_H^2}{2} \norm{\bz_{k+\frac12} - \bz_{k-\frac12}}^2
+
\frac12 \left[ 
\norm{\zholder_k - \omega_z}^2 - \norm{\zholder_{k+1} - \omega_z}^2 - \norm{\zholder_k - \zholder_{k+\frac12}}^2 
\right]\label{eq:ttt}
,
\end{align}
where the last inequality follows by the $L_H$-Lipschitzness of the $H$ operator.
Combining~\eqref{eq:ttt} with~\eqref{eq:H_prop_yields_1}, we obtain
\begin{align}\label{eq:before}
0
=
\eta_k \left\langle 
H(\omega_{\bz}), \bz_{k+\frac12} - \omega_{\bz}
\right\rangle 
\leq 
\frac{\eta_k^2 L_H^2}{2} \norm{\bz_{k+\frac12} - \bz_{k-\frac12}}^2
+
\frac12 \left[ 
\norm{\zholder_k - \omega_z}^2 - \norm{\zholder_{k+1} - \omega_z}^2 - \norm{\zholder_k - \zholder_{k+\frac12}}^2 
\right]
.
\end{align}
Next, we move on to estimate $\norm{\zholder_{k+\frac12} - \zholder_{k-\frac12}}^2$. As we know that via Young’s and Cauchy-Schwarz’s inequalities and the update rules~\eqref{eq:agog_extra_step1} and~\eqref{eq:agog_extra_step2}, for all $k \geq 1$
\begin{align*}
\norm{\bz_{k + \frac12} - \bz_{k - \frac12}}^2
&\leq 
2\norm{\bz_{k + \frac12} - \bz_k}^2
+
2\norm{\bz_{k} - \bz_{k - \frac12}}^2
\\&\leq 
2 \norm{\bz_{k + \frac12} - \bz_{k}}^2
+
2 \eta_{k-1}^2 L_H^2\norm{\bz_{k - \frac12} - \bz_{k - \frac32}}^2
.
\end{align*}
Multiplying both sides by 2 and moving one term to the right hand gives for all $k\ge 1$
\begin{align*}
\norm{\bz_{k + \frac12} - \bz_{k - \frac12}}^2
&\leq 
4 \norm{\bz_{k + \frac12} - \bz_{k}}^2
+
4 \eta_{k-1}^2 L_H^2\norm{\bz_{k - \frac12} - \bz_{k - \frac32}}^2
-
\norm{\bz_{k + \frac12} - \bz_{k - \frac12}}^2
.
\end{align*}
Bringing this into~\eqref{eq:before} and noting that $\eta_{k-1} \leq \frac{1}{2L_H}$ as well as $\eta_k \leq \frac{1}{2L_H}$, we have
\begin{align*}
&
0
\leq 
\frac{\eta_k^2 L_H^2}{2} \norm{\bz_{k+\frac12} - \bz_{k-\frac12}}^2
+
\frac12 \left[ 
\norm{\zholder_k - \omega_z}^2 - \norm{\zholder_{k+1} - \omega_z}^2 - \norm{\zholder_k - \zholder_{k+\frac12}}^2    
\right]
\\&\leq 
\frac12 \left[ 
\norm{\bz_k - \omega_{\bz}}^2 - \norm{\bz_{k+1} - \omega_{\bz}}^2    
\right] 
+
\frac{\eta_k^2 L_H^2}{2}
\left[ 
\norm{\bz_{k-\frac12} - \bz_{k - \frac32}}^2 - \norm{\bz_{k+\frac12} - \bz_{k - \frac12}}^2    
\right]
-
\left(
\frac12 - 
2 \eta_k^2 L_H^2
\right) \norm{\zholder_k - \zholder_{k+\frac12}}^2
\\&\leq 
\frac12 \left[ 
\norm{\bz_k - \omega_{\bz}}^2 + \frac14 \norm{\bz_{k-\frac12} - \bz_{k - \frac32}}^2 
-
\norm{\bz_{k+1} - \omega_{\bz}}^2 - \frac14 \norm{\bz_{k+\frac12} - \bz_{k - \frac12}}^2
\right]
.
\end{align*}
Rearranging the above inequality and take $\omega_{\bz} = \zstar$ and we conclude that 
\begin{align*}
\norm{\zholder_{k+1} - \zstar}^2
+ 
\frac14 \norm{\zholder_{k+\frac12} - \zholder_{k-\frac12}}^2 
\leq 
\norm{\zholder_k - \zstar}^2 + \frac14\norm{\zholder_{k-\frac12} - \zholder_{k-\frac32}}^2
.
\end{align*}
Telescoping over $k = 0, 1, \ldots, K-1$ and noting that $\zholder_{-\frac12} = \zholder_{-\frac32} = \zholder_0$, we have 
\begin{align*}
\norm{\zholder_K - \zstar}^2 \leq \norm{\zholder_K - \zstar}^2
+
\frac14 \norm{\zholder_{K-\frac12} - \zholder_{K-\frac32}}^2 
\leq 
\norm{\zholder_0 - \zstar}
,
\end{align*}
which concludes our proof of Lemma~\ref{lem:deter_bound_bilinear}.
\end{proof}

\end{document}